\documentclass{article}

\usepackage[utf8]{inputenc}
\usepackage[dvipsnames]{xcolor}

\usepackage{amsmath,amsthm,amssymb,amsfonts}
\usepackage{mathtools}
\usepackage{altvars,mathextras}
\usepackage[algo2e,linesnumbered,ruled]{algorithm2e}
\DontPrintSemicolon

\usepackage{textcomp}
\usepackage{verbatim}
\usepackage{graphicx}
\usepackage[caption=false, font=footnotesize]{subfig}
\usepackage{cite}

\usepackage{enumitem}
\setlist[description]{listparindent=\parindent}

\usepackage[hidelinks]{hyperref}
\usepackage[capitalise]{cleveref}
\crefname{equation}{}{}
\Crefname{equation}{Equation}{Equations}
\theoremstyle{plain}
\newtheorem{theorem}{Theorem}[section]
\newtheorem{lemma}[theorem]{Lemma}
\newtheorem{proposition}[theorem]{Proposition}

\newtheorem{assumption}[theorem]{Assumption}

\theoremstyle{definition}

\theoremstyle{remark}
\newtheorem{remark}{Remark}

\newcommand{\mmt}{\mathsf{p}} % << \mmt-th moment
\newcommand{\quantile}{\mathsf{q}} % << quantile

\newcommand{\qmin}{q_{\mathrm{min}}} % << see DCSBM
\newcommand{\qmax}{q_{\mathrm{max}}} % << see DCSBM
\newcommand{\Mmax}{M_{\mathrm{max}}} % << see DCSBM

\usepackage{xspace}
\newcommand{\normadj}{normalized adjacency\xspace}
\newcommand{\Normadj}{Normalized adjacency\xspace}
\newcommand{\NormAdj}{Normalized Adjacency\xspace}

\newcommand{\ellP}[2][\mmt]{\left\| #2 \right\|_{\mathsf{L}_{#1}}}
\newcommand{\opnorm}[1]{\left\| #1 \right\|_{2}}
\newcommand{\twoinfnorm}[1]{\left\| #1 \right\|_{2,\infty}}
\newcommand{\subG}[1]{\left\| #1 \right\|_{\psi_{2}}}
\newcommand{\subE}[1]{\left\| #1 \right\|_{\psi_{1}}}
\newcommand{\frob}{\operatorname{F}} % << subscript only
\newcommand{\op}{2} % << subscript only
\allowdisplaybreaks

\usepackage{fullpage}
\newcommand{\IEEEPARstart}[2]{#1#2}
\newenvironment{IEEEkeywords}{\noindent\textbf{Keywords:}}{}
\newcommand{\appendices}{\appendix}
\newcommand{\blfootnote}[1]{% for fund line from https://tex.stackexchange.com/a/30726
    \begingroup
    \renewcommand\thefootnote{}\footnote{#1}%
    \addtocounter{footnote}{-1}%
    \endgroup
}

\begin{document}

\title{Network Signflip Parallel Analysis\\for Selecting the Embedding Dimension}

\author{%
    David Hong\thanks{Department of Electrical and Computer Engineering, University of Delaware. Email: hong@udel.edu.}
    \and
    Joshua Cape\thanks{Department of Statistics, University of Wisconsin--Madison. Email: jrcape@wisc.edu.}
}

\maketitle

\blfootnote{%
    Equal contribution by the authors.
    D. Hong was supported in part by NSF Mathematical Sciences Postdoctoral Research Fellowship DMS 2103353.
    J. Cape was supported in part by the National Science Foundation under grant DMS 2413552 and by the University of Wisconsin--Madison, Office of the Vice Chancellor for Research and Graduate Education, with funding from the Wisconsin Alumni Research Foundation.
}
\blfootnote{%
  \textcopyright\, 2026 IEEE.
  Personal use of this material is permitted.
  Permission from IEEE must be obtained for all other uses, in any current or future media,
  including reprinting/republishing this material for advertising or promotional purposes,
  creating new collective works, for resale or redistribution to servers or lists,
  or reuse of any copyrighted component of this work in other works.
}

\begin{abstract}
    This paper investigates the problem of selecting the embedding dimension for large heterogeneous networks that have weakly distinguishable community structure.
    For a broad family of embeddings based on normalized adjacency matrices, we introduce a novel spectral method that compares the eigenvalues of the normalized adjacency matrix to those obtained after randomly signflipping its entries.
    The proposed method, called \emph{network signflip parallel analysis (NetFlipPA)}, is interpretable, simple to implement, data driven, and does not require users to carefully tune parameters.
    For large random graphs arising from degree-corrected stochastic blockmodels with weakly distinguishable community structure (and consequently, non-diverging eigenvalues), NetFlipPA provably recovers the spectral noise floor (i.e., the operator norm of the noise component of the normalized adjacency matrix).
    NetFlipPA thus provides a statistically rigorous randomization-based method for selecting the embedding dimension by keeping the eigenvalues whose magnitudes rise above the recovered spectral noise floor.
    Compared to traditional cutoff-based methods, the data-driven threshold used in NetFlipPA is provably effective under milder assumptions on the node degree heterogeneity and the number of node communities.
    Our main results rely on careful non-asymptotic perturbation analysis and leverage recent progress on local laws for nonhomogeneous Wigner-type random matrices.
\end{abstract}

\begin{IEEEkeywords}
    Spectral methods,
    network analysis,
    random graphs,
    rank estimation.
\end{IEEEkeywords}

% =========================
\section{Introduction}
\label{sec:introduction}

\IEEEPARstart{N}{etworks} and graph-structured data convey pairwise interactions (i.e., edges or links) between entities of interest (i.e., vertices or nodes) and are widespread in the modern world, appearing in domains across science and engineering \cite{newman2018networks}.
Examples include sender-receiver information networks in computer systems \cite{peterson2011computer}, peer-peer friendship and enmity networks in sociology \cite{snijders2011statistical}, country-country trade networks in economics \cite{graham2020econometric}, protein-protein interaction networks in biology \cite{junker2007biological}, and region-region co-activation networks in neuroimaging \cite{fornito2016fundamentals}.
In practice, the observable interactions or edges in these settings are often noisy and possess measurement error.
For example, noisy correlation-based associations may be thresholded to construct binary-valued brain networks in connectomics \cite{bordier2017graph,theis2023threshold}.
Further, node degrees are often heterogeneous, with nodes differing in their overall level of connectivity.
To analyze these noisy heterogeneous networks, a popular approach is to construct a low-dimensional Euclidean representation of the nodes, i.e., an embedding, that captures the salient connectivity structure of the network while reducing the influence of noise and degree heterogeneity.
When there is significant noise, such as when communities present in the network are weakly distinguishable, the choice of dimension for the embedding is crucial.
Selecting too small an embedding dimension risks missing informative connectivity structure in the network, whereas selecting too large an embedding dimension risks retaining uninformative variability due to noise.
This paper tackles the problem of selecting the embedding dimension in these challenging settings.

A common approach to obtaining embeddings is to use the eigenvectors and eigenvalues of a matrix-valued representation of the network that reflect salient graph-theoretic properties encompassing connectivity and low-rank structure \cite{chung1997spectral}.
Numerous matrix-valued representations have been studied in the literature, notably the adjacency matrix \cite{biggs1974algebraic}, graph Laplacians \cite{chung1997spectral}, regularizations and normalizations thereof \cite{qin2013regularized,le2017concentration,sarkar2015normalization}, and the Bethe--Hessian matrix or deformed Laplacian \cite{saade2014spectral}.
The relative strengths and weaknesses of these matrices continue to be studied under various scenarios, both theoretically (e.g., see the references above) and empirically (e.g., see \cite{priebe2019two}), though this is traditionally done on a case-by-case basis.
In this paper, we consider how to select the dimension for embeddings derived from a broad class of matrix-valued representations.

Random graph models, in particular stochastic blockmodels \cite{holland1983stochastic,anderson1992building} and their generalizations, provide a natural environment in which to rigorously study this task due to their ground-truth low-rank generative structure.
Indeed, there has been substantial progress on principled methods for analyzing these models in recent years.
Notably, significant effort has been devoted to estimating the community (i.e., block) memberships of nodes \cite{sussman2012consistent,hajek2016achieving,su2020strong,zhang2024fundamental}, estimating the number of node communities \cite{cerqueira2020estimation,jin2023optimal,hwang2024estimation}, and testing for various node properties \cite{fan2022simple,fan2022simplerc:arxiv:v1,du2023hypothesis}.
However, much of the existing literature has focused on regimes where network community structure is, loosely speaking, moderately or strongly distinguishable.
Moreover, while estimating the number of node communities is closely related to our problem of selecting the embedding dimension, these problems are in fact distinct since the best embedding dimension is not necessarily the same as the number of communities.
We illustrate this distinction in \cref{sec:netflipPA} (see \cref{fig:netflippa}) and elaborate more on the broader relationships to related works in \cref{sec:literature}.

Selecting the embedding dimension for networks with underlying low-dimensional structure is closely related to the classical problems of estimating the number of components in principal component analysis and estimating the number of factors in factor analysis.
A widely used approach for these problems is parallel analysis \cite{horn1965,buja1992rop}, wherein one generates noise-like matrices (e.g., by permuting the entries of the data matrix to ``destroy'' the signal structure) and then selects the principal components that ``rise above'' their noise-like analogues.
Methods in the style of parallel analysis have had a long history of successful use; see for example the discussion in \cite{dobriban2017permutation}.
Moreover, they have received recent attention in the context of heterogeneous noise and weak-signal regimes \cite{hong2020stn:arxiv:v4}, which correspond to weakly distinguishable community structure in the context of networks.
That said, parallel analysis methods have remained relatively unexplored in statistical network analysis.
Indeed, randomization-based selection methods in general have received comparatively less attention in the network literature to date, possibly because it is not always clear how to properly account for the nature of pairwise interactions.
In this paper, we develop a novel network-centric parallel analysis procedure and analyze its performance in the challenging setting of large heterogeneous networks that have weakly distinguishable community structure.

The main contributions of this paper are summarized as follows.

\begin{enumerate}
    \item We propose \emph{network signflip parallel analysis} (NetFlipPA, \cref{alg:netflippa}), a spectral method for selecting the embedding dimension for large heterogeneous networks, namely the number of positive and negative network eigenvalues and corresponding eigenvectors to use for the embedding.
    Our methodology synthesizes and builds on recent, previously separate developments in the study of rank estimation under heteroscedasticity \cite{hong2020stn:arxiv:v4} and the study of embeddings via \normadj matrices \cite{ali2017improved}.
    Our proposed method applies to a broad class of \normadj matrices and is based on comparing their eigenvalues to the operator norms obtained after their entries have undergone multiple trials of symmetric signflipping via multiplication with independent Rademacher random variables.
    Notably, the method is interpretable, simple to implement, and does not require users to carefully tune parameters.

    \item We show that the proposed NetFlipPA provably recovers the spectral noise floor in the challenging setting of large degree-corrected stochastic blockmodel (DCSBM) graphs with weakly distinguishable community structure (\cref{thm:DCSBM:noise:recovery}).
    The spectral noise floor, namely the operator norm of the noise component of the \normadj matrix, is the point below which network eigenvalues are likely to have very noisy corresponding eigenvectors that would introduce uninformative variability into the embedding.
    Thus, by recovering the spectral noise floor, NetFlipPA provides a rigorous method for selecting the embedding dimension.

    \item To prove our main results, we undertake a detailed investigation of the high-dimensional concentration properties of certain random matrices that arise in DCSBM models before and after signflipping.
    Specifically, \cref{thm:DCSBM:approx:L:signflip:decay} establishes a quantitative error bound for a low-rank signal-plus-noise approximation to the \normadj matrix that holds both with and without signflipping.
    \Cref{thm:DCSBM:signal:signflip:destroy} quantifies the decay of the signal component after signflipping, whereas \cref{thm:DCSBM:noise:signflip:preserve} quantifies the preservation of the noise component after signflipping.
    Together, these three theorems finally lead to \cref{thm:DCSBM:noise:recovery} which establishes the performance of NetFlipPA.
    Proving \cref{thm:DCSBM:approx:L:signflip:decay,thm:DCSBM:signal:signflip:destroy,thm:DCSBM:noise:signflip:preserve} involves careful non-asymptotic analysis of the moments of various random matrices (going beyond the existing asymptotic analyses in \cite{ali2017improved,hong2020stn:arxiv:v4})
    and leverages recent progress on local laws for nonhomogeneous Wigner-type random matrices \cite{ajanki2016ufg,ajanki2019qve}.
\end{enumerate}

The remainder of this paper is organized as follows.
\Cref{sec:notation} establishes notation and mathematical preliminaries.
\Cref{sec:netflipPA} introduces \emph{network signflip parallel analysis (NetFlipPA)}, our proposed methodology for selecting the embedding dimension in large heterogeneous networks.
\Cref{sec:numerical} demonstrates the performance of NetFlipPA with a couple of illustrative numerical examples,
and \cref{sec:polblogs} applies it to a real data example.
\Cref{sec:theory} presents our main theoretical results, namely non-asymptotic properties and performance guarantees for NetFlipPA when applied to degree-corrected stochastic blockmodel graphs with weakly distinguishable community structure (and hence with non-diverging eigenvalues).
\Cref{sec:literature} provides in-depth discussion of the existing related literature and context for the theory and methods in the present paper.
\Cref{sec:conclusion} offers concluding discussion and mentions opportunities for future work.
Proofs and supporting lemmas are collected in the appendix.

% =========================
\section{Notation and Mathematical Preliminaries}
\label{sec:notation}

In this paper, scalars are denoted by $x$ or by $x_{i}$ when associated with the index $i$.
Vectors are treated as columns and are denoted entrywise by $\bmx \coloneqq (x_{1}, \dots, x_{n})^{\prime}$, where $\phantom{}^{\prime}$ denotes the transpose operation.
We let $(\bmx_{i})_{j}$ denote the $j$-th entry of the vector $\bmx_{i}$, and we let both $A_{ij}$ and $(\bmA)_{ij}$ denote the $(i,j)$-th entry of the matrix $\bmA$.
Constants are typically denoted by $c$ or $C$ and may change from line to line unless specified otherwise.
Constants that depend on other quantities are adorned with superscripts or subscripts when relevant, for example, when writing $|x_{ij}| \le C_{i}$.
Throughout, $\bbE(\cdot)$ denotes taking an expectation.

Standard little-o and big-O notation are written as $o(\cdot)$ and $O(\cdot)$, respectively.
Likewise, $f(n) \lesssim g(n)$ denotes that there exist $C,N > 0$ so that for all $n \geq N$ it holds that $|f(n)| \leq C g(n)$, while $f(n) \gtrsim g(n)$ denotes that $g(n) \lesssim f(n)$, and $f(n) \asymp g(n)$ denotes that $f(n) \lesssim g(n)$ and $f(n) \gtrsim g(n)$.

We also adopt the following common notation.
The Frobenius norm is given by $\|\bmA\|_{\frob} \coloneqq (\sum_{i,j} A_{ij}^{2})^{1/2}$.
The operator (i.e., spectral) norm is given by $\|\bmA\|_{\op} \coloneqq \sup_{\|\bmx\|_{2} = 1} \|\bmA\bmx\|_{2}$.
The induced $(2,\infty)$ norm (i.e., the maximum Euclidean row norm) is given by
$
\|\bmA\|_{2,\infty}
\coloneqq
\sup_{\|\bmx\|_{2} = 1} \|\bmA\bmx\|_{\infty}
\equiv
\max_{i} ( \sum_{j} A_{ij}^{2} )^{1/2}
.
$
The maximum absolute entry norm is given by $\|\bmA\|_{\max} \coloneqq \max_{i,j} |A_{ij}|$.
The Hadamard (entrywise) product between $\bmA$ and $\bmB$ is given elementwise by $(\bmA \circ \bmB)_{ij} \coloneqq A_{ij} B_{ij}$.
Given a positive integer $\mmt \ge 1$ and random variable $X$ satisfying $\bbE[|X|^{\mmt}] < \infty$, the $\mmt$-th moment norm is given by $\ellP{X} \coloneqq \bbE[|X|^{\mmt}]^{1/\mmt}$.

We adopt the convention that $0^{p} = 0$ for any $p \in \mathbb{R}$.
This convention accommodates the possibility of encountering isolated nodes in the adjacency matrix $\bmA$ in which case the diagonal entries of $\bmD^{-\alpha} \coloneqq \diag(\bmA \bm1)^{-\alpha}$ are not otherwise well defined.

% =========================
\section{Network Signflip Parallel Analysis (NetFlipPA)}
\label{sec:netflipPA}

This section introduces \emph{network signflip parallel analysis (NetFlipPA)}, our proposed methodology for selecting the embedding dimension in large heterogeneous networks.
Specifically, we consider embeddings derived from \normadj matrices of the form
\begin{equation}
    \label{eq:normadj}
    \bmL_{\alpha}
    \coloneqq
    (2m)^{\alpha}\tfrac{1}{\sqrt{n}}\bmD^{-\alpha}\left(\bmA - \tfrac{1}{2m}\bmd\bmd'\right)\bmD^{-\alpha}
    \quad
\end{equation}
where
\begin{equation*}
    \bmd
    \coloneqq
    \bmA \bm1,
    \quad
    \bmD
    \coloneqq
    \diag(\bmd),
    \quad
    m
    \coloneqq
    \tfrac{1}{2}\bmd'\bm1
    .
\end{equation*}
This family is defined by the single hyper-parameter $\alpha \in \bbR$, which controls the amount of degree-correction.
As highlighted in \cite{ali2017improved}, this setting encompasses many popularly used matrix-valued network representations.
Concretely, $\bmL_{0}$ corresponds to the modularity matrix \cite{newman2006modularity,jin2015fast}, $\bmL_{1/2}$ corresponds to a modularity-type centered symmetric normalized Laplacian matrix \cite{qin2013regularized,chung1997spectral}, and $\bmL_{1}$ corresponds to a graph Laplacian with aggressive degree normalization \cite{coja-oghlan2009finding,gulikers2017spectral}.

NetFlipPA is a randomization-based method for selecting the embedding dimension from the \normadj matrix $\bmL_{\alpha}$.
The method is inspired by parallel analysis approaches to rank selection in principal component analysis and can be viewed as a network-centric version of parallel analysis.
NetFlipPA works by randomly signflipping the entries of the \normadj matrix to generate noise-like counterparts that can be used to estimate the spectral noise floor.
The embedding dimension is then chosen by selecting all the network eigenvalues whose magnitudes rise above the estimated spectral noise floor.

\begin{figure*}[!htbp]
    \centering
    \includegraphics[scale=0.6]{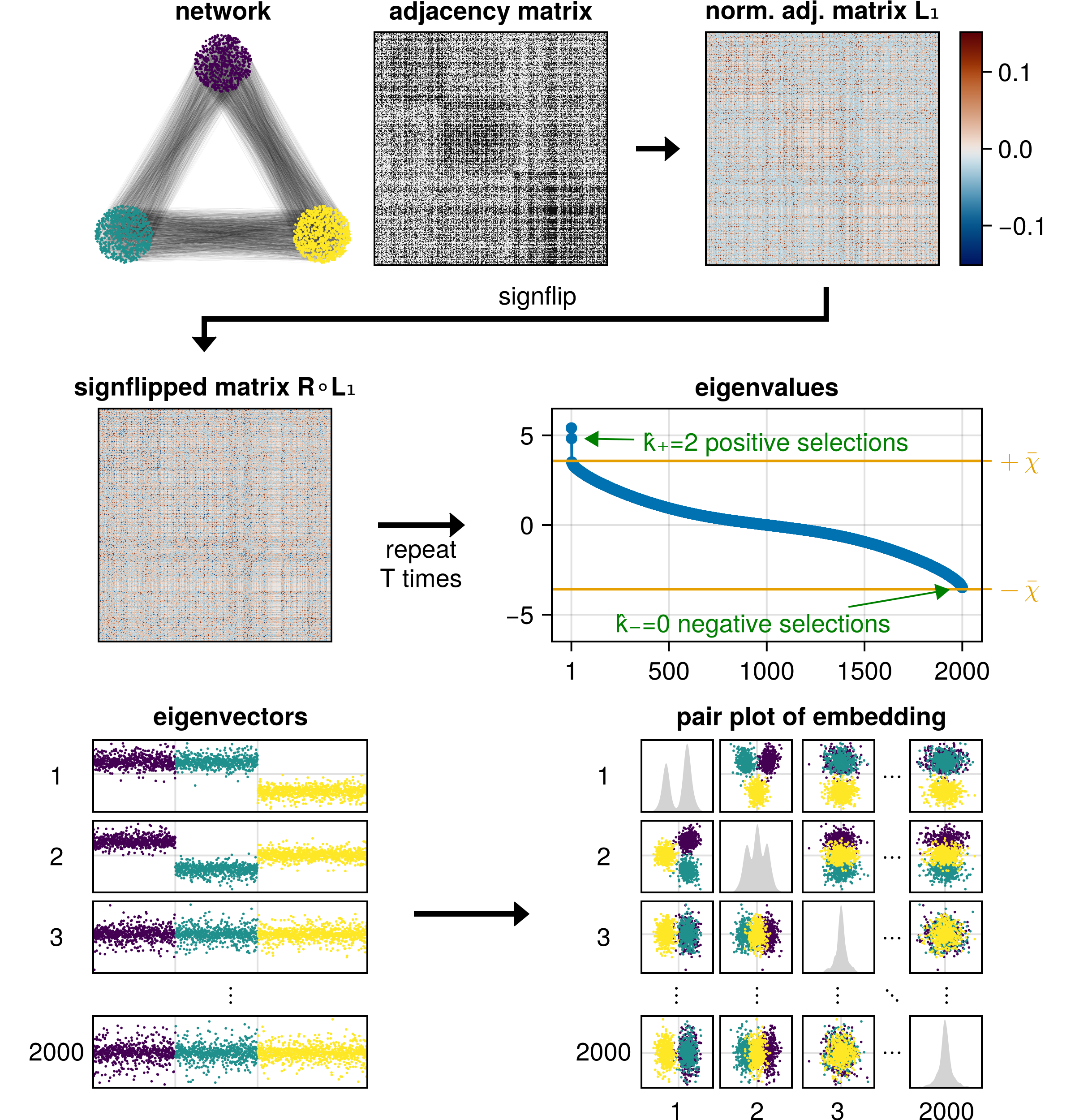}
    \caption{Illustration of network signflip parallel analysis (NetFlipPA). The positive and negative cutoffs $\pm \bar{\chi}$ here were computed from $T = 100$ trials with a quantile of $\quantile = 1$.}
    \label{fig:netflippa}
\end{figure*}

\Cref{fig:netflippa} illustrates the idea of NetFlipPA for a network with three communities colored purple, green, and yellow, respectively.
Nodes belonging to the same community are generally more connected than nodes belonging to different communities, as can be seen by the assortative block structure in the adjacency matrix and in the \normadj matrix $\bmL_{1}$.
Randomly signflipping the entries of the \normadj matrix produces a signflipped matrix $\bmR \circ \bmL_{1}$ with similar characteristics to the original \normadj matrix but without the block structure; it is a noise-like counterpart.
Repeating this procedure $T$ times and collecting the operator norms of these signflipped matrices yields an empirical estimate of the distribution for the spectral noise floor that can then be compared with the eigenvalues of the original \normadj matrix to select the embedding dimension.
Specifically, we select the positive and negative eigenvalues of the \normadj matrix whose magnitudes rise above the estimated spectral noise floor, given by a user-specified quantile of the empirical distribution for the signflipped operator norms.
For this example, we obtain an embedding dimension of two positive eigenvalues and zero negative eigenvalues.
Examining the extreme eigenvectors (with values colored according to the communities) reveals that indeed only the first two eigenvectors reflect the community structure; the third eigenvector is uninformative, as is the final eigenvector.
The corresponding pair plots show that an embedding with just the first two eigenvectors captures the community structure, whereas the other eigenvectors merely add noise.
NetFlipPA correctly estimates the spectral noise floor here and consequently correctly selects an embedding dimension of two positive eigenvalues and zero negative eigenvalues.%
\footnote{%
    In this example, observe that the correct embedding dimension is not equal to the number of communities.
    This is a subtle but important distinction.
    The correct embedding dimension reflects how many eigenvectors (corresponding to both positive and negative eigenvalues) are anticipated to reflect the community structure, which may or may not equal the number of communities in the network.}
\Cref{alg:netflippa} provides a complete description of the NetFlipPA procedure.

\begin{algorithm2e}[h]
    \caption{Network Signflip Parallel Analysis (NetFlipPA)}
    \label{alg:netflippa}
    \SetKwInOut{Input}{Input}
    \SetKwInOut{Output}{Output}

    \Input{\Normadj matrix $\bmL \equiv \bmL_{\alpha} \in \bbR^{n \times n}$,
        quantile $\quantile \in [0,1]$,
        number of trials $T$.}
    \Output{Selected embedding dimension $\htk_+$ and $\htk_-$ for positive and negative eigenvalues, respectively.}
    \For{$t = 1$ \KwTo $T$}{
        Randomly signflip the entries of $\bmL$
        with symmetric signflips:
        \begin{equation*}
            \htL^{(t)}_{ij} = \htL^{(t)}_{ji}
            \overset{\text{ind}}{\sim}
            \begin{cases}
                +L_{ij} , & \text{with probability $1/2$} , \\
                -L_{ij} , & \text{with probability $1/2$} , \\
            \end{cases}
        \end{equation*}
        i.e., $\bhtL^{(t)} = \bmR^{(t)} \circ \bmL$
        where $\bmR^{(t)} \in \{-1,+1\}^{n \times n}$
        is a Wigner matrix with Rademacher entries\;
        \label{alg:netflippa:signflip}

        $\chi^{(t)} \gets \|\bhtL^{(t)}\|_{\op}$, i.e., compute the operator norm of the signflipped matrix\;
    }
    $\brchi \gets \text{$\quantile$-quantile of } \big(\chi^{(1)},\dots,\chi^{(T)}\big)$\;
    $\lambda_1,\dots,\lambda_n \gets \text{eigenvalues of $\bmL$}$\;
    $\htk_+ \gets \#\{k : \lambda_k > +\brchi\}$
    and
    $\htk_- \gets \#\{k : \lambda_k < -\brchi\}$,
    i.e., select the number of eigenvalues exceeding $\brchi$ in magnitude\;
    \label{alg:netflippa:select}
\end{algorithm2e}

\Cref{alg:netflippa:signflip} in \cref{alg:netflippa} uses symmetric signflips, i.e., $\bmR^{(t)}$ is a symmetric matrix, in contrast to the nonsymmetric signflips used by FlipPA in \cite{hong2020stn:arxiv:v4}.
This preserves the symmetry of $\bmL$ in the signflipped analogue $\bhtL^{(t)}$.
Moreover, here we use eigenvalues rather than singular values and return both positive and negative selections since they are the relevant quantities for our network setting; positive and negative eigenvalues correspond to assortative and disassortative community structure, respectively.
In \cref{alg:netflippa:select}, the selection is obtained by comparing the eigenvalues of the original \normadj matrix with the operator norms of the signflipped analogues.
This is sometimes called an ``upper-edge comparison'' in the parallel analysis literature for PCA since the operator norm corresponds to the largest singular value, i.e., the upper edge of the corresponding spectral distribution.
One could also use the so-called ``pairwise comparison'' approach where each eigenvalue is successively compared against its signflipped analogue; see \cite[Section~2]{hong2020stn:arxiv:v4} for further discussion of these comparison methods.
Similarly, one could also add an optional threshold for how much the eigenvalues must rise above the spectral noise floor.
Doing so can help reduce the risk of over-selection and was needed for the theoretical analysis in \cite{hong2020stn:arxiv:v4}, but this minor modification is often not necessary in practice so we omit it here to simplify the presentation.

The choice of $T$ affects the stochasticity of the output of NetFlipPA; large choices of $T$ help make the algorithm behave more deterministically from run to run, which can be especially helpful for small networks.
The choice of $\quantile$ tunes the trade-off between over- and under-selection, and allows users to incorporate domain knowledge about the expected strength of structures in the network.
\Cref{sec:qT:experiment} illustrates the influence of these two parameters on the performance of NetFlipPA.
Notably, they typically do not have a major impact for large networks due to high-dimensional concentration phenomena.
Overall, in the absence of other external considerations, we recommend choosing a number of trials $T$ as large as can be computationally accommodated and choosing $\quantile$ between $0.95$ and $1$.

% =========================
\section{Simulation Examples}
\label{sec:numerical}

To briefly illustrate NetFlipPA and highlight its performance, this section simulates a degree-corrected stochastic blockmodel (DCSBM) inspired by \cite[Figure~2]{ali2017improved}.
Specifically, we consider a network with $n=2000$ nodes and $K=3$ communities whose adjacency matrix $\bmA \in \{0,1\}^{n \times n}$ is generated as
\begin{equation*}
    A_{ij}
    = A_{ji}
    \overset{\textnormal{ind}}{\sim}
    \operatorname{Bernoulli}(q_{i} q_{j} C_{g_{i} g_{j}}),
    \qquad
    1 \leq i \leq j \leq n
    ,
\end{equation*}
where
$C_{g_{i} g_{j}} = 1 + \frac{M_{g_{i} g_{j}}}{\sqrt{n}}$ with $M_{g_{i} g_{j}} = -4$ for $g_{i} \neq g_{j}$ and $M_{g_{i} g_{i}} = 10$.
The vector of node community memberships is set to be $\bmg = [1 \cdot \bm1_{600}, 2 \cdot \bm1_{600}, 3 \cdot \bm1_{800}]$.
The entries of $\bmq \in (0,1)^{n}$ are node-specific degree parameters that induce degree heterogeneity.
Throughout, we use $T = 100$ trials and a quantile of $\quantile = 1$ in NetFlipPA.

\begin{figure*}[!htbp]
    \centering
    \includegraphics[scale=0.6]{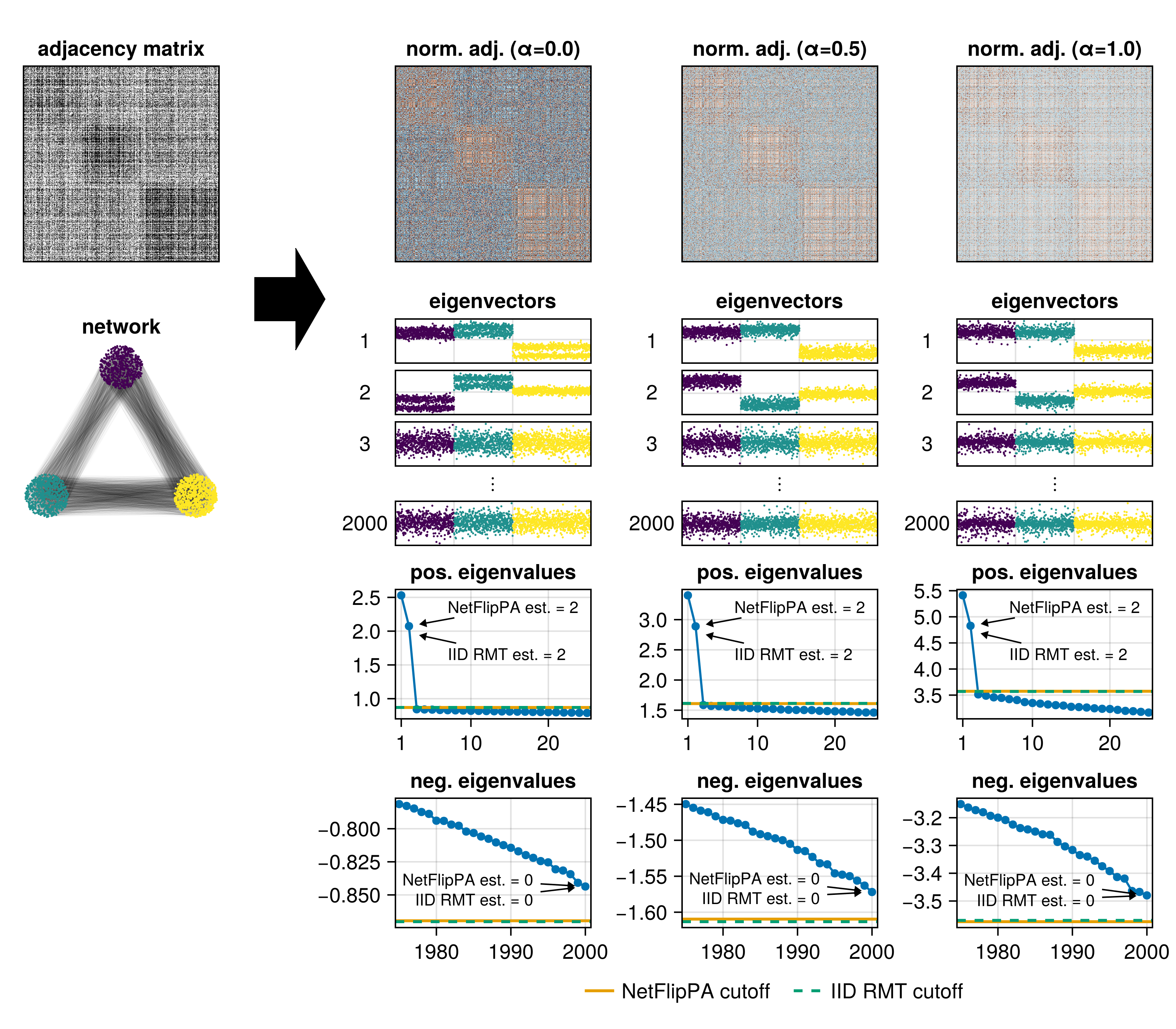}
    \caption{Embedding dimension selected by NetFlipPA versus the RMT-based cutoff from \cite[Theorem 4]{ali2017improved} where the node-specific degree parameters are i.i.d. and so satisfy the assumptions of \cite[Theorem 4]{ali2017improved}. Both methods correctly select an embedding dimension of two positive eigenvalues and zero negative eigenvalues for all choices of $\alpha \in \{0,0.5,1\}$.}
    \label{fig:numerical:1}
\end{figure*}

We first consider a setting with i.i.d. node-specific degree parameters, where $q_{i} = 0.4$ with probability $1/2$ and $q_{i} = 0.9$ otherwise for $1 \le i \le n$.
In this setting, the assumptions of \cite{ali2017improved} are satisfied and so \cite[Theorem 4]{ali2017improved} provides a method for computing a random matrix theory (RMT) estimate of the spectral noise floor that can be used as a cutoff for dimension selection.
\Cref{fig:numerical:1} shows the \normadj matrix for $\alpha \in \{0,0.5,1\}$, the corresponding eigenvectors, and the selections made by NetFlipPA and the RMT-based cutoff.
For each choice of $\alpha$, the correct embedding dimension is two positive eigenvalues and zero negative eigenvalues; only the first two eigenvectors of the \normadj matrix capture the community structure while the other eigenvectors appear uninformative.
In each case, NetFlipPA produces an estimate of the spectral noise floor for $\bmL_{\alpha}$ that closely matches the RMT estimate.
Consequently, both NetFlipPA and the RMT-based cutoff correctly select an embedding dimension of two positive eigenvalues and zero negative eigenvalues.
Indeed, as we will show in \cref{sec:theory}, NetFlipPA provably recovers the spectral noise floor asymptotically in this model and for a broad class of DCSBMs more generally.

\begin{figure*}[!htbp]
    \centering
    \includegraphics[scale=0.6]{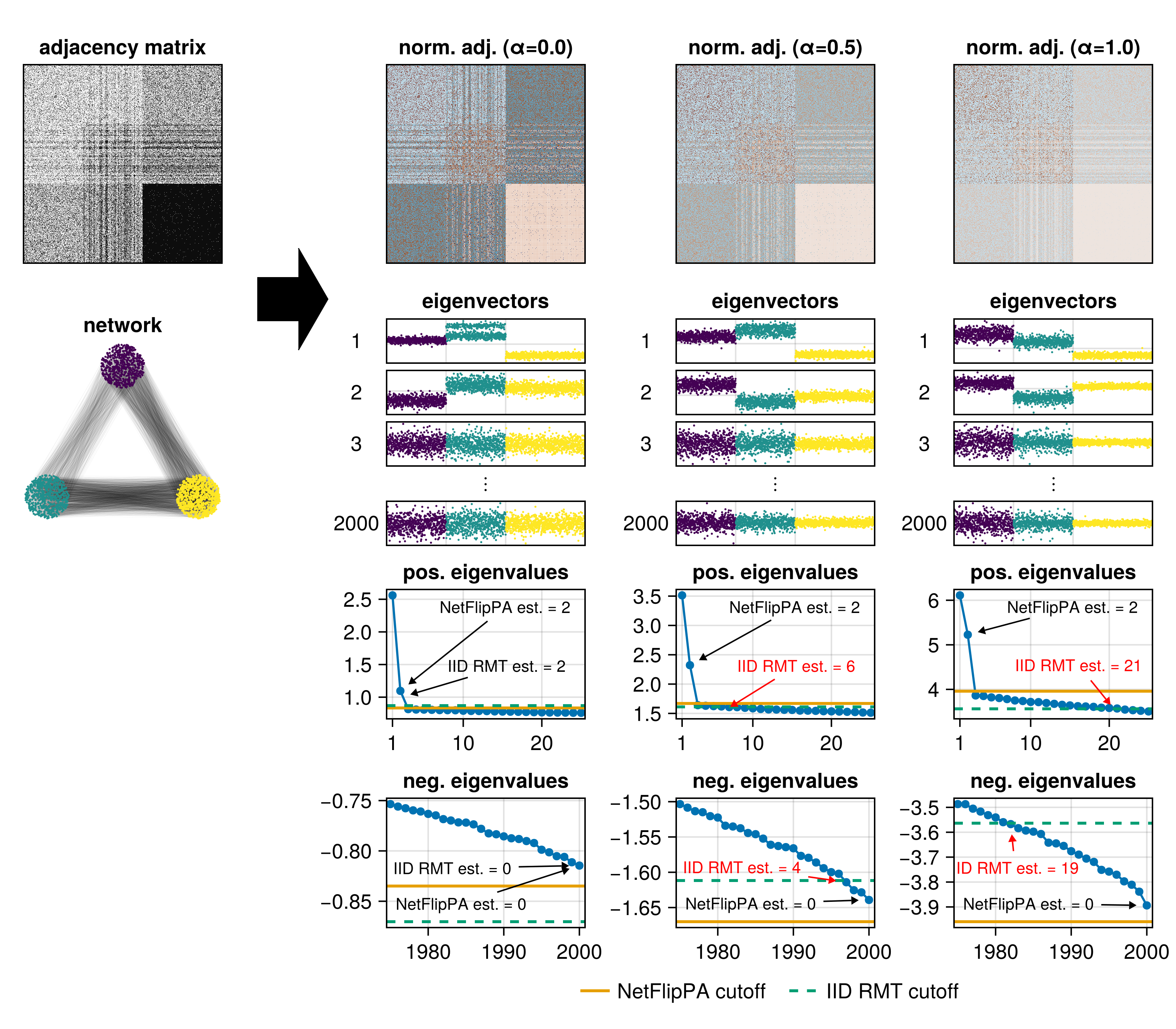}
    \caption{Embedding dimension selected by NetFlipPA versus the RMT-based cutoff from \cite[Theorem 4]{ali2017improved} where the node-specific degree parameters are not i.i.d. and so no longer satisfy the assumptions of \cite[Theorem 4]{ali2017improved}. In this case, only NetFlipPA correctly selects an embedding dimension of two positive eigenvalues and zero negative eigenvalues for all choices of $\alpha \in \{0,0.5,1\}$. The RMT-based cutoff over-selects both positive and negative eigenvalues for $\alpha \in \{0.5, 1\}$.}
    \label{fig:numerical:2}
\end{figure*}

Next, we consider a setting where the node-specific degree parameters are no longer i.i.d. but instead vary from community to community.
Such behavior can naturally arise when different communities have differing levels of interaction.
In particular, we consider the same network model as in \cref{fig:numerical:1} but where the node-specific degree parameters are now generated as
\begin{equation*}
    q_i \mid g_i = 1 \overset{\textnormal{ind}}{\sim}
    \begin{cases}
        0.4 & \text{w.p. } 1.0 \\
        0.9 & \text{w.p. } 0.0
    \end{cases}
    ,
    \qquad
    q_i \mid g_i = 2 \overset{\textnormal{ind}}{\sim}
    \begin{cases}
        0.4 & \text{w.p. } 2/3 \\
        0.9 & \text{w.p. } 1/3
    \end{cases}
    ,
    \qquad
    q_i \mid g_i = 3 \overset{\textnormal{ind}}{\sim}
    \begin{cases}
        0.4 & \text{w.p. } 0.0 \\
        0.9 & \text{w.p. } 1.0
    \end{cases}
    .
\end{equation*}
These distributions are chosen so that the expected number of nodes with $q_{i} = 0.4$ and with $q_{i} = 0.9$ match those of the simulation in \cref{fig:numerical:1}.
\Cref{fig:numerical:2} shows the corresponding \normadj matrix, eigenvectors, and embedding dimension selections for this setting.
As before, the correct embedding dimension here is two positive eigenvalues and zero negative eigenvalues for all three choices of $\alpha$.
However, in this setting, only NetFlipPA produces an accurate estimate of the spectral noise floor and correctly selects the embedding dimension.
The RMT-based cutoff (derived for i.i.d. node-specific degree parameters) under-estimates the spectral noise floor and consequently over-selects both positive and negative eigenvalues for $\alpha = 0.5$ and $\alpha = 1$.
Indeed, this setting does not satisfy \cite[Assumption 1]{ali2017improved}.
In contrast, our theoretical analysis allows for non-i.i.d. node-specific degree parameters (see \cref{assump:model:DCSBM}), and thus NetFlipPA provably recovers the spectral noise floor asymptotically in this setting, too.

% =========================
\section{Real Data Example}
\label{sec:polblogs}

We consider a real-world network of political blogs from \cite{adamic2005political}, in which nodes represent blogs about U.S. politics, and edges represent web links between them from a single day in 2005.
In this dataset, each blog has an associated label indicating its political leaning, namely liberal or conservative, interpretable as two-block ground truth structure.
This dataset is widely encountered in the network analysis literature, often in the study of DCSBMs beginning with \cite{karrer2011stochastic}.
Following convention, we symmetrize the original (directed) network and keep the largest connected component of the original network, yielding an undirected (sub)graph with $1222$ nodes with degree distribution having 0.0, 0.25, 0.5, 0.75, and 1.0 quantiles of 1, 3, 13, 36, and 351, respectively.

The prior work \cite{ali2017improved} investigates normalized adjacency matrix representations for this network dataset, finding that the choice $\alpha = 0$ is preferred on the basis of ground-truth overlap and modularity metrics.
With this in mind, we apply NetFlipPA to $\bmL_{\alpha}$ with $\alpha = 0$, using $\quantile=1$ and $T=100$.
Four dimensions are selected, corresponding to two positive and two negative eigenvalues.
\Cref{fig:polblogs} shows the pair plot for the corresponding degree-scaled eigenvectors, namely the eigenvectors of $\bmL_{0}$ pre-multiplied by $\bmD^{-1}$ as suggested in \cite{ali2017improved}.
The (degree-scaled) eigenvectors are evidently informative of political leaning, where, as described in the figure caption, liberal is shown in blue and conservative is shown in red.
In more detail, eigenvector~1 (whose associated eigenvalue is far above the NetFlipPA cutoff) visibly separates the blogs by political leaning, while eigenvector~2 (whose corresponding eigenvalue is close to the NetFlipPA cutoff) appears to be less informative.
On the other hand, eigenvectors~1221 and 1222 corresponding to the selected negative eigenvalues (often suggestive of disassortative network structure) interestingly exhibit two radial streaks corresponding to political leaning.
Overall, NetFlipPA provides a reasonable selection for both the positive and negative eigenvalues of this dataset.

\begin{figure*}[!htbp]
    \centering
    \includegraphics[scale=0.6]{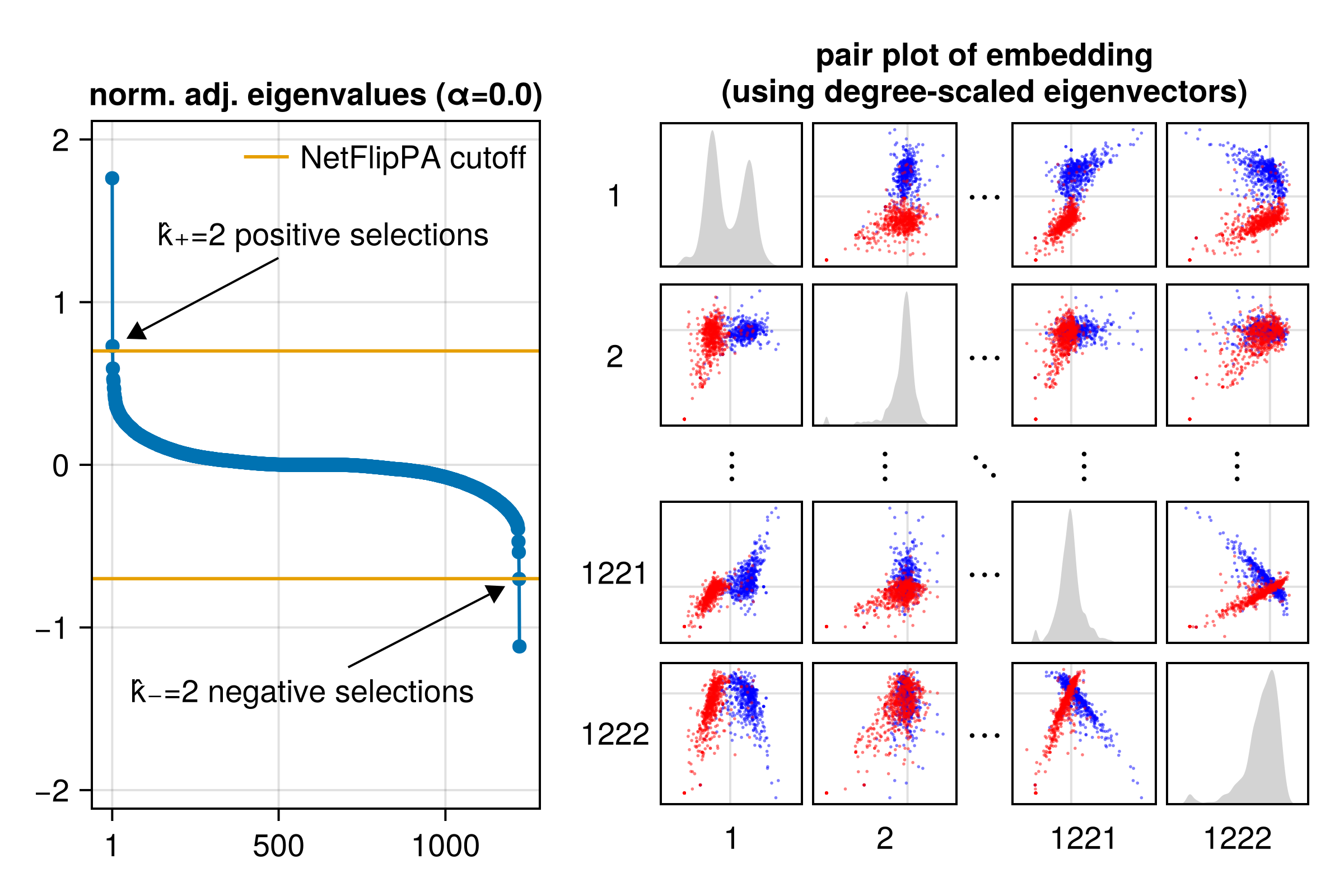}
    \caption{Embedding dimension selection for political blogs network from \cite{adamic2005political} with corresponding pair plot of the embedding vectors colored by political leaning, where blue indicates liberal and red indicates conservative.}
    \label{fig:polblogs}
\end{figure*}

% =========================
\section{Theoretical Guarantees}
\label{sec:theory}

This section analyzes the performance of NetFlipPA (\cref{alg:netflippa}) for embeddings derived from the family of \normadj matrices given in \cref{eq:normadj}.
The overarching objective is to show that the signflipped \normadj matrix, $\bmR \circ \bmL_{\alpha}$, recovers the operator norm of the noise in $\bmL_{\alpha}$ for large random graphs.
As discussed and illustrated in \cref{sec:netflipPA,sec:numerical}, the network eigenvalues below this spectral noise floor are likely to have very noisy corresponding eigenvectors that contain limited or no information about the node community structure.

\Cref{sec:theory:model} describes the degree-corrected stochastic blockmodel random graphs we study.
\Cref{sec:theory:signalplusnoise} establishes a signal-plus-noise approximation of the \normadj matrix under this model.
\Cref{sec:theory:signflip:influence} then analyzes the influence of signflipping on the resulting signal and noise parts.
Finally, \cref{sec:theory:netflippa:performance} provides the main result, namely recovery of the spectral noise floor by NetFlipPA.

\subsection{Degree-corrected Stochastic Blockmodels}
\label{sec:theory:model}

We study degree-corrected stochastic blockmodels (DCSBMs) \cite{coja-oghlan2009finding,karrer2011stochastic} with weakly distinguishable community structure.
In this setting, the extreme network eigenvalues do not diverge, which makes estimation of the embedding dimension particularly challenging.
Specifically, under the DCSBM model we consider, unweighted undirected graphs on $n$ nodes have symmetric adjacency matrices $\bmA \in \{0,1\}^{n \times n}$ given by
\begin{equation}
    \label{eq:model:DCSBM}
    A_{ij}
    =
    A_{ji}
    \overset{\textnormal{ind}}{\sim}
    \operatorname{Bernoulli}(q_{i} q_{j} C_{g_{i} g_{j}}),
    \qquad
    1 \leq i \leq j \leq n
    ,
\end{equation}
where the entries of $\bmq \in (0,1)^n$ are node-specific degree parameters, the entries of $\bmg \in \{1,\dots,K\}^n$ define community memberships, $K$ is the number of communities, and $\bmC \in \bbR^{K \times K}$ is a matrix of community weights that defines the base inter- and intra-community edge probabilities.
The expected adjacency matrix here takes the form
\begin{equation}
    \label{eq:model:DCSBM:expectedA}
    \bbE [\bmA]
    =
    \bmD_{q}
    \bmJ\bmC\bmJ'
    \bmD_{q}
    ,
\end{equation}
where $\bmD_{q} \coloneqq \diag(\bmq)$, and $\bmJ \in \{0,1\}^{n \times K}$ denotes the community membership matrix corresponding to $\bmg$, i.e., $J_{ik} = 1$ if $g_{i} = k$ and $J_{ik} = 0$ otherwise.

The regime considered in this paper is given by the following assumption.
We emphasize that the assumption is weaker and hence more widely applicable than \cite[Assumption~1]{ali2017improved}.

\begin{assumption}
    \label{assump:model:DCSBM}
    As $n \to \infty$, for all $i,j \in \{1,\dots,n\}$,
    \begin{enumerate}
        \item (Weakly distinguishable community structure)
        $C_{g_{i} g_{j}} = 1 + \frac{M_{g_{i} g_{j}}}{\sqrt{n}}$, with $|M_{g_{i} g_{j}}| \leq \Mmax$,

        \item (Mild degree heterogeneity)
        $q_{i} \in [\qmin, \qmax] \subset (0,1)$,
    \end{enumerate}
    where $\Mmax$, $\qmin$, and $\qmax$ are universal positive constants not depending on $n$ or $K$.
    Here, $K$ is allowed to depend on $n$.
\end{assumption}

\begin{remark}[Consequences of Weakly Distinguishable Community Structure]
    The community structure condition in \cref{assump:model:DCSBM}
    implies that $\lim_{n \to \infty} \bbE[A_{ij}] = q_{i} q_{j}$ for all indices $i$ and $j$.
    Moreover, $\bbE[\bmA]$ converges to $\bmq\bmq'$ uniformly, since $\| \bbE[\bmA] - \bmq\bmq' \|_{\max} / \| \bmq\bmq' \|_{\max} \lesssim n^{-1/2}$.
    In other words, the population-level or expected community structure vanishes asymptotically, so the node communities are only weakly present when $n$ is large.
\end{remark}

\begin{remark}[Consequences of Mild Degree Heterogeneity]
    The mild degree heterogeneity condition in \cref{assump:model:DCSBM}
    implies that the node-specific degree parameters $\{q_{i}\}_{i=1}^{n}$ are of the same, constant order in magnitude and do not decay to zero as $n \to \infty$.
    As such, this condition permits mild degree heterogeneity but precludes severe degree heterogeneity.
\end{remark}

\Cref{assump:model:DCSBM} produces dense graphs, namely graphs whose expected number of edges are order $n^{2}$.
While perfect recovery of community structure from dense graphs is asymptotically possible for many previously studied models, that is not the case for the weakly distinguishable community structure in this regime.
Here, the normalized adjacency matrix exhibits non-diverging spiked eigenvalues, which generally precludes asymptotic perfect node clustering and leads to mathematical intricacies for establishing theoretical performance guarantees.
For additional related discussion, see \cite{ali2017improved,kadavankandy2019agf}.

\subsection{Signal-Plus-Noise Approximation of the \NormAdj Matrix}
\label{sec:theory:signalplusnoise}

We begin our analysis by showing that $\bmL_{\alpha}$ is well-approximated by a matrix $\btlL_{\alpha}$, which is a low-rank random signal matrix plus a random noise matrix, given as
\begin{equation}
    \label{eq:DCSBM:signalplusnoise}
    \btlL_{\alpha}
    \coloneqq
    \underbrace{
        \bmU\bmLambda\bmU'
    }_{\textnormal{signal } \btlS_{\alpha}}
    +
    \underbrace{
        \tfrac{1}{\sqrt{n}}\bmD_{q}^{-\alpha}\bmX\bmD_{q}^{-\alpha}
    }_{\textnormal{noise } \btlN_{\alpha}}
    ,
\end{equation}
where
\begin{align}
    \label{eq:DCSBM:signalparts}
    \bmU
    &\coloneqq
    \begin{bmatrix}
        \frac{\bmD_{q}^{1-\alpha} \bmJ}{\sqrt{n}}
        &
        \frac{\bmD_{q}^{-\alpha} \bmX \bm1_n}{\bmq' \bm1_n}
    \end{bmatrix}
    , &
    \bmLambda
    &\coloneqq
    \begin{bmatrix}
        (\bmI_{K} - \bm1_{K} \bbvc') \bmM (\bmI_{K} - \bbvc \bm1_{K}') & -\bm1_{K} \\
        -\bm1_{K}' & 0
    \end{bmatrix}
    , &
    \bbvc
    &\coloneqq
    \frac{\bmJ'\bmq}{\bm1_{n}'\bmq}
    ,
\end{align}
defines the signal matrix,
and
\begin{equation}
    \label{eq:DCSBM:noiseparts}
    \bmX
    \coloneqq
    \bmA
    -
    \bbE [\bmA]
    =
    \bmA
    -
    \left(
        \bmq\bmq'
        +
        \frac{1}{\sqrt{n}}
        \bmB
    \right)
\end{equation}
is a symmetric mean-zero random matrix with independent entries $X_{ij}$ for $i \geq j$, where $\bmB \coloneqq \bmD_{q} \bmJ \bmM \bmJ' \bmD_{q}$.

The following \lcnamecref{thm:DCSBM:approx:L:signflip:decay} formalizes this statement and is a refined, quantitative version of \cite[Theorem~2]{ali2017improved} that here also incorporates signflipping.
The proof, provided in \cref{proof:DCSBM:approx:L:signflip:decay}, relies on detailed matrix perturbation analysis that carefully controls the statistical moments of the non-asymptotic errors.
In words, \cref{thm:DCSBM:approx:L:signflip:decay} bounds the $\ellP{\cdot}$ norm of the operator norm $\opnorm{\cdot}$ of the difference between the \normadj matrix and its signal-plus-noise approximation.

\begin{theorem}[Signal-Plus-Noise Approximation of the \NormAdj Matrix]
    \label{thm:DCSBM:approx:L:signflip:decay}
    For any $\alpha \in \bbR$ and any $\mmt \in \bbN_{+}$,
    it holds that
    \begin{equation}
        \label{eq:DCSBM:approx:L:decay}
        \ellP{\opnorm{ \bmL_{\alpha} - \btlL_{\alpha} }}
        \lesssim
        \left(\frac{\log n}{n}\right)^{1/2}
        ,
    \end{equation}
    where $\btlL_{\alpha}$ is defined in \cref{eq:DCSBM:signalplusnoise,eq:DCSBM:signalparts,eq:DCSBM:noiseparts}.
    Moreover,
    for any $\alpha \in \bbR$ and any $\mmt \in \bbN_+$,
    it holds that
    \begin{equation}
        \label{eq:DCSBM:approx:L:signflip:decay}
        \sup_{\bmR \in \{-1,1\}^{n \times n} : \bmR = \bmR'}
        \;\;
        \ellP{\opnorm{ \bmR \circ \bmL_{\alpha} - \bmR \circ \btlL_{\alpha} }}
        \lesssim
        \left(\frac{\log n}{n}\right)^{1/2}
        .
    \end{equation}
\end{theorem}

\subsection{Influence of Signflipping on the Signal and Noise Matrices}
\label{sec:theory:signflip:influence}

The approximation established in \cref{thm:DCSBM:approx:L:signflip:decay} enables us to study the behavior of NetFlipPA by studying the influence of signflipping on the signal $\btlS_{\alpha}$ and the noise $\btlN_{\alpha}$ separately.
In particular, we will show that signflipping destroys the signal while preserving the operator norm of the noise.

We begin with \cref{thm:DCSBM:signal:signflip:destroy} which establishes two key properties of the signal $\btlS_{\alpha}$.
The first property is a characterization of the $\ell_{2,\infty}$ norm of $\btlS_{\alpha}$ that controls the extent to which the signal is (de)localized on individual rows.
The second property is a characterization of the operator norm of the signflipped signal $\bmR \circ \btlS_{\alpha}$ that quantifies the strength of the signal after signflipping.
Notably, \cref{thm:DCSBM:signal:signflip:destroy} bounds the rate of decay of the sub-Gaussian norm $\psi_{2}$ and the sub-Exponential norm $\psi_{1}$ of these quantities, which directly implies bounds on the statistical moments.
The proof, provided in \cref{proof:DCSBM:2_infty}, involves a careful analysis of these norms that exploits the special structure of $\btlS_{\alpha}$.

\begin{theorem}[Signal Delocalization and Signflipped Signal Destruction]
    \label{thm:DCSBM:2_infty}
    \label{thm:DCSBM:signal:signflip:destroy}
    For any $\alpha \in \bbR$ and for $\btlS_{\alpha}$ defined in \cref{eq:DCSBM:signalplusnoise,eq:DCSBM:signalparts}, it holds that
    \begin{equation}
        \subG{\twoinfnorm{ \btlS_{\alpha} }}
        \lesssim
        \left(\frac{\log n}{n}\right)^{1/2}
        ,
        \qquad
        \text{and}
        \qquad
        \subE{\opnorm{ \bmR \circ \btlS_{\alpha} }}
        \lesssim
        \left(\frac{\log n}{n}\right)^{1/2}
        ,
    \end{equation}
    where $\bmR \in \{-1,1\}^{n \times n}$
    is drawn uniformly at random from the set of symmetric signflip matrices.
\end{theorem}

Next, we proceed with \cref{thm:DCSBM:noise:signflip:preserve} below which establishes that the spectral noise floor, i.e., the operator norm of $\btlN_{\alpha}$, is preserved by random symmetric signflipping.
This preservation is nontrivial since $\ellP{\opnorm{\btlN_{\alpha}}} \gtrsim 1$.
To establish a precise operator norm comparison, our proof, provided in \cref{proof:DCSBM:noise:signflip:preserve}, leverages recent progress in random matrix theory on establishing local laws for general (nonhomogeneous) Wigner-type matrices \cite{ajanki2016ufg,ajanki2019qve}.

\begin{theorem}[Signflipped Noise Floor Preservation]
    \label{thm:DCSBM:noise:signflip:preserve}
    For any $\alpha \in \bbR$, any $\mmt \in \bbN_{+}$, and for $\btlN_{\alpha}$ defined in \cref{eq:DCSBM:signalplusnoise,eq:DCSBM:noiseparts}, it holds that
    \begin{equation}
        \label{eq:signflipped:noise}
        \ellP{\opnorm{\bmR \circ \btlN_{\alpha}} - \opnorm{\btlN_{\alpha}}}
        \lesssim
        n^{-1/2}
        ,
    \end{equation}
    where $\bmR \in \{-1,1\}^{n \times n}$
    is drawn uniformly at random from the set of symmetric signflip matrices.
\end{theorem}

\subsection{Main Result: Performance of NetFlipPA}
\label{sec:theory:netflippa:performance}

Having characterized the influence of signflipping on the signal $\btlS_{\alpha}$ and the noise $\btlN_{\alpha}$, we are now ready to state our main result.
\Cref{thm:DCSBM:noise:recovery} guarantees that the signflipped \normadj matrix asymptotically recovers the spectral noise floor $\opnorm{\btlN_{\alpha}}$, thus providing theoretical support for NetFlipPA with $T = 1$ and $\quantile = 1$ in the present DCSBM setting.
The proof is provided in \cref{proof:DCSBM:noise:recovery}.

\begin{theorem}[NetFlipPA Recovers the Noise Floor]
    \label{thm:DCSBM:noise:recovery}
    For any $\alpha \in \bbR$ and any $\mmt \in \bbN_{+}$, and for $\btlN_{\alpha}$ defined in \cref{eq:DCSBM:signalplusnoise,eq:DCSBM:noiseparts},
    it holds that
    \begin{equation}
        \ellP{ \opnorm{\bmR \circ \bmL_{\alpha}} - \opnorm{\btlN_{\alpha}} }
        \lesssim
        \left(\frac{\log n}{n}\right)^{1/2}
        ,
    \end{equation}
    where $\bmR \in \{-1,1\}^{n \times n}$ is drawn uniformly at random from the set of symmetric signflip matrices.
\end{theorem}

\Cref{thm:DCSBM:noise:recovery} immediately implies that various modes of convergence hold for
$\opnorm{\bmR \circ \bmL_{\alpha}} - \opnorm{\btlN_{\alpha}} \to 0$
in the large-network limit $n \to \infty$.
Notably, convergence in $L_{1}$ holds (as a special case) and almost sure convergence holds (by a standard Borel--Cantelli argument), both of which imply convergence in probability and hence also convergence in distribution.
Similarly, the moment control in \Cref{thm:DCSBM:noise:recovery} immediately implies a standard high-probability bound by, for example, applying Markov's inequality.

A key feature of \cref{thm:DCSBM:noise:recovery} is that it holds for any $\alpha \in \bbR$.
As such, NetFlipPA provides practitioners flexibility in choosing $\alpha$, namely in their choice of how aggressively to degree-normalize the data.
Notably, downstream of selecting the embedding dimension, the choice of $\alpha$ influences community detection, namely recovery of node community memberships as encoded in informative eigenvectors of the \normadj matrix.
This aspect, including the problem of optimizing $\alpha$ was studied in \cite{ali2017improved} for the setting where the degree heterogeneity parameters $\{q_{i}\}_{i=1}^{n}$ are independent and identically distributed.

Along similar lines, \cref{thm:DCSBM:noise:recovery} does not require explicit assumptions on the population eigenvectors, i.e., on the eigenvectors of \cref{eq:model:DCSBM:expectedA}.
As such, NetFlipPA is applicable even when the community sizes are very imbalanced.
In models with low-rank population structure and diverging spiked eigenvalues, such an imbalance would typically lead to spectral localization (i.e., mass concentrating on one or several entries of individual eigenvectors) and in turn to poor performance for spectral clustering.
However, in the present setting with non-diverging spiked eigenvalues, eigenvector delocalization still holds for the family of \normadj matrices we consider, as established by \cref{thm:DCSBM:approx:L:signflip:decay}.
Indeed, selecting the correct embedding dimension is in some sense a less stringent problem than obtaining clustering guarantees, and it does not rely on the amount of community membership information encoded in the extreme population eigenvectors.
Prior work in \cite{ali2016performance}, which considers the DCSBM setting from \cite{ali2017improved} and choice $\alpha = 1$, studies properties of informative leading eigenvectors for establishing clustering recovery guarantees but does not consider the problem of dimension selection.
Indeed, our work complements \cite{ali2016performance} which derives the phase transition where eigenvector-based clustering becomes asymptotically possible, determines the alignment between dominant eigenvectors and block membership vectors, and explores the implications of conjectured component-wise eigenvector asymptotic normality, the latter of which is further investigated for SBMs in \cite{kadavankandy2019agf}.

% =========================
\section{Related Work}
\label{sec:literature}

The present paper lies at the intersection of statistical network analysis and rank estimation for low-rank models.
This section reviews related work coming from these areas.

\subsection{Network Estimation and Recovery Under Random Graph Models}
\label{sec:literature:networks}

Selecting the embedding dimension in large heterogeneous networks is closely related to network estimation and recovery problems studied in random graph models.
These problems are often formulated as seeking to infer properties of a latent `population network' or `expected network' (e.g., community memberships, the number of communities, or probabilities of edge occurrence) given access only to noisy data \cite{abbe2018community,athreya2018statistical}.
Broadly speaking, the difficulty of these problems is influenced by the interaction of various modeling aspects, namely
the choice of network representation,
the amount of graph sparsity,
the extent of degree heterogeneity,
and the signal strength as reflected in the spectral properties of the population network.
Here, we elaborate upon similarities and differences between the modeling specifications in this paper and those of related works in the literature.

The NetFlipPA methodology developed in this paper applies broadly to a large family of network representations $\{\bmL_{\alpha} : \alpha \in \bbR\}$ and as such is not limited to a specific choice.
Historically, a more traditional approach when using perturbation analysis to study large random graphs has been to consider a single, specific choice of network representation, with the (unnormalized) adjacency matrix often being simpler and easier to analyze than graph Laplacians \cite{tang2018limit,ke2025optimal}.
The study of families of generalized adjacency or Laplacian matrices appears to be a relatively recent development in statistical machine learning, including and dating back to at least \cite{ali2017improved}.
More recently, \cite{ke2025optimal,fan2025asymptotic:arxiv:v1} study spectral-based asymptotics and inference for generalized Laplacians of the form $\widetilde{\bmD}^{-\alpha} \widetilde{\bmA} \widetilde{\bmD}^{-\alpha}$ with $\alpha \ge 0$ for some (possibly regularized) adjacency-type matrix $\widetilde{\bmA}$ and (possibly regularized) degree-type matrix $\widetilde{\bmD}$ but in the different low-rank setting of diverging spiked eigenvalues.

The DCSBM model in the weakly distinguishable community structure regime studied in this paper yields dense graphs, where expected node degrees satisfy $\Omega(n)$.
Broadly speaking, dense graphs are often expected to have the simplest mathematical properties (e.g., perfect clustering is possible), particularly when the population-level expected adjacency matrix has diverging spiked eigenvalues. 
As a result, more attention is often paid to non-dense or (at least somewhat) sparse graphs, namely graphs with expected node degrees that are sub-linear in $n$, e.g., satisfying $\Omega(n^{1/2})$ or $\Omega(\log n)$ or even $\Omega(1)$, since graph sparsity together with the number of nodes controls the amount of information present or available in the data via the number of observable or expected edges, and this can be quantified in the form of recovery guarantees and phase transition phenomena.
A distinguishing feature of this paper is that here the dense graphs have non-diverging network eigenvalues, and strong consistency or perfect recovery of the community memberships is not possible asymptotically, as noted in \cite{ali2017improved,kadavankandy2019agf}.
Consequently, the perturbation analysis used to obtain our main results differs substantially from existing analyses for dense and semi-sparse graph regimes; see for example \cite{athreya2018statistical}.

The DCSBM model studied in this paper permits graphs to exhibit mild degree heterogeneity, enabling node-specific differences in the propensity for edge formation.
Degree heterogeneity is a common modeling assumption in the literature, as are other variants of SBMs allowing more flexible network structure than just pure blocks.
Recently, extreme forms of degree heterogeneity have also been investigated in the literature, though for settings where networks have diverging spiked eigenvalues \cite{ke2025optimal}.

Interestingly, \cref{assump:model:DCSBM} does not impose restrictions on the block sizes, namely the cardinalities $\#\{j : g_{j} = k\}$ for $1 \le k \le K$.
This aspect is in stark contrast to much of the existing literature in statistical network analysis, including \cite{ali2017improved}, where positing relatively balanced block sizes is needed in order for stated estimation results to hold.
Furthermore, \cref{assump:model:DCSBM} does not restrict the relationship between the number of blocks $K$ and the number of nodes $n$, nor does it impose an explicit rank assumption on the expected adjacency matrix.
This can be understood by virtue of the fact that correctly selecting the embedding dimension is different from and antecedent to correctly estimating community memberships or edge probabilities.
For more on the distinction between the number of node communities versus matrix rank in network models and downstream statistical implications, see \cite{tang2022asymptotically,xie2023efficient}.

Finally, we note that selecting the embedding dimension is also related to the problem of estimating the number of communities $K$ in network data (though they are different problems as discussed in \cref{sec:netflipPA}).
This topic has received considerable attention in the literature for stochastic blockmodels and their generalizations, which beyond the citations and discussion provided earlier in \cref{sec:introduction} include \cite{chen2021estimating:arxiv:v1,ma2021determining} and the references therein.
As noted there, existing works generally fall into one of three categories: spectral-based, cross-validation-based, and likelihood-based.
To the best of our knowledge, none of these existing works adopt a parallel analysis perspective or accommodate the setting specified in \cref{sec:theory:model}.

\subsection{Rank Estimation for Principal Component Analysis and Factor Analysis}
\label{sec:literature:rankestimation}

Selecting the embedding dimension in large heterogeneous networks is closely related to the general problem of rank estimation in low-rank signal-plus-noise models, which has been heavily studied in the context of principal component analysis (PCA) and factor analysis (FA).
Classical, well-established methods include tests based on likelihood ratios~\cite{bartlett1954chisquare,lawley1956test}, the Kaiser--Guttman criterion~\cite{guttman1954snc,kaiser1960factor}, the scree plot~\cite{cattell1966scree,cattell1977scree,zhu2006automatic}, parallel analysis~\cite{horn1965, buja1992rop}, and methods based on information theoretic criteria~\cite{wax1985detection, fishler2002detection}, to name just a few.
As in our network setting, selecting too small a rank in PCA or FA risks losing potentially important information in the data while selecting too large a rank risks keeping too much of the noise, both of which can in turn deteriorate the performance of downstream analyses.

Recent years have witnessed significant ongoing interest in the problem of rank estimation for PCA and FA, with a particular focus on high-dimensional settings where both the row and column dimensions of the data matrix are large and may go to infinity.
These settings are applicable for many modern datasets in which the number of samples is comparable to the number of features, rather than having far more samples than features as is assumed in classical asymptotic analyses.
Various new and notable phenomena arise in high-dimensional settings that can impact rank estimation, such as eigenvalue spreading, bias, and universality; see \cite{johnstone2018pih} for a recent survey.
A number of recent works use new theoretical insights about high-dimensional phenomena to re-visit (and sometimes adapt) classical methods as well as to develop new rank estimation methods.
One area of work has been analyzing methods based on various information criteria and appropriately modifying them for high-dimensional settings~\cite{bai2002determining,nadakuditi2008sample,alessi2010factor,bai2018consistency,hu2020lla:arxiv:v1}.
Likewise, variants of parallel analysis have been analyzed and developed for high-dimensional settings~\cite{dobriban2017permutation,dobriban2018dpa}.
Cross-validation approaches have also been investigated, including bi-cross-validation~\cite{owen2016bicross} and double cross validation~\cite{zeng2019double:arxiv:v1}.
Finally, a major area of work has been in developing techniques that exploit the asymptotic behavior of eigenvalues, their consecutive differences, or their consecutive ratios in high-dimensional regimes~\cite{kapetanios2004factor,kapetanios2010factor,kritchman2009non,onatski2010factor,lam2012factor,ahn2013eigenvalue,passemier2014estimation,li2017identifying,fan2020eno,xu2022eigenvalue}.

Several recent works have also begun to develop techniques for data that are not only high-dimensional but also have heteroscedastic noise, which is common in modern settings.
Indeed, the network data we consider are heteroscedastic because the variances of the Bernoulli entries in the adjacency matrix are functions of the corresponding edge probabilities.
The same occurs in general for Bernoulli or Poisson data, which are common in modern binary and count data.
Moreover, heteroscedastic noise arises naturally in the modern setting where datasets are formed by combining data from multiple sources that may differ in their level of quality.
Methods developed for homoscedastic noise can degrade dramatically when the noise is heteroscedastic, motivating the need for new methods tailored to heteroscedastic settings.
For example, heteroscedasticity in the noise can significantly degrade the quality of the eigenvectors of the sample covariance matrix, and a number of recent works study how to appropriately modify PCA to improve its performance~\cite{zhang2022hpa,leeb2021oss,leeb2021mdf,hong2016tat,hong2018asymptotic,hong2023owp,hong2021hpp,gilman2025shp,cavazos2025asl}.

Heteroscedasticity in the noise can also significantly impact rank estimation.
For example, \cite[Section~3.2]{hong2020stn:arxiv:v4} shows how permutation-based parallel analysis can degrade dramatically in heteroscedastic settings.
One approach to rank estimation for heteroscedastic settings has been to develop techniques that appropriately whiten the data to make the noise behave as though it were homoscedastic: \cite{landa2022brt} proposes a technique for whitening the noise in count data, and \cite{landa2025tde} proposes a technique based on random matrix theoretic properties of the Dyson equation.
Beyond whitening approaches, \cite{donoho2023sem} proposes a data-adaptive random matrix theoretic threshold that seeks to minimize the mean-square error of the low-rank matrix estimate obtained after thresholding.
Data-adaptive random matrix theoretic thresholding techniques have also been developed for settings where the noise variances are themselves Gamma distributed \cite{ke2021eot}.
Yet another, different approach, developed in \cite{han2023universal}, is to test for matrix rank on the basis of residual subsampling, which is provably effective in both network and non-network heteroscedastic models that have diverging spiked eigenvalues.

Most closely related to this paper is \cite{hong2020stn:arxiv:v4} which develops a parallel analysis variant for non-network data with heteroscedastic noise based on random signflips.
In contrast to that work, the network data we consider are symmetric matrices that can have both positive and negative eigenvalues of interest and the noise entries do not have symmetric distributions.
Consequently, the theoretical guarantees therein cannot be applied directly.
Moreover, their consistency results assume asymptotic signal delocalization and the existence of a well-defined asymptotic spectral noise floor.
In this paper, we instead perform a non-asymptotic perturbation analysis that shows these properties occur in our DCSBM setting and yields corresponding convergence rates.
For these reasons, the theoretical techniques used to obtain our main results differ substantially from the analyses in \cite{hong2020stn:arxiv:v4}.

% =========================
\section{Conclusion}
\label{sec:conclusion}

This paper develops a novel spectral approach to selecting the embedding dimension for network data that applies to a broad class of \normadj matrices.
The proposed algorithm, network signflip parallel analysis (NetFlipPA), compares the eigenvalues of the given matrix to the operator norms obtained after multiple trials of randomly signflipping its entries.
NetFlipPA is interpretable, simple to implement, and does not require users to carefully tune parameters.

This paper analyzes NetFlipPA in the challenging setting of large degree-corrected stochastic blockmodel random graphs with weakly distinguishable community structure, where the network eigenvalues corresponding to community structure do not diverge away from those corresponding to randomness in the network.
Proper selection in this setting requires careful calibration to keep only the network eigenvalues whose magnitudes rise above the spectral noise floor coming from randomness in the network, since the network eigenvalues with magnitudes below the spectral noise floor are likely to have uninformative corresponding eigenvectors.
One does not typically know the spectral noise floor \emph{a priori}, and high-probability upper bounds such as those obtained via high-dimensional probability can be overly conservative, making this a challenging problem and motivating our work on NetFlipPA.
This paper shows that NetFlipPA provably recovers the spectral noise floor (\cref{thm:DCSBM:noise:recovery}) and thus provides a rigorous approach to selecting the embedding dimension.
Proving our main result involves a nontrivial non-asymptotic analysis to first characterize the relevant signal and noise structure in the \normadj matrix (\cref{thm:DCSBM:approx:L:signflip:decay}) and then to quantify the effect of signflipping on the signal (\cref{thm:DCSBM:signal:signflip:destroy}) and the noise (\cref{thm:DCSBM:noise:signflip:preserve}).
Our intermediate results provide a refined non-asymptotic characterization of the \normadj matrix that may be of independent interest.

A natural direction for future work is to develop a better practical understanding of how to choose the quantile $\quantile$ and number of trials $T$ in NetFlipPA.
While the choice of $\quantile$ and $T$ typically do not have a major impact for large networks (as illustrated in \cref{sec:qT:experiment}), they can significantly impact the trade-off between Type I error and statistical power for smaller networks.
Another direction for future work is to develop techniques for networks that contain both weakly distinguishable community structure and strongly distinguishable community structure.
Strongly distinguishable community structure can produce very large signal eigenvalues that are not easily destroyed by signflipping, which can in turn lead to over-estimation of the spectral noise floor and cause weakly distinguishable community structure to be overlooked.
This problem is a challenge for parallel analysis methods in general, where it often goes under the name ``shadowing''.
It would also be interesting to study the performance of NetFlipPA in the context of moderately sparse graphs, possibly exhibiting moderate or severe degree heterogeneity.
In such settings, it is expected that augmenting NetFlipPA with graph regularization techniques would be needed to ensure underlying high-dimensional concentration and thus accurate performance (e.g., see \cite{le2017concentration, zhang2024fundamental}).
Finally, given recent renewed interest in signed networks \cite{cucuringu2019sponge,feng2020testing,pensky2024signed:arxiv:v2,tang2025population}, it might also be interesting to extend NetFlipPA to signed networks, or to more general edge-weighted networks.

% =========================
\bibliographystyle{IEEEtran}
\bibliography{refs-arxiv}

\begin{thebibliography}{10}
\providecommand{\url}[1]{#1}
\csname url@samestyle\endcsname
\providecommand{\newblock}{\relax}
\providecommand{\bibinfo}[2]{#2}
\providecommand{\BIBentrySTDinterwordspacing}{\spaceskip=0pt\relax}
\providecommand{\BIBentryALTinterwordstretchfactor}{4}
\providecommand{\BIBentryALTinterwordspacing}{\spaceskip=\fontdimen2\font plus
\BIBentryALTinterwordstretchfactor\fontdimen3\font minus
  \fontdimen4\font\relax}
\providecommand{\BIBforeignlanguage}[2]{{%
\expandafter\ifx\csname l@#1\endcsname\relax
\typeout{** WARNING: IEEEtran.bst: No hyphenation pattern has been}%
\typeout{** loaded for the language `#1'. Using the pattern for}%
\typeout{** the default language instead.}%
\else
\language=\csname l@#1\endcsname
\fi
#2}}
\providecommand{\BIBdecl}{\relax}
\BIBdecl

\bibitem{newman2018networks}
M.~Newman, \emph{Networks}, 2nd~ed.\hskip 1em plus 0.5em minus 0.4em\relax
  London, U.K.: Oxford University Press, 2018.

\bibitem{peterson2011computer}
L.~L. Peterson and B.~S. Davie, \emph{Computer Networks: A Systems Approach},
  5th~ed.\hskip 1em plus 0.5em minus 0.4em\relax Burlington, MA, USA: Morgan
  Kaufmann Publishers Inc., 2011.

\bibitem{snijders2011statistical}
T.~A.~B. Snijders, ``Statistical models for social networks,'' \emph{Annual
  Review of Sociology}, vol.~37, no.~1, pp. 131--153, 2011.

\bibitem{graham2020econometric}
B.~Graham and {\'{A}}.~de~Paula, \emph{The Econometric Analysis of Network
  Data}.\hskip 1em plus 0.5em minus 0.4em\relax Academic Press, 2020.

\bibitem{junker2007biological}
B.~H. Junker and F.~Schreiber, \emph{Analysis of Biological Networks}.\hskip
  1em plus 0.5em minus 0.4em\relax Hoboken, NJ, USA: Wiley, 2007.

\bibitem{fornito2016fundamentals}
A.~Fornito, A.~Zalesky, and E.~Bullmore, \emph{Fundamentals of Brain Network
  Analysis}.\hskip 1em plus 0.5em minus 0.4em\relax Academic Press, 2016.

\bibitem{bordier2017graph}
C.~Bordier, C.~Nicolini, and A.~Bifone, ``Graph analysis and modularity of
  brain functional connectivity networks: Searching for the optimal
  threshold,'' \emph{Frontiers in Neuroscience}, vol.~11, 2017.

\bibitem{theis2023threshold}
N.~Theis, J.~Rubin, J.~Cape, S.~Iyengar, and K.~M. Prasad, ``Threshold
  selection for brain connectomes,'' \emph{Brain Connectivity}, vol.~13, no.~7,
  pp. 383--393, 2023.

\bibitem{chung1997spectral}
F.~Chung, \emph{Spectral Graph Theory}.\hskip 1em plus 0.5em minus 0.4em\relax
  Providence, RI, USA: American Mathematical Society, 1997.

\bibitem{biggs1974algebraic}
N.~Biggs, \emph{Algebraic Graph Theory}, 1st~ed.\hskip 1em plus 0.5em minus
  0.4em\relax Cambridge, U.K.: Cambridge University Press, 1974.

\bibitem{qin2013regularized}
T.~Qin and K.~Rohe, ``Regularized spectral clustering under the
  degree-corrected stochastic blockmodel,'' in \emph{Advances in Neural
  Information Processing Systems}, vol.~26, 2013.

\bibitem{le2017concentration}
C.~M. Le, E.~Levina, and R.~Vershynin, ``Concentration and regularization of
  random graphs,'' \emph{Random Structures \& Algorithms}, vol.~51, no.~3, pp.
  538--561, 2017.

\bibitem{sarkar2015normalization}
P.~Sarkar and P.~J. Bickel, ``Role of normalization in spectral clustering for
  stochastic blockmodels,'' \emph{The Annals of Statistics}, vol.~43, no.~3,
  pp. 962--990, 2015.

\bibitem{saade2014spectral}
A.~Saade, F.~Krzakala, and L.~Zdeborov{\'{a}}, ``Spectral clustering of graphs
  with the {Bethe Hessian},'' in \emph{Advances in Neural Information
  Processing Systems}, vol.~27, 2014.

\bibitem{priebe2019two}
C.~E. Priebe, Y.~Park, J.~T. Vogelstein, J.~M. Conroy, V.~Lyzinski, M.~Tang,
  A.~Athreya, J.~Cape, and E.~Bridgeford, ``On a two-truths phenomenon in
  spectral graph clustering,'' \emph{Proceedings of the National Academy of
  Sciences}, vol. 116, no.~13, pp. 5995--6000, 2019.

\bibitem{holland1983stochastic}
P.~W. Holland, K.~B. Laskey, and S.~Leinhardt, ``Stochastic blockmodels: First
  steps,'' \emph{Social Networks}, vol.~5, no.~2, pp. 109--137, 1983.

\bibitem{anderson1992building}
C.~J. Anderson, S.~Wasserman, and K.~Faust, ``Building stochastic
  blockmodels,'' \emph{Social Networks}, vol.~14, no. 1--2, pp. 137--161, 1992.

\bibitem{sussman2012consistent}
D.~L. Sussman, M.~Tang, D.~E. Fishkind, and C.~E. Priebe, ``A consistent
  adjacency spectral embedding for stochastic blockmodel graphs,''
  \emph{Journal of the American Statistical Association}, vol. 107, no. 499,
  pp. 1119--1128, 2012.

\bibitem{hajek2016achieving}
B.~Hajek, Y.~Wu, and J.~Xu, ``Achieving exact cluster recovery threshold via
  semidefinite programming,'' \emph{IEEE Transactions on Information Theory},
  vol.~62, no.~5, pp. 2788--2797, 2016.

\bibitem{su2020strong}
L.~Su, W.~Wang, and Y.~Zhang, ``Strong consistency of spectral clustering for
  stochastic block models,'' \emph{IEEE Transactions on Information Theory},
  vol.~66, no.~1, pp. 324--338, 2020.

\bibitem{zhang2024fundamental}
A.~Y. Zhang, ``Fundamental limits of spectral clustering in stochastic block
  models,'' \emph{IEEE Transactions on Information Theory}, vol.~70, no.~10,
  pp. 7320--7348, 2024.

\bibitem{cerqueira2020estimation}
A.~Cerqueira and F.~Leonardi, ``Estimation of the number of communities in the
  stochastic block model,'' \emph{IEEE Transactions on Information Theory},
  vol.~66, no.~10, pp. 6403--6412, 2020.

\bibitem{jin2023optimal}
J.~Jin, Z.~T. Ke, S.~Luo, and M.~Wang, ``Optimal estimation of the number of
  network communities,'' \emph{Journal of the American Statistical
  Association}, vol. 118, no. 543, pp. 2101--2116, 2023.

\bibitem{hwang2024estimation}
N.~Hwang, J.~Xu, S.~Chatterjee, and S.~Bhattacharyya, ``On the estimation of
  the number of communities for sparse networks,'' \emph{Journal of the
  American Statistical Association}, vol. 119, no. 547, pp. 1895--1910, 2024.

\bibitem{fan2022simple}
J.~Fan, Y.~Fan, X.~Han, and J.~Lv, ``{SIMPLE}: Statistical inference on
  membership profiles in large networks,'' \emph{Journal of the Royal
  Statistical Society Series B: Statistical Methodology}, vol.~84, no.~2, pp.
  630--653, 2022.

\bibitem{fan2022simplerc:arxiv:v1}
\BIBentryALTinterwordspacing
J.~Fan, Y.~Fan, J.~Lv, and F.~Yang, ``{SIMPLE-RC}: Group network inference with
  non-sharp nulls and weak signals,'' 2022. [Online]. Available:
  \url{https://arxiv.org/abs/2211.00128v1}
\BIBentrySTDinterwordspacing

\bibitem{du2023hypothesis}
X.~Du and M.~Tang, ``Hypothesis testing for equality of latent positions in
  random graphs,'' \emph{Bernoulli}, vol.~29, no.~4, pp. 3221--3254, 2023.

\bibitem{horn1965}
J.~L. Horn, ``A rationale and test for the number of factors in factor
  analysis,'' \emph{Psychometrika}, vol.~30, no.~2, pp. 179--185, 1965.

\bibitem{buja1992rop}
A.~Buja and N.~Eyuboglu, ``Remarks on parallel analysis,'' \emph{Multivariate
  Behavioral Research}, vol.~27, no.~4, pp. 509--540, 1992.

\bibitem{dobriban2017permutation}
E.~Dobriban, ``Permutation methods for factor analysis and {PCA},'' \emph{The
  Annals of Statistics}, vol.~48, no.~5, pp. 2824--2847, 2020.

\bibitem{hong2020stn:arxiv:v4}
\BIBentryALTinterwordspacing
D.~Hong, Y.~Sheng, and E.~Dobriban, ``Selecting the number of components in
  {PCA} via random signflips,'' 2025. [Online]. Available:
  \url{https://arxiv.org/abs/2012.02985v4}
\BIBentrySTDinterwordspacing

\bibitem{ali2017improved}
H.~T. Ali and R.~Couillet, ``Improved spectral community detection in large
  heterogeneous networks,'' \emph{Journal of Machine Learning Research},
  vol.~18, no.~1, pp. 8344--8392, 2017.

\bibitem{ajanki2016ufg}
O.~H. Ajanki, L.~Erd{\H{o}}s, and T.~Kr{\"{u}}ger, ``Universality for general
  {Wigner}-type matrices,'' \emph{Probability Theory and Related Fields}, vol.
  169, no. 3--4, pp. 667--727, 2016.

\bibitem{ajanki2019qve}
------, ``Quadratic vector equations on complex upper half-plane,''
  \emph{Memoirs of the American Mathematical Society}, vol. 261, no. 1261, pp.
  1--133, 2019.

\bibitem{newman2006modularity}
M.~E.~J. Newman, ``Modularity and community structure in networks,''
  \emph{Proceedings of the National Academy of Sciences}, vol. 103, no.~23, pp.
  8577--8582, 2006.

\bibitem{jin2015fast}
J.~Jin, ``Fast community detection by {SCORE},'' \emph{The Annals of
  Statistics}, vol.~43, no.~1, pp. 57--89, 2015.

\bibitem{coja-oghlan2009finding}
A.~Coja-Oghlan and A.~Lanka, ``Finding planted partitions in random graphs with
  general degree distributions,'' \emph{SIAM Journal on Discrete Mathematics},
  vol.~23, no.~4, pp. 1682--1714, 2010.

\bibitem{gulikers2017spectral}
L.~Gulikers, M.~Lelarge, and L.~Massouli{\'{e}}, ``A spectral method for
  community detection in moderately sparse degree-corrected stochastic block
  models,'' \emph{Advances in Applied Probability}, vol.~49, no.~3, pp.
  686--721, 2017.

\bibitem{adamic2005political}
L.~A. Adamic and N.~Glance, ``The political blogosphere and the 2004 {U.S}.
  election: divided they blog,'' in \emph{Proceedings of the 3rd International
  Workshop on Link Discovery}, 2005, pp. 36--43.

\bibitem{karrer2011stochastic}
B.~Karrer and M.~E.~J. Newman, ``Stochastic blockmodels and community structure
  in networks,'' \emph{Physical Review E}, vol.~83, no.~1, 2011.

\bibitem{kadavankandy2019agf}
A.~Kadavankandy and R.~Couillet, ``Asymptotic {Gaussian} fluctuations of
  spectral clustering eigenvectors,'' in \emph{Proceedings of the 8th IEEE
  International Workshop on Computational Advances in Multi-Sensor Adaptive
  Processing}, 2019, pp. 694--698.

\bibitem{ali2016performance}
H.~T. Ali and R.~Couillet, ``Performance analysis of spectral community
  detection in realistic graph models,'' in \emph{Proceedings of the IEEE
  International Conference on Acoustics, Speech and Signal Processing}, 2016,
  pp. 4548--4552.

\bibitem{abbe2018community}
E.~Abbe, ``Community detection and stochastic block models: Recent
  developments,'' \emph{Journal of Machine Learning Research}, vol.~18, no.
  177, pp. 1--86, 2018.

\bibitem{athreya2018statistical}
A.~Athreya, D.~E. Fishkind, M.~Tang, C.~E. Priebe, Y.~Park, J.~T. Vogelstein,
  K.~Levin, V.~Lyzinski, Y.~Qin, and D.~L. Sussman, ``Statistical inference on
  random dot product graphs: A survey,'' \emph{Journal of Machine Learning
  Research}, vol.~18, no. 226, pp. 1--92, 2018.

\bibitem{tang2018limit}
M.~Tang and C.~E. Priebe, ``Limit theorems for eigenvectors of the normalized
  {Laplacian} for random graphs,'' \emph{The Annals of Statistics}, vol.~46,
  no.~5, pp. 2360--2415, 2018.

\bibitem{ke2025optimal}
Z.~T. Ke and J.~Wang, ``Optimal network membership estimation under severe
  degree heterogeneity,'' \emph{Journal of the American Statistical
  Association}, vol. 120, no. 550, pp. 948--962, 2025.

\bibitem{fan2025asymptotic:arxiv:v1}
\BIBentryALTinterwordspacing
J.~Fan, Y.~Fan, J.~Lv, F.~Yang, and D.~Yu, ``Asymptotic theory of eigenvectors
  for latent embeddings with generalized {Laplacian} matrices,'' 2025.
  [Online]. Available: \url{https://arxiv.org/abs/2503.00640v1}
\BIBentrySTDinterwordspacing

\bibitem{tang2022asymptotically}
M.~Tang, J.~Cape, and C.~E. Priebe, ``Asymptotically efficient estimators for
  stochastic blockmodels: The naive {MLE}, the rank-constrained {MLE}, and the
  spectral estimator,'' \emph{Bernoulli}, vol.~28, no.~2, pp. 1049--1073, 2022.

\bibitem{xie2023efficient}
F.~Xie and Y.~Xu, ``Efficient estimation for random dot product graphs via a
  one-step procedure,'' \emph{Journal of the American Statistical Association},
  vol. 118, no. 541, pp. 651--664, 2023.

\bibitem{chen2021estimating:arxiv:v1}
\BIBentryALTinterwordspacing
F.~Chen, S.~Roch, K.~Rohe, and S.~Yu, ``Estimating graph dimension with
  cross-validated eigenvalues,'' 2021. [Online]. Available:
  \url{https://arxiv.org/abs/2108.03336v1}
\BIBentrySTDinterwordspacing

\bibitem{ma2021determining}
S.~Ma, L.~Su, and Y.~Zhang, ``Determining the number of communities in
  degree-corrected stochastic block models,'' \emph{Journal of Machine Learning
  Research}, vol.~22, no.~69, pp. 1--63, 2021.

\bibitem{bartlett1954chisquare}
M.~S. Bartlett, ``A note on the multiplying factors for various $\chi^2$
  approximations,'' \emph{Journal of the Royal Statistical Society Series B:
  Statistical Methodology}, vol.~16, no.~2, pp. 296--298, 1954.

\bibitem{lawley1956test}
D.~N. Lawley, ``Tests of significance for the latent roots of covariance and
  correlation matrices,'' \emph{Biometrika}, vol.~43, no. 1--2, pp. 128--136,
  1956.

\bibitem{guttman1954snc}
L.~Guttman, ``Some necessary conditions for common-factor analysis,''
  \emph{Psychometrika}, vol.~19, no.~2, pp. 149--161, 1954.

\bibitem{kaiser1960factor}
H.~F. Kaiser, ``The application of electronic computers to factor analysis,''
  \emph{Educational and Psychological Measurement}, vol.~20, no.~1, pp.
  141--151, 1960.

\bibitem{cattell1966scree}
R.~B. Cattell, ``The scree test for the number of factors,'' \emph{Multivariate
  Behavioral Research}, vol.~1, no.~2, pp. 245--276, 1966.

\bibitem{cattell1977scree}
R.~B. Cattell and S.~Vogelmann, ``A comprehensive trial of the scree and {KG}
  criteria for determining the number of factors,'' \emph{Multivariate
  Behavioral Research}, vol.~12, no.~3, pp. 289--325, 1977.

\bibitem{zhu2006automatic}
M.~Zhu and A.~Ghodsi, ``Automatic dimensionality selection from the scree plot
  via the use of profile likelihood,'' \emph{Computational Statistics \& Data
  Analysis}, vol.~51, no.~2, pp. 918--930, 2006.

\bibitem{wax1985detection}
M.~Wax and T.~Kailath, ``Detection of signals by information theoretic
  criteria,'' \emph{IEEE Transactions on Acoustics, Speech, and Signal
  Processing}, vol.~33, no.~2, pp. 387--392, 1985.

\bibitem{fishler2002detection}
E.~Fishler, M.~Grosmann, and H.~Messer, ``Detection of signals by information
  theoretic criteria: general asymptotic performance analysis,'' \emph{IEEE
  Transactions on Signal Processing}, vol.~50, no.~5, pp. 1027--1036, 2002.

\bibitem{johnstone2018pih}
I.~M. Johnstone and D.~Paul, ``{PCA} in high dimensions: An orientation,''
  \emph{Proceedings of the IEEE}, vol. 106, no.~8, pp. 1277--1292, 2018.

\bibitem{bai2002determining}
J.~Bai and S.~Ng, ``Determining the number of factors in approximate factor
  models,'' \emph{Econometrica}, vol.~70, no.~1, pp. 191--221, 2002.

\bibitem{nadakuditi2008sample}
R.~R. Nadakuditi and A.~Edelman, ``Sample eigenvalue based detection of
  high-dimensional signals in white noise using relatively few samples,''
  \emph{IEEE Transactions on Signal Processing}, vol.~56, no.~7, pp.
  2625--2638, 2008.

\bibitem{alessi2010factor}
L.~Alessi, M.~Barigozzi, and M.~Capasso, ``Improved penalization for
  determining the number of factors in approximate factor models,''
  \emph{Statistics \& Probability Letters}, vol.~80, no. 23--24, pp.
  1806--1813, 2010.

\bibitem{bai2018consistency}
Z.~Bai, K.~P. Choi, and Y.~Fujikoshi, ``Consistency of {AIC} and {BIC} in
  estimating the number of significant components in high-dimensional principal
  component analysis,'' \emph{The Annals of Statistics}, vol.~46, no.~3, pp.
  1050--1076, 2018.

\bibitem{hu2020lla:arxiv:v1}
\BIBentryALTinterwordspacing
J.~Hu, J.~Zhang, J.~Zhu, and J.~Guo, ``Limiting laws and consistent estimation
  criteria for fixed and diverging number of spiked eigenvalues,'' 2020.
  [Online]. Available: \url{https://arxiv.org/abs/2012.08371v1}
\BIBentrySTDinterwordspacing

\bibitem{dobriban2018dpa}
E.~Dobriban and A.~B. Owen, ``Deterministic parallel analysis: An improved
  method for selecting factors and principal components,'' \emph{Journal of the
  Royal Statistical Society Series B: Statistical Methodology}, vol.~81, no.~1,
  pp. 163--183, 2019.

\bibitem{owen2016bicross}
A.~B. Owen and J.~Wang, ``Bi-cross-validation for factor analysis,''
  \emph{Statistical Science}, vol.~31, no.~1, pp. 119--139, 2016.

\bibitem{zeng2019double:arxiv:v1}
\BIBentryALTinterwordspacing
X.~Zeng, Y.~Xia, and L.~Zhang, ``Double cross validation for the number of
  factors in approximate factor models,'' 2019. [Online]. Available:
  \url{https://arxiv.org/abs/1907.01670v1}
\BIBentrySTDinterwordspacing

\bibitem{kapetanios2004factor}
G.~Kapetanios, ``A new method for determining the number of factors in factor
  models with large datasets,'' Queen Mary University of London, School of
  Economics and Finance, Working Papers 525, 2004.

\bibitem{kapetanios2010factor}
------, ``A testing procedure for determining the number of factors in
  approximate factor models with large datasets,'' \emph{Journal of Business \&
  Economic Statistics}, vol.~28, no.~3, pp. 397--409, 2010.

\bibitem{kritchman2009non}
S.~Kritchman and B.~Nadler, ``Non-parametric detection of the number of
  signals: Hypothesis testing and random matrix theory,'' \emph{IEEE
  Transactions on Signal Processing}, vol.~57, no.~10, pp. 3930--3941, 2009.

\bibitem{onatski2010factor}
A.~Onatski, ``Determining the number of factors from empirical distribution of
  eigenvalues,'' \emph{The Review of Economics and Statistics}, vol.~92, no.~4,
  pp. 1004--1016, 2010.

\bibitem{lam2012factor}
C.~Lam and Q.~Yao, ``Factor modeling for high-dimensional time series:
  Inference for the number of factors,'' \emph{The Annals of Statistics},
  vol.~40, no.~2, pp. 694--726, 2012.

\bibitem{ahn2013eigenvalue}
S.~C. Ahn and A.~R. Horenstein, ``Eigenvalue ratio test for the number of
  factors,'' \emph{Econometrica}, vol.~81, no.~3, pp. 1203--1227, 2013.

\bibitem{passemier2014estimation}
D.~Passemier and J.~Yao, ``Estimation of the number of spikes, possibly equal,
  in the high-dimensional case,'' \emph{Journal of Multivariate Analysis}, vol.
  127, pp. 173--183, 2014.

\bibitem{li2017identifying}
Z.~Li, Q.~Wang, and J.~Yao, ``Identifying the number of factors from singular
  values of a large sample auto-covariance matrix,'' \emph{The Annals of
  Statistics}, vol.~45, no.~1, pp. 257--288, 2017.

\bibitem{fan2020eno}
J.~Fan, J.~Guo, and S.~Zheng, ``Estimating number of factors by adjusted
  eigenvalues thresholding,'' \emph{Journal of the American Statistical
  Association}, vol. 117, no. 538, pp. 852--861, 2022.

\bibitem{xu2022eigenvalue}
Y.~Xu, Z.~Liu, and J.~Yao, ``An eigenvalue ratio approach to inferring
  population structure from whole genome sequencing data,'' \emph{Biometrics},
  vol.~79, no.~2, pp. 891--902, 2023.

\bibitem{zhang2022hpa}
A.~R. Zhang, T.~T. Cai, and Y.~Wu, ``Heteroskedastic {PCA}: Algorithm,
  optimality, and applications,'' \emph{The Annals of Statistics}, vol.~50,
  no.~1, pp. 53--80, 2022.

\bibitem{leeb2021oss}
W.~Leeb and E.~Romanov, ``Optimal spectral shrinkage and {PCA} with
  heteroscedastic noise,'' \emph{IEEE Transactions on Information Theory},
  vol.~67, no.~5, pp. 3009--3037, 2021.

\bibitem{leeb2021mdf}
W.~E. Leeb, ``Matrix denoising for weighted loss functions and heterogeneous
  signals,'' \emph{SIAM Journal on Mathematics of Data Science}, vol.~3, no.~3,
  pp. 987--1012, 2021.

\bibitem{hong2016tat}
D.~Hong, L.~Balzano, and J.~A. Fessler, ``Towards a theoretical analysis of
  {PCA} for heteroscedastic data,'' in \emph{Proceedings of the 54th Annual
  Allerton Conference on Communication, Control, and Computing}, 2016, pp.
  496--503.

\bibitem{hong2018asymptotic}
------, ``Asymptotic performance of {PCA} for high-dimensional heteroscedastic
  data,'' \emph{Journal of Multivariate Analysis}, vol. 167, pp. 435--452,
  2018.

\bibitem{hong2023owp}
D.~Hong, F.~Yang, J.~A. Fessler, and L.~Balzano, ``Optimally weighted {PCA} for
  high-dimensional heteroscedastic data,'' \emph{SIAM Journal on Mathematics of
  Data Science}, vol.~5, no.~1, pp. 222--250, 2023.

\bibitem{hong2021hpp}
D.~Hong, K.~Gilman, L.~Balzano, and J.~A. Fessler, ``{HePPCAT}: Probabilistic
  {PCA} for data with heteroscedastic noise,'' \emph{IEEE Transactions on
  Signal Processing}, vol.~69, pp. 4819--4834, 2021.

\bibitem{gilman2025shp}
K.~Gilman, D.~Hong, J.~A. Fessler, and L.~Balzano, ``Streaming heteroscedastic
  probabilistic {PCA} with missing data,'' \emph{Transactions on Machine
  Learning Research}, 2025.

\bibitem{cavazos2025asl}
J.~S. Cavazos, J.~A. Fessler, and L.~Balzano, ``{ALPCAH}: Subspace learning for
  sample-wise heteroscedastic data,'' \emph{IEEE Transactions on Signal
  Processing}, vol.~73, pp. 876--886, 2025.

\bibitem{landa2022brt}
B.~Landa, T.~T. C.~K. Zhang, and Y.~Kluger, ``Biwhitening reveals the rank of a
  count matrix,'' \emph{SIAM Journal on Mathematics of Data Science}, vol.~4,
  no.~4, pp. 1420--1446, 2022.

\bibitem{landa2025tde}
B.~Landa and Y.~Kluger, ``The {Dyson} equalizer: adaptive noise stabilization
  for low-rank signal detection and recovery,'' \emph{Information and
  Inference: A Journal of the IMA}, vol.~14, no.~1, 2025.

\bibitem{donoho2023sem}
D.~Donoho, M.~Gavish, and E.~Romanov, ``{ScreeNOT}: Exact {MSE-optimal}
  singular value thresholding in correlated noise,'' \emph{The Annals of
  Statistics}, vol.~51, no.~1, pp. 122--148, 2023.

\bibitem{ke2021eot}
Z.~T. Ke, Y.~Ma, and X.~Lin, ``Estimation of the number of spiked eigenvalues
  in a covariance matrix by bulk eigenvalue matching analysis,'' \emph{Journal
  of the American Statistical Association}, vol. 118, no. 541, pp. 374--392,
  2023.

\bibitem{han2023universal}
X.~Han, Q.~Yang, and Y.~Fan, ``Universal rank inference via residual
  subsampling with application to large networks,'' \emph{The Annals of
  Statistics}, vol.~51, no.~3, pp. 1109--1133, 2023.

\bibitem{cucuringu2019sponge}
M.~Cucuringu, P.~Davies, A.~Glielmo, and H.~Tyagi, ``{SPONGE}: A generalized
  eigenproblem for clustering signed networks,'' in \emph{Proceedings of the
  Twenty-Second International Conference on Artificial Intelligence and
  Statistics}, vol.~89, 2019, pp. 1088--1098.

\bibitem{feng2020testing}
D.~Feng, R.~Altmeyer, D.~Stafford, N.~A. Christakis, and H.~H. Zhou, ``Testing
  for balance in social networks,'' \emph{Journal of the American Statistical
  Association}, vol. 117, no. 537, pp. 156--174, 2022.

\bibitem{pensky2024signed:arxiv:v2}
M.~Pensky, ``Signed diverse multiplex networks: Clustering and inference,''
  \emph{IEEE Transactions on Information Theory}, vol.~71, no.~9, pp.
  7076--7096, 2025.

\bibitem{tang2025population}
W.~Tang and J.~Zhu, ``Population-level balance in signed networks,''
  \emph{Journal of the American Statistical Association}, vol. 120, no. 550,
  pp. 751--763, 2025.

\bibitem{vaart1996wca}
A.~W. {van der Vaart} and J.~A. Wellner, \emph{Weak Convergence and Empirical
  Processes}, 1st~ed.\hskip 1em plus 0.5em minus 0.4em\relax New York, NY, USA:
  Springer-Verlag, 1996.

\bibitem{vershynin2018hdp}
R.~Vershynin, \emph{High-Dimensional Probability}.\hskip 1em plus 0.5em minus
  0.4em\relax Cambridge, U.K.: Cambridge University Press, 2018.

\end{thebibliography}

% =========================
\clearpage
\appendices
\crefalias{section}{appendix}

% =========================
\section{Supporting Lemmas}

Here, we state and prove several concentration inequalities that will be used in the proofs of our main results.

\begin{lemma}
    \label{lem:max:subgaussian}
    Let $X_{1},X_{2},\dots,X_{n}$ be sub-Gaussian random variables
    (not necessarily independent),
    where $n \geq 2$.
    Then, $\max_{i=1,\dots,n} |X_{i}|$ is also sub-Gaussian
    with
    \begin{equation*}
        \left\| \max_{i=1,\dots,n} |X_{i}| \right\|_{\psi_{2}}
        \leq
        C \left( \max_{i=1,\dots,n} \| X_{i} \|_{\psi_{2}} \right) \sqrt{\log n}
        ,
    \end{equation*}
    where $C > 0$ is a universal constant.
\end{lemma}

\Cref{lem:max:subgaussian} is a special case of
\cite[Lemma 2.2.2]{vaart1996wca};
here we provide an alternative elementary proof.

\begin{proof}[Proof of \cref{lem:max:subgaussian}]
    We prove the lemma by showing that
    \begin{equation*}
        \forall_{p \geq 1}
        \qquad
        \left(
            \bbE \left|
                \max_{i=1,\dots,n} |X_{i}|
            \right|^{p}
        \right)^{1/p}
        \leq
        \left[
            C \left( \max_{i=1,\dots,n} \| X_{i} \|_{\psi_{2}} \right) \sqrt{\log n}
        \right]
        \sqrt{p}
        .
    \end{equation*}
    Let $p \geq 1$ be arbitrary.
    Then,
    defining $\ell \coloneqq \log_2 n$
    yields
    \begin{equation*}
        \left| \max_{i=1,\dots,n} |X_{i}| \right|^{p}
        =
        \max_{i=1,\dots,n} |X_{i}|^{p}
        =
        \begin{Vmatrix}
            |X_{1}|^{p} \\ \vdots \\ |X_{n}|^{p}
        \end{Vmatrix}_{\infty}
        \leq
        \begin{Vmatrix}
            |X_{1}|^{p} \\ \vdots \\ |X_{n}|^{p}
        \end{Vmatrix}_{\ell}
        =
        \left(
            \sum_{i=1}^{n} \big| |X_{i}|^{p} \big|^{\ell}
        \right)^{1/\ell}
        =
        \left(
            \sum_{i=1}^{n} |X_{i}|^{p\ell}
        \right)^{1/\ell}
        .
    \end{equation*}
    Next,
    note that
    $f(x) = x^{1/\ell}$
    is concave over $x \geq 0$
    (since $\ell \geq 1$ for $n \geq 2$),
    so Jensen's inequality yields
    \begin{equation*}
        \bbE \left[
            \left(
                \sum_{i=1}^{n} |X_{i}|^{p\ell}
            \right)^{1/\ell}
        \right]
        \leq
        \left(
            \bbE \left[
                \sum_{i=1}^{n} |X_{i}|^{p\ell}
            \right]
        \right)^{1/\ell}
        =
        \left(
            \sum_{i=1}^{n} \bbE \Big[ |X_{i}|^{p\ell} \Big]
        \right)^{1/\ell}
        ,
    \end{equation*}
    and thus
    \begin{equation*}
        \left(
            \bbE \left|
                \max_{i=1,\dots,n} |X_{i}|
            \right|^{p}
        \right)^{1/p}
        \leq
        \left(
            \bbE \left[
                \left(
                    \sum_{i=1}^{n} |X_{i}|^{p\ell}
                \right)^{1/\ell}
            \right]
        \right)^{1/p}
        \leq
        \left(
            \sum_{i=1}^{n} \bbE \Big[ |X_{i}|^{p\ell} \Big]
        \right)^{1/(p\ell)}
        .
    \end{equation*}
    Next, note that
    \begin{equation*}
        \left(
            \sum_{i=1}^{n} \bbE \Big[ |X_{i}|^{p\ell} \Big]
        \right)^{1/(p\ell)}
        \leq
        \left(
            n
            \max_{i=1,\dots,n} \bbE \Big[ |X_{i}|^{p\ell} \Big]
        \right)^{1/(p\ell)}
        \leq
        2
        \max_{i=1,\dots,n} \left( \bbE \Big[ |X_{i}|^{p\ell} \Big] \right)^{1/(p\ell)}
        ,
    \end{equation*}
    since
    $n^{1/(p\ell)} = 2^{\log_2(n^{1/(p\ell)})} = 2^{(1/(p\ell))\log_2 n} = 2^{1/p} \leq 2$.

    Finally,
    since $p\ell \geq 1$ and each $X_{i}$ is sub-Gaussian,
    we have
    \begin{equation*}
        \left( \bbE \Big[ |X_{i}|^{p\ell} \Big] \right)^{1/(p\ell)}
        \leq
        \tlC \| X_{i} \|_{\psi_{2}} \sqrt{p\ell}
        =
        \tlC \| X_{i} \|_{\psi_{2}} \sqrt{p \log_2 n}
        =
        \frac{\tlC}{\sqrt{\log 2}} \| X_{i} \|_{\psi_{2}} \sqrt{\log n} \sqrt{p}
        ,
    \end{equation*}
    where $\tlC > 0$ is a universal constant,
    and so
    \begin{equation*}
        \left(
            \bbE \left|
                \max_{i=1,\dots,n} |X_{i}|
            \right|^{p}
        \right)^{1/p}
        \leq
        \left(
            \sum_{i=1}^{n} \bbE \Big[ |X_{i}|^{p\ell} \Big]
        \right)^{1/(p\ell)}
        \leq
        2
        \max_{i=1,\dots,n} \frac{\tlC}{\sqrt{\log 2}} \| X_{i} \|_{\psi_{2}} \sqrt{\log n} \sqrt{p}
        .
    \end{equation*}
    Rearranging and defining $C \coloneqq 2 \tlC / \sqrt{\log 2}$ concludes the proof.
\end{proof}

\begin{lemma}
    \label{lem:opnorm:subgaussian}
    Let $\bmX \in \bbR^{n \times n}$ be a symmetric random matrix
    whose entries on and above the diagonal are independent mean-zero sub-Gaussian
    random variables.
    Then, $\|\bmX\|_{\op}$ is also sub-Gaussian with
    \begin{equation*}
        \big\| \|\bmX\|_{\op} \big\|_{\psi_{2}}
        \leq
        C \left( \max_{i,j=1,\dots,n} \| X_{ij} \|_{\psi_{2}} \right) \sqrt{n}
        ,
    \end{equation*}
    where $C > 0$ is a universal constant.
\end{lemma}

The statement and proof of \cref{lem:opnorm:subgaussian} amount to a slight reformulation of \cite[Theorem 4.4.5 and Corollary 4.4.8]{vershynin2018hdp};
we provide a proof here for the reader's convenience.

\begin{proof}[Proof of \cref{lem:opnorm:subgaussian}]
    By the triangle inequality,
    \begin{equation*}
        \big\| \| \bmX \|_{\op} \big\|_{\psi_{2}}
        = \big\| \| \bmX^+ + \bmX^- \|_{\op} \big\|_{\psi_{2}}
        \leq \big\| \| \bmX^+ \|_{\op} + \| \bmX^- \|_{\op} \big\|_{\psi_{2}}
        \leq \big\| \| \bmX^+ \|_{\op} \big\|_{\psi_{2}} + \big\| \| \bmX^- \|_{\op} \big\|_{\psi_{2}}
        ,
    \end{equation*}
    where
    $\bmX^+ \in \bbR^{n \times n}$ is the upper-triangular part of $\bmX$
    and
    $\bmX^- \in \bbR^{n \times n}$ is the lower-triangular part of $\bmX$,
    i.e.,
    \begin{align*}
        X^+_{ij}
        &
        \coloneqq
        \begin{cases}
            X_{ij}  & \text{if } i \leq j , \\
            0       & \text{otherwise} ,
        \end{cases}
        &
        X^-_{ij}
        &
        \coloneqq
        \begin{cases}
            X_{ij}  & \text{if } i > j , \\
            0       & \text{otherwise} .
        \end{cases}
    \end{align*}
    The problem thus reduces to bounding
    $\big\| \| \bmX^+ \|_{\op} \big\|_{\psi_{2}}$
    and
    $\big\| \| \bmX^- \|_{\op} \big\|_{\psi_{2}}$.

    Let $\clN$ be an $\varepsilon$-net of the sphere $\bbS^{n-1}$
    with $\varepsilon = 1/4$ and cardinality $|\clN| \leq 9^n$.
    It follows that (see, e.g., \cite[Exercise 4.4.3]{vershynin2018hdp})
    \begin{equation*}
        \|\bmX^+\|_{\op}
        \leq
        2 \max_{\bmu, \bmv \in \clN} \bmu' \bmX^+ \bmv
        =
        2 \max_{\bmu, \bmv \in \clN} \sum_{i,j=1}^n X^+_{ij} u_i v_j
        .
    \end{equation*}
    Note that the entries of $\bmX^+$
    are independent sub-Gaussian random variables
    (since it only includes the upper-triangular entries),
    so
    \begin{align*}
        \left\| \sum_{i,j=1}^n X^+_{ij} u_i v_j \right\|_{\psi_2}^2
        &
        \leq
        \tlC_1 \sum_{i,j=1}^n \left\| X^+_{ij} u_i v_j \right\|_{\psi_2}^2
        =
        \tlC_1 \sum_{i,j=1}^n \left\| X^+_{ij} \right\|_{\psi_2}^2 u_i^2 v_j^2
        \\&
        \leq
        \tlC_1
        \max_{i,j = 1,\dots,n} \left\| X^+_{ij} \right\|_{\psi_2}^2
        \sum_{i,j=1}^n u_i^2 v_j^2
        =
        \tlC_1
        \max_{i,j = 1,\dots,n} \left\| X^+_{ij} \right\|_{\psi_2}^2
        \\&
        \leq
        \tlC_1
        \left(\max_{i,j = 1,\dots,n} \left\| X_{ij} \right\|_{\psi_2}^2\right)
        =
        \tlC_1
        \left(\max_{i,j = 1,\dots,n} \left\| X_{ij} \right\|_{\psi_2}\right)^2
        ,
    \end{align*}
    where $\tlC_1 > 0$ is a universal constant.

    Applying \cref{lem:max:subgaussian} then yields
    \begin{align*}
        \big\| \|\bmX^+\|_{\op} \big\|_{\psi_{2}}
        &
        \leq
        2
        \left\|
            \max_{\bmu, \bmv \in \clN} \sum_{i,j=1}^n X^+_{ij} u_i v_j
        \right\|_{\psi_{2}}
        \leq
        2
        \tlC_2
        \left(
            \max_{\bmu, \bmv \in \clN}
            \left\| \sum_{i,j=1}^n X^+_{ij} u_i v_j \right\|_{\psi_{2}}
        \right)
        \sqrt{\log(|\clN|^2)}
        \\&
        \leq
        2
        \tlC_2
        \left(
            \sqrt{\tlC_1}
            \left(\max_{i,j = 1,\dots,n} \left\| X_{ij} \right\|_{\psi_2}\right)
        \right)
        \sqrt{\log(|\clN|^2)}
        \\&
        =
        \frac{C}{2}
        \left(\max_{i,j = 1,\dots,n} \left\| X_{ij} \right\|_{\psi_2}\right)
        \sqrt{n}
        ,
    \end{align*}
    where $\tlC_2 > 0$ is the universal constant from \cref{lem:max:subgaussian},
    and $C/2$ collects all the universal constants.
    Repeating this calculation for $\bmX^-$
    and combining the results completes the proof.
\end{proof}

\section{Proof of \cref{thm:DCSBM:approx:L:signflip:decay}}
\label{proof:DCSBM:approx:L:signflip:decay}

We use a similar high-level progression of ideas as in \cite[Section 6.1]{ali2017improved} but develop bounds on the statistical moments instead of high probability bounds.
Specifically, the following sections of this \lcnamecref{proof:DCSBM:approx:L:signflip:decay} systematically characterize the terms in the \normadj matrix $\bmL_\alpha$, which lead us in the end to the desired characterization of $\bmL_\alpha$ itself.

Before we begin, we note (as in \cite[Equations~(7) and~(8)]{ali2017improved}) that the adjacency matrix can be written as
\begin{equation}
    \frac{1}{\sqrt{n}} \bmA
    =
    \frac{1}{\sqrt{n}}\bmq\bmq'
    +
    \frac{1}{n}
    \bmB
    +
    \frac{1}{\sqrt{n}}
    \bmX
    ,
\end{equation}
where $\bmX$ is a symmetric random matrix whose upper triangular entries are independent, mean zero, uniformly bounded random variables,
and we have defined $\bmB \coloneqq \bmD_{q} \bmJ \bmM \bmJ' \bmD_{q}$.
Similar to \cite[Equation~(9)]{ali2017improved}, we also decompose the degree vector $\bmd$ as
\begin{equation}
    \label{eq:expand:bmd}
    \bmd
    =
    \bmq \bmq' \bm1_{n}
    +
    \frac{1}{\sqrt{n}}\bmB\bm1_{n}
    +
    \bmX\bm1_{n}
    =
    \bmq'\bm1_{n}
    \left(
        \bmq
        +
        \frac{1}{\sqrt{n}}
        \frac{\bmB\bm1_{n}}{\bmq'\bm1_{n}}
        +
        \frac{\bmX\bm1_{n}}{\bmq'\bm1_{n}}
    \right)
    ,
\end{equation}
where we note that
\begin{align}
    \label{eq:bound_prelim_observation}
    \|\bmq\|_{\infty}
    &\asymp
    1
    ,
    &
    \left\|
        \frac{1}{\sqrt{n}}
        \frac{\bmB \bm1_{n}}{\bmq'\bm1_{n}}
    \right\|_{\infty}
    &\asymp
    \frac{1}{\sqrt{n}}
    ,
    &
    \big\| \|\bmX\bm1_{n}\|_{\infty} \big\|_{\psi_{2}}
    &\lesssim
    \sqrt{n \cdot \log n}
    .
\end{align}
The first expression in \cref{eq:bound_prelim_observation} holds since $0 < \qmin \le q_{i} \le \qmax < 1$ almost surely for all $1 \le i \le n$.
The second expression holds since simultaneously $\| \bmB \bm1_{n} \|_{\infty} \asymp n$ and $\bmq' \bm1_{n} \asymp n$.
The third expression holds since $\bmX$ is a matrix of independent (up to symmetry) mean zero random variables that are uniformly bounded, so the entries are each sub-Gaussian random variables with $\|X_{ij}\|_{\psi_{2}} \leq 1/\sqrt{\ln 2}$ by \cite[Example~2.5.8]{vershynin2018hdp}.
Consequently, each entry of the vector $\bmX \bm1_{n}$ is a sum of independent sub-Gaussian random variables, which is itself sub-Gaussian with $\|(\bmX\bm1_{n})_{i}\|_{\psi_{2}} \lesssim \sqrt{n}$ by \cite[Proposition~2.6.1]{vershynin2018hdp}, emphasizing that the implicit constant for the right-hand side does not depend on the index $i$.
Finally, a direct application of \cref{lem:max:subgaussian} yields the stated bound.

\subsection{Characterizing \texorpdfstring{$(\bmd'\bm1_{n})^{\alpha}$}{(d'1)\textasciicircum alpha}}
Similar to \cite[Equation~(10)]{ali2017improved},
we decompose $\bmd'\bm1_{n}$ as
\begin{equation}
    \bmd'\bm1_{n}
    =
    (\bmq'\bm1_{n})^2
    \left[
        1
        +
        \frac{1}{\sqrt{n}}
        \frac{\bm1_{n}'\bmB\bm1_{n}}{(\bmq'\bm1_{n})^2}
        +
        \frac{\bm1_{n}'\bmX\bm1_{n}}{(\bmq'\bm1_{n})^2}
    \right]
    ,
\end{equation}
where here we note that
\begin{align}
    \label{eq:bound_prelim_observation_2}
    \left|
        \frac{1}{\sqrt{n}}
        \frac{\bm1_{n}'\bmB \bm1_{n}}{(\bmq'\bm1_{n})^2}
    \right|
    &\lesssim n^{-1/2}
    , &
    \left\| \frac{\bm1_{n}'\bmX\bm1_{n}}{(\bmq'\bm1_{n})^2} \right\|_{\psi_{2}}
    &\lesssim n^{-1}
    .
\end{align}
The first expression in \cref{eq:bound_prelim_observation_2} holds since $|\bm1_{n}'\bmB \bm1_{n}| \asymp n^{2}$ and $|(\bmq'\bm1_{n})^2| \asymp n^{2}$.
The second expression relies on the further observation that $\| \bm1_{n}'\bmX\bm1_{n} \|_{\psi_{2}} \lesssim n$ by \cite[Proposition~2.5.2(ii) and Proposition~2.6.1]{vershynin2018hdp}.

\newcommand{\highlighted}[1]{{#1}}
\newcommand{\remone}{{\highlighted{r_{1}}}}
\newcommand{\lowone}{{\highlighted{\delta_{1}}}}

Similar to \cite[Equation~(11)]{ali2017improved}, we next apply the Taylor expansion
$f(x) = (1+x)^\alpha = 1 + \alpha x + O(|x|^2)$ centered at $x_0 = 0$
to obtain
\begin{align}
    \label{eq:jmlr:11}
    (\bmd'\bm1_{n})^{\alpha}
    &
    =
    (\bmq'\bm1_{n})^{2\alpha}
    \left(
        1
        +
        \left[
            \frac{1}{\sqrt{n}}
            \frac{\bm1_{n}'\bmB \bm1_{n}}{(\bmq'\bm1_{n})^2}
            +
            \frac{\bm1_{n}'\bmX\bm1_{n}}{(\bmq'\bm1_{n})^2}
        \right]
    \right)^{\alpha}
    \\& \nonumber
    =
    (\bmq'\bm1_{n})^{2\alpha}
    \left(
        1 +
        \alpha
        \left[
            \frac{1}{\sqrt{n}}
            \frac{\bm1_{n}'\bmB \bm1_{n}}{(\bmq'\bm1_{n})^2}
            +
            \frac{\bm1_{n}'\bmX\bm1_{n}}{(\bmq'\bm1_{n})^2}
        \right]
        +
        \remone
    \right),
\end{align}
where the $O(|x|^{2})$ remainder term $\remone$
can be written as
\begin{equation}
    \remone
    =
    \frac{\alpha(\alpha-1)}{2!} (1+c)^{\alpha-2} \cdot x^2
    ,
\end{equation}
where
\begin{equation*}
    x
    =
    \frac{1}{\sqrt{n}}
    \frac{\bm1_{n}'\bmB \bm1_{n}}{(\bmq'\bm1_{n})^2}
    +
    \frac{\bm1_{n}'\bmX\bm1_{n}}{(\bmq'\bm1_{n})^2}
    ,
\end{equation*}
and the (stochastic) term $c$ is between $0$ and $x$.

Bounding $\bbE\{|\remone|^{\mmt}\}$ requires us to address that $(1+c)^{\alpha-2}$ may diverge if $c$ approaches $-1$ depending on the value of $\alpha$.
To resolve this, we apply the law of total expectation to obtain:
\begin{equation}
    \label{eq:remone:loe}
    \bbE\{|\remone|^{\mmt}\}
    = \bbE\left\{ |\remone|^{\mmt} : \frac{|\bm1_{n}'\bmX\bm1_{n}|}{(\bmq'\bm1_{n})^2} \geq \epsilon \right\}
    \Pr\left( \frac{|\bm1_{n}'\bmX\bm1_{n}|}{(\bmq'\bm1_{n})^2} \geq \epsilon \right)
    + \bbE\left\{ |\remone|^{\mmt} : \frac{|\bm1_{n}'\bmX\bm1_{n}|}{(\bmq'\bm1_{n})^2} < \epsilon \right\}
    \Pr\left( \frac{|\bm1_{n}'\bmX\bm1_{n}|}{(\bmq'\bm1_{n})^2} < \epsilon \right)
    ,
\end{equation}
where $\epsilon \in (0,1)$ is arbitrary.
We now bound the terms in the above display equation.
\begin{enumerate}
    \item Observe first that
    \begin{equation*}
        |\remone|
        = \left| \frac{(\bmd'\bm1_{n})^{\alpha}}{(\bmq'\bm1_{n})^{2\alpha}} - 1 - \alpha x \right|
        \leq \left| \frac{(\bmd'\bm1_{n})^{\alpha}}{(\bmq'\bm1_{n})^{2\alpha}} \right| + | 1 - \alpha x |
        .
    \end{equation*}
    Now, consider the following case-by-case analysis:
    \begin{enumerate}
        \item $\bmd = \bm0_{n}$:
        Since we defined $0^{\alpha} = 0$ for all $\alpha$,
        in this case we have
        \begin{equation*}
            \left| \frac{(\bmd'\bm1_{n})^{\alpha}}{(\bmq'\bm1_{n})^{2\alpha}} \right|
            = \left| \frac{0}{(\bmq'\bm1_{n})^{2\alpha}} \right|
            = 0
            .
        \end{equation*}
        \item $\bmd \neq \bm0_{n}$:
        If $\alpha \geq 0$ then
        \begin{equation*}
            \left| \frac{(\bmd'\bm1_{n})^{\alpha}}{(\bmq'\bm1_{n})^{2\alpha}} \right|
            \leq \left| \frac{(n^2)^{\alpha}}{(n \qmin)^{2\alpha}} \right|
            = \qmin^{-2\alpha}
            \lesssim 1
            ,
        \end{equation*}
        and if $\alpha < 0$ then
        \begin{equation*}
            \left| \frac{(\bmd'\bm1_{n})^{\alpha}}{(\bmq'\bm1_{n})^{2\alpha}} \right|
            \leq \left| \frac{(1)^{\alpha}}{(n\qmax)^{2\alpha}} \right|
            = \qmax^{-2\alpha} n^{-2\alpha}
            \lesssim n^{-2\alpha}
            .
        \end{equation*}
    \end{enumerate}
    Moreover, we have that
    \begin{equation*}
        |x|
        \le
        \frac{1}{\sqrt{n}} \frac{|\bm1_{n}'\bmB \bm1_{n}|}{(\bmq'\bm1_{n})^2} + \frac{|\bm1_{n}'\bmX\bm1_{n}|}{(\bmq'\bm1_{n})^2}
        \lesssim
        \frac{1}{\sqrt{n}} + 1
        .
    \end{equation*}
    Thus, it holds surely that
    \begin{equation*}
        |\remone|
        \lesssim
        n^{\max(0,-2\alpha)}
        ,
    \end{equation*}
    and so
    \begin{equation}
        \label{eq:remone:loe:1}
        |\remone|^{\mmt}
        \leq
        C^{\mmt} \left(n^{\max(0,-2\alpha)}\right)^{\mmt}
        =
        C^{\mmt} n^{\max(0,-2\mmt\alpha)}
        ,
    \end{equation}
    where $C$ is a constant that depends only on the model parameters.

    \item It follows from \cref{eq:bound_prelim_observation_2} that
    there exists a constant $C > 0$
    so that provided $n$ is sufficiently large,
    \begin{equation}
        \label{eq:remone:loe:2}
        \Pr\left( \frac{|\bm1_{n}'\bmX\bm1_{n}|}{(\bmq'\bm1_{n})^2} \geq \epsilon \right)
        \leq 2 \exp\left(-\frac{\epsilon^2}{(C/n)^2}\right)
        =
        2 \exp(-C^{-2} \epsilon^2 n^2)
        .
    \end{equation}

    \item Observe that for $|\remone|^{\mmt}$,
    there exist constants $C, N > 0$ so that for $n \geq N$,
    \begin{align*}
        \bbE\left\{ |\remone|^{\mmt} : \frac{|\bm1_{n}'\bmX\bm1_{n}|}{(\bmq'\bm1_{n})^2} < \epsilon \right\}
        &
        =
        \bbE\left\{ \left|\frac{\alpha(\alpha-1)}{2!} (1+c)^{\alpha-2}\right|^{\mmt} \cdot |x^2|^{\mmt} : \frac{|\bm1_{n}'\bmX\bm1_{n}|}{(\bmq'\bm1_{n})^2} < \epsilon \right\}
        \\&
        \leq
        \bbE\left\{ C^{\mmt} \cdot |x^2|^{\mmt} : \frac{|\bm1_{n}'\bmX\bm1_{n}|}{(\bmq'\bm1_{n})^2} < \epsilon \right\}
        = C^{\mmt} \cdot \bbE\left\{ |x^2|^{\mmt} : \frac{|\bm1_{n}'\bmX\bm1_{n}|}{(\bmq'\bm1_{n})^2} < \epsilon \right\}
        ,
    \end{align*}
    and so
    \begin{align}
        \label{eq:remone:loe:3}
        &
        \bbE\left\{ |\remone|^{\mmt} : \frac{|\bm1_{n}'\bmX\bm1_{n}|}{(\bmq'\bm1_{n})^2} < \epsilon \right\}
        \Pr\left( \frac{|\bm1_{n}'\bmX\bm1_{n}|}{(\bmq'\bm1_{n})^2} < \epsilon \right)
        \\ \nonumber
        &\qquad
        \leq
        C^{\mmt}
        \cdot
        \bbE\left\{ |x^2|^{\mmt} : \frac{|\bm1_{n}'\bmX\bm1_{n}|}{(\bmq'\bm1_{n})^2} < \epsilon \right\}
        \Pr\left( \frac{|\bm1_{n}'\bmX\bm1_{n}|}{(\bmq'\bm1_{n})^2} < \epsilon \right)
        \\ \nonumber
        &\qquad
        =
        C^{\mmt}
        \cdot
        \bigg[ \bbE |x^2|^{\mmt} - \bbE\left\{ |x^2|^{\mmt} : \frac{|\bm1_{n}'\bmX\bm1_{n}|}{(\bmq'\bm1_{n})^2} \geq \epsilon \right\}
        \Pr\left( \frac{|\bm1_{n}'\bmX\bm1_{n}|}{(\bmq'\bm1_{n})^2} \geq \epsilon \right) \bigg]
        \\ \nonumber
        &\qquad
        \leq
        C^{\mmt}
        \cdot
        \bbE |x^2|^{\mmt}
        =
        C^{\mmt}
        \cdot
        \bbE |x|^{2\mmt}
        \\ \nonumber
        &\qquad
        \leq
        C^{\mmt}
        2^{2\mmt-1}
        \left[
            \bbE \left|
                \frac{1}{\sqrt{n}}
                \frac{\bm1_{n}'\bmB \bm1_{n}}{(\bmq'\bm1_{n})^2}
            \right|^{2\mmt}
            +
            \bbE \left|
                \frac{\bm1_{n}'\bmX\bm1_{n}}{(\bmq'\bm1_{n})^2}
            \right|^{2\mmt}
        \right]
        \\ \nonumber
        &\qquad
        \leq
        \frac{C_1^{\mmt}}{n^{\mmt}} + \frac{C_2^{\mmt} \mmt^{\mmt}}{n^{2\mmt}}
        \\ \nonumber
        &
        \qquad
        \leq
        \frac{C_3^{\mmt}\mmt^{\mmt}}{n^{\mmt}}
        .
    \end{align}
\end{enumerate}
Combining \cref{eq:remone:loe,eq:remone:loe:1,eq:remone:loe:2,eq:remone:loe:3} yields
that there exist constants $C_1, C_2, C_3, N > 0$ so that for any $n \geq N$ and any $\mmt \geq 1$
\begin{equation*}
    \bbE\{|\remone|^{\mmt}\}
    \leq
    C_1^{\mmt} n^{\max(0,-2\mmt\alpha)}
    \exp(-C_2^{-1} \epsilon^2 n^2)
    + \frac{C_3^{\mmt}\mmt^{\mmt}}{n^{\mmt}}
    ,
\end{equation*}
and so we have after simplifying that for any $\mmt \geq 1$
\begin{equation}
    \ellP{\remone}
    \lesssim
    n^{-1}
    .
\end{equation}
Summarizing, we have
\begin{equation}
    \label{eq:decomp:d:one}
    (\bmd'\bm1_{n})^{\alpha}
    =
    (\bmq'\bm1_{n})^{2\alpha}
    \left(
        1 + \lowone
    \right)
    ,
\end{equation}
where
\begin{equation}
    \lowone
    \coloneqq
    \alpha
    \left[
        \frac{1}{\sqrt{n}}
        \frac{\bm1_{n}'\bmB \bm1_{n}}{(\bmq'\bm1_{n})^2}
        +
        \frac{\bm1_{n}'\bmX\bm1_{n}}{(\bmq'\bm1_{n})^2}
    \right]
    +
    \remone
    ,
\end{equation}
and the fact that $\bm1_{n}'\bmX\bm1_{n}$ is sub-Gaussian with $\|\bm1_{n}'\bmX\bm1_{n}\|_{\psi_{2}} \lesssim n$ yields $\ellP{\bm1_{n}'\bmX\bm1_{n}} \lesssim n$, so it follows from our above analysis that for any $\mmt \geq 1$,
\begin{align}
    \ellP{\lowone}
    &
    =
    \ellP{
        \frac{\alpha}{\sqrt{n}}\frac{\bm1_{n}'\bmB \bm1_{n}}{(\bmq'\bm1_{n})^2}
        + \alpha\frac{\bm1_{n}'\bmX\bm1_{n}}{(\bmq'\bm1_{n})^2}
        + \remone
    }
    \\& \nonumber
    \leq
    \ellP{
        \frac{\alpha}{\sqrt{n}}\frac{\bm1_{n}'\bmB \bm1_{n}}{(\bmq'\bm1_{n})^2}
    }
    + \ellP{\alpha\frac{\bm1_{n}'\bmX\bm1_{n}}{(\bmq'\bm1_{n})^2}}
    + \ellP{\remone}
    \\& \nonumber
    \lesssim
    n^{-1/2}
    + n^{-1}
    + n^{-1}
    \\& \nonumber
    \lesssim
    n^{-1/2}
    .
\end{align}
Hence, we have obtained the desired characterization of $(\bmd'\bm1_{n})^{-\alpha}$ that will be used in the remainder of the proof.

\subsection{Characterizing \texorpdfstring{$\frac{1}{\sqrt{n}} \frac{\bmd\bmd'}{\bmd'\bm1_{n}}$}{1/sqrt(n) dd'/(d'1)}}
\newcommand{\remthree}{{\highlighted{\bmDelta_{3}}}}
Substituting $\alpha = -1$ into \cref{eq:jmlr:11} yields
\begin{equation*}
    (\bmd'\bm1_{n})^{-1}
    =
    (\bmq'\bm1_{n})^{-2}
    \left[
        1 - \frac{1}{\sqrt{n}} \frac{\bm1_{n}'\bmB\bm1_{n}}{(\bmq'\bm1_{n})^2} - \frac{\bm1_{n}'\bmX\bm1_{n}}{(\bmq'\bm1_{n})^2} + \remone
    \right]
    ,
\end{equation*}
and so we have the following analogue of \cite[Equation~(13)]{ali2017improved}, namely
\begin{align}
    \frac{1}{\sqrt{n}} \frac{\bmd\bmd'}{\bmd'\bm1_{n}}
    &=
    \frac{\bmq\bmq'}{\sqrt{n}}
    + \frac{1}{n} \frac{\bmq\bm1_{n}'\bmB}{\bmq'\bm1_{n}}
    + \frac{1}{n} \frac{\bmB\bm1_{n}\bmq'}{\bmq'\bm1_{n}}
    + \frac{1}{\sqrt{n}} \frac{\bmq\bm1_{n}'\bmX}{\bmq'\bm1_{n}}
    + \frac{1}{\sqrt{n}} \frac{\bmX\bm1_{n}\bmq'}{\bmq'\bm1_{n}}
    \\&\quad\nonumber
    - \frac{1}{n} \frac{\bm1_{n}'\bmB\bm1_{n}}{(\bmq'\bm1_{n})^2} \bmq\bmq'
    - \frac{1}{\sqrt{n}} \frac{\bm1_{n}'\bmX\bm1_{n}}{(\bmq'\bm1_{n})^2} \bmq\bmq'
    + \remthree
    ,
\end{align}
where the residual term $\remthree$ is explicitly given by
\begin{align}
    \remthree
    &\coloneqq
    \frac{1}{\sqrt{n}} \frac{1}{n} \frac{\bmB\bm1_{n}\bm1_{n}'\bmB}{(\bmq'\bm1_{n})^2}
    + \frac{1}{n} \frac{\bmB\bm1_{n}\bm1_{n}'\bmX}{(\bmq'\bm1_{n})^2}
    \\&\qquad\nonumber
    + \frac{1}{n} \frac{\bmX\bm1_{n}\bm1_{n}'\bmB}{(\bmq'\bm1_{n})^2}
    + \frac{1}{\sqrt{n}} \frac{\bmX\bm1_{n}\bm1_{n}'\bmX}{(\bmq'\bm1_{n})^2}
    - \frac{1}{n} \frac{\bm1_{n}'\bmB\bm1_{n}}{(\bmq'\bm1_{n})^2} \frac{1}{\sqrt{n}} \frac{\bmq\bm1_{n}'\bmB}{\bmq'\bm1_{n}}
    \\&\qquad\nonumber
    - \frac{1}{n} \frac{\bm1_{n}'\bmB\bm1_{n}}{(\bmq'\bm1_{n})^2} \frac{1}{\sqrt{n}} \frac{\bmB\bm1_{n}\bmq'}{\bmq'\bm1_{n}}
    - \frac{1}{n} \frac{\bm1_{n}'\bmB\bm1_{n}}{(\bmq'\bm1_{n})^2} \frac{\bmq\bm1_{n}'\bmX}{\bmq'\bm1_{n}}
    \\&\qquad\nonumber
    - \frac{1}{n} \frac{\bm1_{n}'\bmB\bm1_{n}}{(\bmq'\bm1_{n})^2} \frac{\bmX\bm1_{n}\bmq'}{\bmq'\bm1_{n}}
    - \frac{1}{n} \frac{\bm1_{n}'\bmB\bm1_{n}}{(\bmq'\bm1_{n})^2} \frac{1}{n} \frac{\bmB\bm1_{n}\bm1_{n}'\bmB}{(\bmq'\bm1_{n})^2}
    \\&\qquad\nonumber
    - \frac{1}{n} \frac{\bm1_{n}'\bmB\bm1_{n}}{(\bmq'\bm1_{n})^2} \frac{1}{\sqrt{n}} \frac{\bmB\bm1_{n}\bm1_{n}'\bmX}{(\bmq'\bm1_{n})^2}
    \\&\qquad\nonumber
    - \frac{1}{n} \frac{\bm1_{n}'\bmB\bm1_{n}}{(\bmq'\bm1_{n})^2} \frac{1}{\sqrt{n}} \frac{\bmX\bm1_{n}\bm1_{n}'\bmB}{(\bmq'\bm1_{n})^2}
    - \frac{1}{n} \frac{\bm1_{n}'\bmB\bm1_{n}}{(\bmq'\bm1_{n})^2} \frac{\bmX\bm1_{n}\bm1_{n}'\bmX}{(\bmq'\bm1_{n})^2}
    \\&\qquad\nonumber
    - \frac{1}{\sqrt{n}} \frac{\bm1_{n}'\bmX\bm1_{n}}{(\bmq'\bm1_{n})^2} \frac{1}{\sqrt{n}} \frac{\bmq\bm1_{n}'\bmB}{\bmq'\bm1_{n}}
    - \frac{1}{\sqrt{n}} \frac{\bm1_{n}'\bmX\bm1_{n}}{(\bmq'\bm1_{n})^2} \frac{1}{\sqrt{n}} \frac{\bmB\bm1_{n}\bmq'}{\bmq'\bm1_{n}}
    \\&\qquad\nonumber
    - \frac{1}{\sqrt{n}} \frac{\bm1_{n}'\bmX\bm1_{n}}{(\bmq'\bm1_{n})^2} \frac{\bmq\bm1_{n}'\bmX}{\bmq'\bm1_{n}}
    - \frac{1}{\sqrt{n}} \frac{\bm1_{n}'\bmX\bm1_{n}}{(\bmq'\bm1_{n})^2} \frac{\bmX\bm1_{n}\bmq'}{\bmq'\bm1_{n}}
    \\&\qquad\nonumber
    - \frac{1}{\sqrt{n}} \frac{\bm1_{n}'\bmX\bm1_{n}}{(\bmq'\bm1_{n})^2} \frac{1}{n} \frac{\bmB\bm1_{n}\bm1_{n}'\bmB}{(\bmq'\bm1_{n})^2}
    - \frac{1}{\sqrt{n}} \frac{\bm1_{n}'\bmX\bm1_{n}}{(\bmq'\bm1_{n})^2} \frac{1}{\sqrt{n}} \frac{\bmB\bm1_{n}\bm1_{n}'\bmX}{(\bmq'\bm1_{n})^2}
    \\&\qquad\nonumber
    - \frac{1}{\sqrt{n}} \frac{\bm1_{n}'\bmX\bm1_{n}}{(\bmq'\bm1_{n})^2} \frac{1}{\sqrt{n}} \frac{\bmX\bm1_{n}\bm1_{n}'\bmB}{(\bmq'\bm1_{n})^2}
    - \frac{1}{\sqrt{n}} \frac{\bm1_{n}'\bmX\bm1_{n}}{(\bmq'\bm1_{n})^2} \frac{\bmX\bm1_{n}\bm1_{n}'\bmX}{(\bmq'\bm1_{n})^2}
    + \frac{1}{\sqrt{n}} \remone \bmq\bmq'
    \\&\qquad\nonumber
    + \frac{1}{\sqrt{n}} \remone \frac{1}{\sqrt{n}} \frac{\bmq\bm1_{n}'\bmB}{\bmq'\bm1_{n}}
    + \frac{1}{\sqrt{n}} \remone \frac{1}{\sqrt{n}} \frac{\bmB\bm1_{n}\bmq'}{\bmq'\bm1_{n}}
    + \frac{1}{\sqrt{n}} \remone \frac{\bmq\bm1_{n}'\bmX}{\bmq'\bm1_{n}}
    \\&\qquad\nonumber
    + \frac{1}{\sqrt{n}} \remone \frac{\bmX\bm1_{n}\bmq'}{\bmq'\bm1_{n}}
    + \frac{1}{\sqrt{n}} \remone \frac{1}{n} \frac{\bmB\bm1_{n}\bm1_{n}'\bmB}{(\bmq'\bm1_{n})^2}
    \\&\qquad\nonumber
    + \frac{1}{\sqrt{n}} \remone \frac{1}{\sqrt{n}} \frac{\bmB\bm1_{n}\bm1_{n}'\bmX}{(\bmq'\bm1_{n})^2}
    + \frac{1}{\sqrt{n}} \remone \frac{1}{\sqrt{n}} \frac{\bmX\bm1_{n}\bm1_{n}'\bmB}{(\bmq'\bm1_{n})^2}
    + \frac{1}{\sqrt{n}} \remone \frac{\bmX\bm1_{n}\bm1_{n}'\bmX}{(\bmq'\bm1_{n})^2}
    .
\end{align}

We now wish to bound $\bbE\|\remthree\|_{\op}^{\mmt}$ for any positive integer $\mmt \ge 1$. We shall make use of the following preliminary bounds.

\begin{proposition}
    \label{prop:prelim:bounds}
    The following bounds hold deterministically:
    \begin{align}
        \label{eq:bound:2}
        \|\bmB\|_{\op} &\lesssim n
        , \\
        \label{eq:bound:3}
        \|\bmq \bm1_{n}'\|_{\op} &\lesssim n
        , \\
        \label{eq:bound:3b}
        |\bmq' \bm1_{n}|^{-1} &\lesssim n^{-1}
        , \\
        \label{eq:bound:3c}
        \|\bmq \bmq'\|_{\op} &\lesssim n
        , \\
        \label{eq:bound:1}
        \ellP{\opnorm{\bmX}} &\lesssim n^{1/2}
        , \\
        \label{eq:bound:1b}
        \ellP{\bm1_{n}'\bmX\bm1_{n}} &\lesssim n
        , \\
        \label{eq:bound:4}
        \ellP{\|\bmX \bm1_{n} \bm1_{n}' \bmX\|_{\op}} &\lesssim n^{2}
        .
    \end{align}
    The following bounds hold with probability one:
    \begin{align}
        \label{eq:bound:5}
        \|\bmX\|_{\op} &\lesssim n
        , \\
        \label{eq:bound:6}
        |\bm1_{n}' \bmX \bm1_{n}| &\lesssim n^{2}
        .
    \end{align}
\end{proposition}

\begin{proof}[Proof of \cref{prop:prelim:bounds}]
    \Cref{eq:bound:2} follows by bounding the operator norm by the maximum absolute row sum
    and noting that $\Mmax$ and $\qmax$ are constants.
    \Cref{eq:bound:3,eq:bound:3b,,eq:bound:3c} hold by noting that $\qmax$ and $\qmin$ are constant.
    \Cref{eq:bound:1} follows from \cref{lem:opnorm:subgaussian}.
    \Cref{eq:bound:1b} follows from \cite[Proposition~2.5.2(ii) and Proposition~2.6.1]{vershynin2018hdp}.
    \Cref{eq:bound:4} holds since
    \begin{align*}
        \ellP{\|\bmX \bm1_{n} \bm1_{n}' \bmX\|_{\op}}
        &
        \leq \ellP{\|\bmX\|_{\op} \|\bm1_{n} \bm1_{n}'\|_{\op} \|\bmX\|_{\op}}
        = n \ellP{\|\bmX\|_{\op}^2}
        \\& \nonumber
        = n \left[ \bbE \|\bmX\|_{\op}^{2\mmt} \right]^{1/\mmt}
        = n \left\{ \left[ \bbE \|\bmX\|_{\op}^{2\mmt} \right]^{1/(2\mmt)} \right\}^2
        \lesssim n \left\{ n^{1/2} \right\}^2
        = n^{2}
        ,
    \end{align*}
    where the final inequality follows from \cref{eq:bound:1}.
    \Cref{eq:bound:5} holds due to the general inequality $\|\bmX\|_{\op} \le \sqrt{\|\bmX\|_{1} \|\bmX\|_{\infty}}$ and since the entries of $\bmX$ are bounded with probability one.
    \Cref{eq:bound:6} holds since
    \begin{equation*}
        |\bm1_{n}' \bmX \bm1_{n}|
        =
        \Big| \sum_{ij} X_{ij} \Big|
        \leq
        \sum_{ij} |X_{ij}|
        \lesssim
        n^{2}
        ,
    \end{equation*}
    where the final inequality follows by noting that
    the entries of $\bmX$ are bounded with probability one.
\end{proof}

Using \cref{prop:prelim:bounds}, we now obtain bounds for all the terms appearing in the definition of $\remthree$ as follows.
Beginning with the first term, we have that for any positive integer $\mmt \ge 1$,
\begin{align*}
    \ellP{ \left\| \frac{1}{\sqrt{n}} \frac{1}{n} \frac{\bmB(\bm1_{n}\bm1_{n}')\bmB}{(\bmq'\bm1_{n})^2} \right\|_{\op} }
    &
    =
    \frac{1}{\sqrt{n}}
    \frac{1}{n}
    \frac{1}{(\bmq'\bm1_{n})^2}
    \ellP{\left\|\bmB(\bm1_{n}\bm1_{n}')\bmB\right\|_{\op}}
    \\&
    =
    \frac{1}{\sqrt{n}}
    \frac{1}{n}
    \frac{1}{(\bmq'\bm1_{n})^2}
    \left\|\bmB(\bm1_{n}\bm1_{n}')\bmB\right\|_{\op}
    \\&
    \leq
    \frac{1}{\sqrt{n}}
    \frac{1}{n}
    \frac{1}{(\bmq'\bm1_{n})^2}
    \left\|
    \bmB
    \right\|_{\op}
    \left\|
    \bm1_{n}\bm1_{n}'
    \right\|_{\op}
    \left\|
    \bmB
    \right\|_{\op}
    \\&
    \lesssim
    n^{-1/2} \cdot n^{-1} \cdot n^{-2} \cdot n \cdot n \cdot n
    \\&
    =
    n^{-1/2}
    ,
\end{align*}
where
the first line follows by homogeneity of $\opnorm{\cdot}$ and $\ellP{\cdot}$,
the second line follows since the terms are deterministic,
the third line follows by submultiplicativity of $\opnorm{\cdot}$,
the fourth line follows by \cref{prop:prelim:bounds},
and the final line simplifies the expression.
Applying similar arguments to the second term, we obtain
\begin{align*}
    &\ellP{ \left\| \frac{1}{n} \frac{\bmB(\bm1_{n}\bm1_{n}')\bmX}{(\bmq'\bm1_{n})^2} \right\|_{\op} } \\
    &\qquad
    \leq
    \frac{1}{n}
    \frac{1}{(\bmq'\bm1_{n})^2}
    \|\bmB\|_{\op}
    \|\bm1_{n}\bm1_{n}'\|_{\op}
    \ellP{ \|\bmX\|_{\op} } \\
    &\qquad
    \lesssim
    n^{-1} \cdot n^{-2} \cdot n \cdot n \cdot n^{1/2} \\
    &\qquad\lesssim
    n^{-1/2}
    ,
\end{align*}
and likewise for the remaining terms (where we have added parentheses to show the appropriate groupings):
\begin{align*}
    \ellP{ \left\| \frac{1}{n} \frac{\bmX(\bm1_{n}\bm1_{n}')\bmB}{(\bmq'\bm1_{n})^2} \right\|_{\op} } &\lesssim n^{-1/2}
    , \\
    \ellP{ \left\| \frac{1}{\sqrt{n}} \frac{(\bmX\bm1_{n}\bm1_{n}'\bmX)}{(\bmq'\bm1_{n})^2} \right\|_{\op} } &\lesssim n^{-1/2}
    , \\
    \ellP{ \left\| \frac{1}{n} \frac{\bm1_{n}'\bmB\bm1_{n}}{(\bmq'\bm1_{n})^2} \frac{1}{\sqrt{n}} \frac{(\bmq\bm1_{n}')\bmB}{\bmq'\bm1_{n}} \right\|_{\op} } &\lesssim n^{-1/2}
    , \\
    \ellP{ \left\| \frac{1}{n} \frac{\bm1_{n}'\bmB\bm1_{n}}{(\bmq'\bm1_{n})^2} \frac{1}{\sqrt{n}} \frac{\bmB(\bm1_{n}\bmq')}{\bmq'\bm1_{n}} \right\|_{\op} } &\lesssim n^{-1/2}
    , \\
    \ellP{ \left\| \frac{1}{n} \frac{\bm1_{n}'\bmB\bm1_{n}}{(\bmq'\bm1_{n})^2} \frac{(\bmq\bm1_{n}')\bmX}{\bmq'\bm1_{n}} \right\|_{\op} } &\lesssim n^{-1/2}
    , \\
    \ellP{ \left\| \frac{1}{n} \frac{\bm1_{n}'\bmB\bm1_{n}}{(\bmq'\bm1_{n})^2} \frac{\bmX(\bm1_{n}\bmq')}{\bmq'\bm1_{n}} \right\|_{\op} } &\lesssim n^{-1/2}
    , \\
    \ellP{ \left\| \frac{1}{n} \frac{\bm1_{n}'\bmB\bm1_{n}}{(\bmq'\bm1_{n})^2} \frac{1}{n} \frac{\bmB(\bm1_{n}\bm1_{n}')\bmB}{(\bmq'\bm1_{n})^2} \right\|_{\op} } &\lesssim n^{-1}
    , \\
    \ellP{ \left\| \frac{1}{n} \frac{\bm1_{n}'\bmB\bm1_{n}}{(\bmq'\bm1_{n})^2} \frac{1}{\sqrt{n}} \frac{\bmB(\bm1_{n}\bm1_{n}')\bmX}{(\bmq'\bm1_{n})^2} \right\|_{\op} } &\lesssim n^{-1}
    , \\
    \ellP{ \left\| \frac{1}{n} \frac{\bm1_{n}'\bmB\bm1_{n}}{(\bmq'\bm1_{n})^2} \frac{1}{\sqrt{n}} \frac{\bmX(\bm1_{n}\bm1_{n}')\bmB}{(\bmq'\bm1_{n})^2} \right\|_{\op} } &\lesssim n^{-1}
    , \\
    \ellP{ \left\| \frac{1}{n} \frac{\bm1_{n}'\bmB\bm1_{n}}{(\bmq'\bm1_{n})^2} \frac{(\bmX\bm1_{n}\bm1_{n}'\bmX)}{(\bmq'\bm1_{n})^2} \right\|_{\op} } &\lesssim n^{-1}
    , \\
    \ellP{ \left\| \frac{1}{\sqrt{n}} \frac{(\bm1_{n}'\bmX\bm1_{n})}{(\bmq'\bm1_{n})^2} \frac{1}{\sqrt{n}} \frac{(\bmq\bm1_{n}')\bmB}{\bmq'\bm1_{n}} \right\|_{\op} } &\lesssim n^{-1}
    , \\
    \ellP{ \left\| \frac{1}{\sqrt{n}} \frac{(\bm1_{n}'\bmX\bm1_{n})}{(\bmq'\bm1_{n})^2} \frac{1}{\sqrt{n}} \frac{\bmB(\bm1_{n}\bmq')}{\bmq'\bm1_{n}} \right\|_{\op} } &\lesssim n^{-1}
    , \\
    \ellP{ \left\| \frac{1}{\sqrt{n}} \frac{(\bm1_{n}'\bmX\bm1_{n})}{(\bmq'\bm1_{n})^2} \frac{(\bmq\bm1_{n}')\bmX}{\bmq'\bm1_{n}} \right\|_{\op} } &\lesssim n^{-1/2}
    , \\
    \ellP{ \left\| \frac{1}{\sqrt{n}} \frac{(\bm1_{n}'\bmX\bm1_{n})}{(\bmq'\bm1_{n})^2} \frac{\bmX(\bm1_{n}\bmq')}{\bmq'\bm1_{n}} \right\|_{\op} } &\lesssim n^{-1/2}
    , \\
    \ellP{ \left\| \frac{1}{\sqrt{n}} \frac{(\bm1_{n}'\bmX\bm1_{n})}{(\bmq'\bm1_{n})^2} \frac{1}{n} \frac{\bmB(\bm1_{n}\bm1_{n}')\bmB}{(\bmq'\bm1_{n})^2} \right\|_{\op} } &\lesssim n^{-3/2}
    , \\
    \ellP{ \left\| \frac{1}{\sqrt{n}} \frac{(\bm1_{n}'\bmX\bm1_{n})}{(\bmq'\bm1_{n})^2} \frac{1}{\sqrt{n}} \frac{\bmB(\bm1_{n}\bm1_{n}')\bmX}{(\bmq'\bm1_{n})^2} \right\|_{\op} } &\lesssim n^{-1}
    , \\
    \ellP{ \left\| \frac{1}{\sqrt{n}} \frac{(\bm1_{n}'\bmX\bm1_{n})}{(\bmq'\bm1_{n})^2} \frac{1}{\sqrt{n}} \frac{\bmX(\bm1_{n}\bm1_{n}')\bmB}{(\bmq'\bm1_{n})^2} \right\|_{\op} } &\lesssim n^{-1}
    , \\
    \ellP{ \left\| \frac{1}{\sqrt{n}} \frac{(\bm1_{n}'\bmX\bm1_{n})}{(\bmq'\bm1_{n})^2} \frac{(\bmX\bm1_{n}\bm1_{n}'\bmX)}{(\bmq'\bm1_{n})^2} \right\|_{\op} } &\lesssim n^{-1/2}
    , \\
    \ellP{ \left\| \frac{1}{\sqrt{n}} \remone (\bmq\bmq') \right\|_{\op} } &\lesssim \ellP{\remone} n^{1/2}
    , \\
    \ellP{ \left\| \frac{1}{\sqrt{n}} \remone \frac{1}{\sqrt{n}} \frac{(\bmq\bm1_{n}')\bmB}{\bmq'\bm1_{n}} \right\|_{\op} } &\lesssim \ellP{\remone}
    , \\
    \ellP{ \left\| \frac{1}{\sqrt{n}} \remone \frac{1}{\sqrt{n}} \frac{\bmB(\bm1_{n}\bmq')}{\bmq'\bm1_{n}} \right\|_{\op} } &\lesssim \ellP{\remone}
    , \\
    \ellP{ \left\| \frac{1}{\sqrt{n}} \remone \frac{(\bmq\bm1_{n}')\bmX}{\bmq'\bm1_{n}} \right\|_{\op} } &\lesssim \ellP{\remone} n^{1/2}
    , \\
    \ellP{ \left\| \frac{1}{\sqrt{n}} \remone \frac{\bmX(\bm1_{n}\bmq')}{\bmq'\bm1_{n}} \right\|_{\op} } &\lesssim \ellP{\remone} n^{1/2}
    , \\
    \ellP{ \left\| \frac{1}{\sqrt{n}} \remone \frac{1}{n} \frac{\bmB(\bm1_{n}\bm1_{n}')\bmB}{(\bmq'\bm1_{n})^2} \right\|_{\op} } &\lesssim \ellP{\remone} n^{-1/2}
    , \\
    \ellP{ \left\| \frac{1}{\sqrt{n}} \remone \frac{1}{\sqrt{n}} \frac{\bmB(\bm1_{n}\bm1_{n}')\bmX}{(\bmq'\bm1_{n})^2} \right\|_{\op} } &\lesssim \ellP{\remone}
    , \\
    \ellP{ \left\| \frac{1}{\sqrt{n}} \remone \frac{1}{\sqrt{n}} \frac{\bmX(\bm1_{n}\bm1_{n}')\bmB}{(\bmq'\bm1_{n})^2} \right\|_{\op} } &\lesssim \ellP{\remone}
    , \\
    \ellP{ \left\| \frac{1}{\sqrt{n}} \remone \frac{\bmX(\bm1_{n}\bm1_{n}')\bmX}{(\bmq'\bm1_{n})^2} \right\|_{\op} } &\lesssim \ellP{\remone} n^{1/2}
    .
\end{align*}
By putting all the pieces together, we have that for any positive integer $\mmt \ge 1$,
\begin{equation}
    \label{eq:bound_remthree}
    \ellP{\opnorm{\remthree}}
    \lesssim
    n^{-1/2} + \ellP{\remone} n^{1/2}
    \lesssim
    n^{-1/2}.
\end{equation}

\subsection{Characterizing \texorpdfstring{$\frac{1}{\sqrt{n}} \left( \bmA - \frac{\bmd\bmd'}{\bmd'\bm1_{n}} \right)$}{1/sqrt(n) (A - dd'/(d'1))}}
Similar to \cite[Equation~(14)]{ali2017improved}, we have
\begin{align}
    \label{eq:jmlr:14}
    \frac{1}{\sqrt{n}} \left( \bmA - \frac{\bmd\bmd'}{\bmd'\bm1_{n}} \right)
    &=
    \frac{1}{n} \bmB
    - \frac{1}{n} \frac{\bmq\bm1_{n}'\bmB}{\bmq'\bm1_{n}}
    - \frac{1}{n} \frac{\bmB\bm1_{n}\bmq'}{\bmq'\bm1_{n}}
    \\&\qquad\nonumber
    + \frac{1}{n} \frac{\bm1_{n}'\bmB\bm1_{n}}{(\bmq'\bm1_{n})^2} \bmq\bmq'
    + \frac{\bmX}{\sqrt{n}}
    - \frac{1}{\sqrt{n}} \frac{\bmq\bm1_{n}'\bmX}{\bmq'\bm1_{n}}
    - \frac{1}{\sqrt{n}} \frac{\bmX\bm1_{n}\bmq'}{\bmq'\bm1_{n}}
    \\&\qquad\nonumber
    + \frac{1}{\sqrt{n}} \frac{\bm1_{n}'\bmX\bm1_{n}}{(\bmq'\bm1_{n})^2} \bmq\bmq'
    - \remthree
    .
\end{align}
The terms in this expansion are individually bounded in the following manner:
\begin{align*}
    \ellP{ \opnorm{ \frac{1}{n} \bmB } }
    &\lesssim 1
    , \\
    \ellP{ \opnorm{ \frac{1}{n} \frac{(\bmq\bm1_{n}')\bmB}{\bmq'\bm1_{n}} } }
    &\lesssim 1
    , \\
    \ellP{ \opnorm{ \frac{1}{n} \frac{\bmB(\bm1_{n}\bmq')}{\bmq'\bm1_{n}} } }
    &\lesssim 1
    , \\
    \ellP{ \opnorm{ \frac{1}{n} \frac{\bm1_{n}'\bmB\bm1_{n}}{(\bmq'\bm1_{n})^2} (\bmq\bmq') } }
    &\lesssim 1
    , \\
    \ellP{ \opnorm{ \frac{\bmX}{\sqrt{n}} } }
    &\lesssim 1
    , \\
    \ellP{ \opnorm{ \frac{1}{\sqrt{n}} \frac{(\bmq\bm1_{n}')\bmX}{\bmq'\bm1_{n}} } }
    &\lesssim 1
    , \\
    \ellP{ \opnorm{ \frac{1}{\sqrt{n}} \frac{\bmX(\bm1_{n}\bmq')}{\bmq'\bm1_{n}} } }
    &\lesssim 1
    , \\
    \ellP{ \opnorm{ \frac{1}{\sqrt{n}} \frac{\bm1_{n}'\bmX\bm1_{n}}{(\bmq'\bm1_{n})^2} (\bmq\bmq') } }
    &\lesssim n^{-1/2}
    , \\
    \ellP{ \opnorm{ \remthree } }
    &\lesssim n^{-1/2}
    ,
\end{align*}
where the final bound on $\ellP{ \opnorm{ \remthree } }$ is from \cref{eq:bound_remthree}. Hence, it follows from an application of the triangle inequality that
\begin{align}
    \ellP{ \opnorm{ \frac{1}{\sqrt{n}} \left( \bmA - \frac{\bmd\bmd'}{\bmd'\bm1_{n}} \right) } }
    \lesssim 1
    .
\end{align}
This gives an expansion of $\frac{1}{\sqrt{n}} \left( \bmA - \frac{\bmd\bmd'}{\bmd'\bm1_{n}} \right)$ with a useful bound as desired.

\newcommand{\remtwo}{{\highlighted{\clD(\bmr_{2})}}}
\newcommand{\lowtwo}{{\highlighted{\clD(\bmdelta_{2})}}}

\subsection{Characterizing \texorpdfstring{$\bmD^{-\alpha}$}{D\textasciicircum(-alpha)}}
Recall from \cref{eq:expand:bmd} that
\begin{equation*}
    \bmd
    =
    \bmq \bmq' \bm1_{n} + \frac{1}{\sqrt{n}}\bmB\bm1_{n} + \bmX\bm1_{n}
    =
    \bmq'\bm1_{n}\left( \bmq + \frac{1}{\sqrt{n}}\frac{\bmB\bm1_{q}}{\bmq'\bm1_{n}} + \frac{\bmX\bm1_{n}}{\bmq'\bm1_{n}}\right)
    .
\end{equation*}
Hence for the diagonal matrix $\bmD = \clD(\bmd)$, we have the expression
\begin{equation}
    \bmD
    =
    \bmq'\bm1_{n}
    \left(
        \bmD_{q}
        + \clD\left(\frac{1}{\sqrt{n}}\frac{\bmB\bm1_{q}}{\bmq'\bm1_{n}}\right)
        + \clD\left(\frac{\bmX\bm1_{n}}{\bmq'\bm1_{n}}\right)
    \right)
    ,
\end{equation}
where we note that
\begin{align}
    \label{eq:bound_prelim_observation_3}
    \max_{i} \left| \left[ \clD\left(\frac{1}{\sqrt{n}}\frac{\bmB\bm1_{q}}{\bmq'\bm1_{n}}\right) \right]_{ii} \right|
    &\lesssim
    n^{-1/2}
    , &
    \max_{i} \left\| \left[\clD\left(\frac{\bmX\bm1_{n}}{\bmq'\bm1_{n}}\right)\right]_{ii} \right\|_{\psi_2}
    &\lesssim
    n^{-1/2}
    .
\end{align}
The first expression in \cref{eq:bound_prelim_observation_3} holds since the entries of the vector $\bmB\bm1_{q}$ are the row sums of the (deterministic) matrix $\bmB = \bmD_{q}\bmJ\bmM\bmJ'\bmD_{q}$ which are of order $n$ provided the degree parameters and entries of $\bmM$ are uniformly bounded.
The second expression relies on the observation that $\| [\bmX\bm1_{n}]_{ii} \|_{\psi_{2}} \lesssim n^{1/2}$ by \cite[Proposition~2.5.2(ii) and Proposition~2.6.1]{vershynin2018hdp}.

Next, we apply the Taylor expansion
$f_i(x) = (q_{i} + x)^{-\alpha} = q_{i}^{-\alpha} - \alpha q_{i}^{-(\alpha+1)} x + O(|x|^2)$
centered at $x_0 = 0$ simultaneously to each diagonal element of $\bmD^{-\alpha}$ to obtain
\begin{align}
    \label{eq:D_neg_alpha_expansion}
    \bmD^{-\alpha}
    &
    = (\bmq'\bm1_{n})^{-\alpha}\Bigg[
        \bmD_{q}^{-\alpha}
        + \clD\left(\frac{1}{\sqrt{n}}\frac{\bmB\bm1_{q}}{\bmq'\bm1_{n}}\right)^{-\alpha}
        + \clD\left(\frac{\bmX\bm1_{n}}{\bmq'\bm1_{n}}\right)^{-\alpha}
    \Bigg]
    \\& \nonumber
    = (\bmq'\bm1_{n})^{-\alpha} \Bigg[
        \bmD_{q}^{-\alpha}
        -
        \alpha\bmD_{q}^{-(\alpha+1)}
        \left[
            \clD\left(\frac{1}{\sqrt{n}}\frac{\bmB\bm1_{q}}{\bmq'\bm1_{n}}\right)
            +
            \clD\left(\frac{\bmX\bm1_{n}}{\bmq'\bm1_{n}}\right)
        \right]
        + \remtwo
    \Bigg]
    ,
\end{align}
where the $O(|x|^{2})$ remainder term $\remtwo$ can be written as
\begin{equation}
    [\remtwo]_{ii}
    =
    \frac{\alpha(\alpha+1)}{2!} (q_{i} + c_{i})^{-\alpha-2} \cdot x_{i}^2
    ,
\end{equation}
where
\begin{equation*}
    x_{i}
    =
    \left[
        \clD\left(\frac{1}{\sqrt{n}}\frac{\bmB\bm1_{q}}{\bmq'\bm1_{n}}\right)
        +
        \clD\left(\frac{\bmX\bm1_{n}}{\bmq'\bm1_{n}}\right)
    \right]_{ii}
    ,
\end{equation*}
and $c_i$ is between $0$ and $x_i$.

Bounding $\bbE \|\remtwo\|_{\op}^{\mmt}$ requires us to address that $(q_i+c_i)^{-\alpha-2}$ may diverge if $c_i$ approaches $-q_i$ depending on the value of $\alpha$.
To resolve this, we apply the law of total expectation to obtain:
\begin{align}
    \label{eq:remtwo:loe}
    \bbE\{\|\remtwo\|_{\op}^{\mmt}\}
    &
    = \bbE\left\{ \|\remtwo\|_{\op}^{\mmt} : \clE_{\epsilon} \right\}
    \Pr\left( \clE_{\epsilon} \right)
    \\&\qquad
    + \bbE\left\{ \|\remtwo\|_{\op}^{\mmt} : \clE_{\epsilon}^{C} \right\}
    \Pr\left( \clE_{\epsilon}^{C} \right)
    , \nonumber
\end{align}
where we have conditioned on the low-probability event
\begin{equation*}
    \clE_{\epsilon}
    =
    \left\{
        \left\|\clD\left(\frac{\bmX\bm1_{n}}{\bmq'\bm1_{n}}\right)\right\|_{\op} \geq \epsilon
    \right\}
\end{equation*}
for some $\epsilon \in (0, \qmin)$.
We now bound each term in the above display equation.
\begin{enumerate}
    \item Observe first that
    \begin{equation*}
        |[\remtwo]_{ii}|
        = \left| \frac{(d_{i})^{-\alpha}}{(\bmq'\bm1_{n})^{-\alpha}} - q_{i}^{-\alpha} + \alpha q_{i}^{-(\alpha+1)} x_i \right|
        \leq \left| \frac{(d_{i})^{-\alpha}}{(\bmq'\bm1_{n})^{-\alpha}} \right| + | q_{i}^{-\alpha} | + \left| \alpha q_{i}^{-(\alpha+1)} \right| | x_i |
        .
    \end{equation*}
    Now, consider the following case-by-case analysis:
    \begin{enumerate}
        \item $d_{i} = 0$:
        Since we defined $0^{\alpha} = 0$ for all $\alpha$, in this case we have
        \begin{equation*}
            \left| \frac{(d_{i})^{-\alpha}}{(\bmq'\bm1_{n})^{-\alpha}} \right|
            = \left| \frac{0}{(\bmq'\bm1_{n})^{-\alpha}} \right|
            =
            0
            .
        \end{equation*}
        \item $d_{i} \neq 0$:
        If $\alpha < 0$ then
        \begin{equation*}
            \left| \frac{(d_{i})^{-\alpha}}{(\bmq'\bm1_{n})^{-\alpha}} \right|
            \lesssim \left| \frac{(n)^{-\alpha}}{(n)^{-\alpha}} \right|
            =
            1
            ,
        \end{equation*}
        and if $\alpha \geq 0$ then
        \begin{equation*}
            \left| \frac{(d_{i})^{-\alpha}}{(\bmq'\bm1_{n})^{-\alpha}} \right|
            \lesssim \left| \frac{(1)^{-\alpha}}{(n)^{-\alpha}} \right|
            =
            n^{\alpha}
            ,
        \end{equation*}
        where the implicit constant in both cases does not depend on $i$.
    \end{enumerate}
    Moreover, we have that
    \begin{equation*}
        |x_i|
        \le
        \frac{1}{\sqrt{n}} \frac{|[\bmB \bm1_{n}]_{i}|}{\bmq'\bm1_{n}}
        +
        \frac{|[\bmX\bm1_{n}]_{i}|}{\bmq'\bm1_{n}}
        \lesssim
        \frac{1}{\sqrt{n}} + 1
        ,
    \end{equation*}
    where the implicit constant does not depend on $i$.
    Thus, we have that surely
    \begin{equation*}
        \|\remtwo\|_{\op}
        =
        \max_{i}
        |[\remtwo]_{ii}|
        \lesssim n^{\max(0,\alpha)}
        ,
    \end{equation*}
    and so
    \begin{equation}
        \label{eq:remtwo:loe:1}
        \|\remtwo\|_{\op}^{\mmt}
        \lesssim
        \left(n^{\max(0,\alpha)}\right)^{\mmt}
        =
        n^{\max(0,\mmt\alpha)}
        .
    \end{equation}

    \item It follows from \cref{eq:bound_prelim_observation_3}
    that there exists a constant $C > 0$
    so that provided $n$ is sufficiently large,
    \begin{align}
        \label{eq:remtwo:loe:2}
        \Pr\left( \clE_{\epsilon} \right)
        &
        =
        \Pr\left( \exists_{i} \; \left|\clD\left(\frac{\bmX\bm1_{n}}{\bmq'\bm1_{n}}\right)_{ii}\right| \geq \epsilon \right)
        \leq
        \sum_{i} \Pr\left(\left|\clD\left(\frac{\bmX\bm1_{n}}{\bmq'\bm1_{n}}\right)_{ii}\right| \geq \epsilon \right)
        \\& \nonumber
        \leq \sum_{i} 2 \exp\left(-\frac{\epsilon^2}{(C/\sqrt{n})^2}\right)
        = 2 n \exp(-C^{-2} \epsilon^2 n)
        .
    \end{align}

    \item
    Observe that for $\|\remtwo\|_{\op}^{\mmt}$, there exists a constant $C > 0$ so that provided $n$ is sufficiently large,
    \begin{align*}
        &
        \bbE\left\{ \|\remtwo\|_{\op}^{\mmt} : \clE_{\epsilon}^C \right\}
        \\
        &\qquad
        =
        \bbE\left\{ \max_{i} \left|\frac{\alpha(\alpha+1)}{2!} (q_{i} + c_{i})^{-\alpha-2}\right|^{\mmt} \cdot |x_{i}^2|^{\mmt} : \clE_{\epsilon}^C \right\}
        \\
        &\qquad
        \leq
        \bbE\left\{ \max_{i} \; C^{\mmt} \cdot |x_{i}^2|^{\mmt} : \clE_{\epsilon}^C \right\}
        =
        C^{\mmt} \cdot \bbE\left\{ \max_{i} |x_{i}^2|^{\mmt} : \clE_{\epsilon}^C \right\}
        ,
    \end{align*}
    where the inequality follows from the fact that $|c_{i}| \leq |x_{i}|$.
    Further,
    \begin{equation*}
        \max_{i} |x_i|
        \leq
        \max_{i}
        \left|
            \left[
                \clD\left(\frac{1}{\sqrt{n}}\frac{\bmB\bm1_{q}}{\bmq'\bm1_{n}}\right)
            \right]_{ii}
        \right|
        +
        \left|
            \left[
                \clD\left(\frac{\bmX\bm1_{n}}{\bmq'\bm1_{n}}\right)
            \right]_{ii}
        \right|
        \leq
        \frac{C}{\sqrt{n}} + \epsilon
        <
        \frac{C}{\sqrt{n}} + \min_{i} q_{i}
        .
    \end{equation*}
    Thus, we proceed and obtain
    \begin{align}
        \label{eq:remtwo:loe:3}
        &
        \bbE\left\{ \|\remtwo\|_{\op}^{\mmt} : \clE_{\epsilon}^{C} \right\}
        \Pr\left( \clE_{\epsilon}^{C} \right)
        \\&\qquad \nonumber
        \leq
        C^{\mmt} \cdot \bbE\left\{ \max_{i} |x_{i}^2|^{\mmt} : \clE_{\epsilon}^{C} \right\}
        \Pr\left( \clE_{\epsilon}^{C} \right)
        \\&\qquad \nonumber
        =
        C^{\mmt} \cdot \left[ \bbE \max_{i} |x_{i}^2|^{\mmt} - \bbE\left\{ \max_{i} |x_{i}^2|^{\mmt} : \clE_{\epsilon} \right\}
        \Pr\left( \clE_{\epsilon} \right) \right]
        \\&\qquad \nonumber
        \leq
        C^{\mmt} \cdot \bbE \max_{i} |x_{i}^2|^{\mmt}
        \\&\qquad \nonumber
        =
        C^{\mmt} \cdot \bbE \max_{i} |x_{i}|^{2\mmt}
        \\&\qquad \nonumber
        =
        C^{\mmt} \cdot \ellP[2\mmt]{\max_{i} |x_{i}|}^{2\mmt}
        \\&\qquad \nonumber
        \leq
        C^{\mmt} \cdot \Bigg[
            \ellP[2\mmt]{\max_{i} \left|\frac{1}{\sqrt{n}} \frac{|[\bmB \bm1_{n}]_{i}|}{\bmq'\bm1_{n}}\right|}
            +
            \ellP[2\mmt]{\max_{i} \left|\frac{|[\bmX\bm1_{n}]_{i}|}{\bmq'\bm1_{n}}\right|}
        \Bigg]^{2\mmt}
        \\&\qquad \nonumber
        \lesssim
        C^{\mmt} \cdot \Bigg[
            \max_{i} \left|\frac{1}{\sqrt{n}} \frac{|[\bmB \bm1_{n}]_{i}|}{\bmq'\bm1_{n}}\right|
            +
            \subG{\max_{i} \left|\frac{|[\bmX\bm1_{n}]_{i}|}{\bmq'\bm1_{n}}\right|}
        \Bigg]^{2\mmt}
        \\&\qquad \nonumber
        \lesssim
        \frac{1}{n^{\mmt}} + \left(\frac{\log n}{n}\right)^{\mmt}
        ,
    \end{align}
    where we have used \cref{eq:bound_prelim_observation_3,lem:max:subgaussian}.
\end{enumerate}
Combining \cref{eq:remtwo:loe,eq:remtwo:loe:1,eq:remtwo:loe:2,eq:remtwo:loe:3} and simplifying yields that for any $\mmt \geq 1$
\begin{equation}
    \ellP{\|\remtwo\|_{\op}}
    =
    \left[ \bbE\{\|\remtwo\|_{\op}^{\mmt}\} \right]^{1/\mmt}
    \lesssim
    \frac{\log n}{n}.
\end{equation}
Summarizing, we have
\begin{equation}
    \label{eq:jmlr:15}
    \bmD^{-\alpha}
    = (\bmq'\bm1_{n})^{-\alpha} ( \bmD_{q}^{-\alpha} + \lowtwo )
    ,
\end{equation}
where
\begin{equation}
    \lowtwo
    \coloneqq
    -
    \alpha\bmD_{q}^{-(\alpha+1)}
    \left[
        \clD\left(\frac{1}{\sqrt{n}}\frac{\bmB\bm1_{q}}{\bmq'\bm1_{n}}\right)
        +
        \clD\left(\frac{\bmX\bm1_{n}}{\bmq'\bm1_{n}}\right)
    \right]
    +
    \remtwo
    ,
\end{equation}
and it follows from our above analysis that for any $\mmt \geq 1$,
\begin{align}
    \ellP{\opnorm{\lowtwo}}
    &
    \lesssim
    \ellP{\opnorm{
            \clD\left(\frac{1}{\sqrt{n}}\frac{\bmB\bm1_{q}}{\bmq'\bm1_{n}}\right)
    }}
    +
    \ellP{\opnorm{
            \clD\left(\frac{\bmX\bm1_{n}}{\bmq'\bm1_{n}}\right)
    }}
    + \ellP{\opnorm{\remtwo}}
    \\& \nonumber
    =
    \ellP{
        \max_{i}\left|\frac{1}{\sqrt{n}} \frac{|[\bmB \bm1_{n}]_{i}|}{\bmq'\bm1_{n}}\right|
    }
    +
    \ellP{
        \max_{i}\left|\frac{|[\bmX\bm1_{n}]_{i}|}{\bmq'\bm1_{n}}\right|
    }
    + \ellP{\opnorm{\remtwo}}
    \\& \nonumber
    \lesssim
    n^{-1/2}
    + \left(\frac{\log n}{n}\right)^{1/2}
    + \frac{\log n}{n}
    \\& \nonumber
    \lesssim \left(\frac{\log n}{n}\right)^{1/2}
    .
\end{align}

\subsection{Characterizing \texorpdfstring{$\bmL_{\alpha}$}{L\_alpha}}
\newcommand{\remfour}{{\highlighted{\bmDelta_{4}}}}
Combining \cref{eq:decomp:d:one,eq:jmlr:15},
we obtain
\begin{align}
    \bmL_{\alpha}
    &
    = (\bmd'\bm1_{n})^{\alpha} \frac{1}{\sqrt{n}} \bmD^{-\alpha} \left[\bmA - \frac{\bmd \bmd'}{\bmd'\bm1_{n}}\right] \bmD^{-\alpha}
    \\& \nonumber
    =
    (\bmq'\bm1_{n})^{2\alpha} [ 1 + \lowone ]
    (\bmq'\bm1_{n})^{-\alpha} [ \bmD_{q}^{-\alpha} + \lowtwo ]
    \left[\frac{1}{\sqrt{n}} \left( \bmA - \frac{\bmd\bmd'}{\bmd'\bm1_{n}} \right)\right]
    (\bmq'\bm1_{n})^{-\alpha} [ \bmD_{q}^{-\alpha} + \lowtwo ]
    \\& \nonumber
    =
    [ 1 + \lowone ]
    [ \bmD_{q}^{-\alpha} + \lowtwo ]
    \left[\frac{1}{\sqrt{n}} \left( \bmA - \frac{\bmd\bmd'}{\bmd'\bm1_{n}} \right)\right]
    [ \bmD_{q}^{-\alpha} + \lowtwo ]
    \\& \nonumber
    =
    \bmD_{q}^{-\alpha}
    \left[\frac{1}{\sqrt{n}} \left( \bmA - \frac{\bmd\bmd'}{\bmd'\bm1_{n}} \right)\right]
    \bmD_{q}^{-\alpha}
    +
    \remfour
    ,
\end{align}
where
\begin{align}
    \remfour
    &
    \coloneqq
    \lowone
    [ \bmD_{q}^{-\alpha} + \lowtwo ]
    \left[\frac{1}{\sqrt{n}} \left( \bmA - \frac{\bmd\bmd'}{\bmd'\bm1_{n}} \right)\right]
    [ \bmD_{q}^{-\alpha} + \lowtwo ]
    \\&\qquad \nonumber
    +
    \lowtwo
    \left[\frac{1}{\sqrt{n}} \left( \bmA - \frac{\bmd\bmd'}{\bmd'\bm1_{n}} \right)\right]
    \bmD_{q}^{-\alpha}
    +
    \bmD_{q}^{-\alpha}
    \left[\frac{1}{\sqrt{n}} \left( \bmA - \frac{\bmd\bmd'}{\bmd'\bm1_{n}} \right)\right]
    \lowtwo
    \\&\qquad \nonumber
    +
    \lowtwo
    \left[\frac{1}{\sqrt{n}} \left( \bmA - \frac{\bmd\bmd'}{\bmd'\bm1_{n}} \right)\right]
    \lowtwo
    .
\end{align}
Now, using \cref{eq:jmlr:14} and substituting the definition of $\bmB$, we get the following analogue to \cite[Equation (15.5)]{ali2017improved}:
\begin{align*}
    \bmL_{\alpha}
    &
    =
    \bmD_{q}^{-\alpha}
    \bigg[
        \frac{1}{n} \bmB
        - \frac{1}{n} \frac{\bmq\bm1_{n}'\bmB}{\bmq'\bm1_{n}}
        - \frac{1}{n} \frac{\bmB\bm1_{n}\bmq'}{\bmq'\bm1_{n}}
        \\&\quad\nonumber
        + \frac{1}{n} \frac{\bm1_{n}'\bmB\bm1_{n}}{(\bmq'\bm1_{n})^2} \bmq\bmq'
        + \frac{\bmX}{\sqrt{n}}
        - \frac{1}{\sqrt{n}} \frac{\bmq\bm1_{n}'\bmX}{\bmq'\bm1_{n}}
        - \frac{1}{\sqrt{n}} \frac{\bmX\bm1_{n}\bmq'}{\bmq'\bm1_{n}}
        \\&\quad\nonumber
        + \frac{1}{\sqrt{n}} \frac{\bm1_{n}'\bmX\bm1_{n}}{(\bmq'\bm1_{n})^2} \bmq\bmq'
        - \remthree
    \bigg]
    \bmD_{q}^{-\alpha}
    +
    \remfour
    \\&
    =
    \bmD_{q}^{-\alpha}\frac{\bmX}{\sqrt{n}}\bmD_{q}^{-\alpha}
    + \frac{1}{n} \bmD_{q}^{1-\alpha}\bmJ\bmM\bmJ'\bmD_{q}^{1-\alpha}
    - \frac{1}{n} \frac{\bmD_{q}^{1-\alpha}\bm1_{n}\bm1_{n}'\bmD_{q}\bmJ\bmM\bmJ'\bmD_{q}^{1-\alpha}}{\bmq'\bm1_{n}}
    \\&\quad\nonumber
    - \frac{1}{n} \frac{\bmD_{q}^{1-\alpha}\bmJ\bmM\bmJ'\bmD_{q}\bm1_{n}\bm1_{n}'\bmD_{q}^{1-\alpha}}{\bmq'\bm1_{n}}
    + \frac{1}{n} \frac{\bm1_{n}'\bmD_{q}\bmJ\bmM\bmJ'\bmD_{q}\bm1_{n}}{(\bmq'\bm1_{n})^2}\bmD_{q}^{1-\alpha}\bm1_{n}\bm1_{n}'\bmD_{q}^{1-\alpha}
    - \frac{1}{\sqrt{n}} \frac{\bmD_{q}^{1-\alpha}\bm1_{n}\bm1_{n}'\bmX\bmD_{q}^{-\alpha}}{\bmq'\bm1_{n}}
    \\&\quad\nonumber
    - \frac{1}{\sqrt{n}} \frac{\bmD_{q}^{-\alpha}\bmX \bm1_{n} \bm1_{n}'\bmD_{q}^{1-\alpha}}{\bmq'\bm1_{n}}
    + \frac{1}{\sqrt{n}} \frac{\bm1_{n}'\bmX\bm1_{n}}{(\bmq'\bm1_{n})^2} \bmD_{q}^{1-\alpha} \bm1_{n} \bm1_{n}' \bmD_{q}^{1-\alpha}
    - \bmD_{q}^{-\alpha}\remthree\bmD_{q}^{-\alpha}
    + \remfour
    ,
\end{align*}
and noting that $\bm1_{n} = \bmJ\bm1_{K}$ and $\bmD_{q}\bm1_{n} = \bmq$, we have
\newcommand{\remfive}{{\highlighted{\bmDelta_{5}}}}
\begin{align}
    \bmL_{\alpha}
    &
    =
    \bmD_{q}^{-\alpha}\frac{\bmX}{\sqrt{n}}\bmD_{q}^{-\alpha}
    + \frac{1}{n} \bmD_{q}^{1-\alpha}\bmJ\bmM\bmJ'\bmD_{q}^{1-\alpha}
    - \frac{1}{n} \bmD_{q}^{1-\alpha}\bmJ\bm1_{K}\left(\frac{\bmJ'\bmq}{\bm1_{n}'\bmq}\right)'\bmM\bmJ'\bmD_{q}^{1-\alpha}
    \\&\quad \nonumber
    - \frac{1}{n} \bmD_{q}^{1-\alpha}\bmJ\bmM\left(\frac{\bmJ'\bmq}{\bm1_{n}'\bmq}\right)\bm1_{K}'\bmJ'\bmD_{q}^{1-\alpha}
    + \frac{1}{n} \bmD_{q}^{1-\alpha}\bmJ\bm1_{K}\left(\frac{\bmJ'\bmq}{\bm1_{n}'\bmq}\right)'\bmM\left(\frac{\bmJ'\bmq}{\bm1_{n}'\bmq}\right)\bm1_{K}'\bmJ'\bmD_{q}^{1-\alpha}
    \\&\quad \nonumber
    - \frac{1}{\sqrt{n}\bmq'\bm1_{n}} \bmD_{q}^{1-\alpha}\bmJ \bm1_{K} \bm1_{n}'\bmX\bmD_{q}^{-\alpha}
    - \frac{1}{\sqrt{n}\bmq'\bm1_{n}} \bmD_{q}^{-\alpha}\bmX \bm1_{n} \bm1_{K}'\bmJ' \bmD_{q}^{1-\alpha}
    \\&\quad \nonumber
    + \frac{1}{\sqrt{n}} \frac{\bm1_{n}'\bmX\bm1_{n}}{(\bmq'\bm1_{n})^2} \bmD_{q}^{1-\alpha} \bm1_{n} \bm1_{n}' \bmD_{q}^{1-\alpha}
    - \bmD_{q}^{-\alpha}\remthree\bmD_{q}^{-\alpha}
    + \remfour
    \\& \nonumber
    =
    \btlL_{\alpha}
    + \remfive
    ,
\end{align}
where
\begin{equation}
    \remfive
    \coloneqq
    \frac{1}{\sqrt{n}} \frac{\bm1_{n}'\bmX\bm1_{n}}{(\bmq'\bm1_{n})^2} \bmD_{q}^{1-\alpha} \bm1_{n} \bm1_{n}' \bmD_{q}^{1-\alpha}
    -
    \bmD_{q}^{-\alpha}\remthree\bmD_{q}^{-\alpha}
    +
    \remfour
    .
\end{equation}
By the triangle inequality and by the homogeneity of norms, we have
\begin{align}
    \ellP{\opnorm{\remfive}}
    &
    \leq
    \ellP{\opnorm{\frac{1}{\sqrt{n}} \frac{\bm1_{n}'\bmX\bm1_{n}}{(\bmq'\bm1_{n})^2} \bmD_{q}^{1-\alpha} \bm1_{n} \bm1_{n}' \bmD_{q}^{1-\alpha}}}
    + \ellP{\opnorm{\bmD_{q}^{-\alpha}\remthree\bmD_{q}^{-\alpha}}}
    + \ellP{\opnorm{\remfour}}
    \\& \nonumber
    =
    \opnorm{ \frac{1}{\sqrt{n}} \frac{1}{(\bmq'\bm1_{n})^2} \bmD_{q}^{1-\alpha} \bm1_{n} \bm1_{n}' \bmD_{q}^{1-\alpha} } \ellP{\bm1_{n}'\bmX\bm1_{n}}
    + \ellP{\opnorm{\bmD_{q}^{-\alpha}\remthree\bmD_{q}^{-\alpha}}}
    + \ellP{\opnorm{\remfour}}
    ,
\end{align}
and so we now bound each of these terms as follows.
\begin{enumerate}
    \item
    Since the elements of $\bmD_{q}$ are bounded and $\alpha$ is fixed and $\bmq'\bm1_{n} \gtrsim n$, it holds deterministically that
    \begin{equation*}
        \opnorm{
            \frac{1}{\sqrt{n}}
            \frac{1}{(\bmq'\bm1_{n})^2}
            \bmD_{q}^{1-\alpha} \bm1_{n} \bm1_{n}' \bmD_{q}^{1-\alpha}
        }
        \leq
        \frac{1}{\sqrt{n}}
        \frac{1}{(\bmq'\bm1_{n})^2}
        \opnorm{ \bmD_{q}^{1-\alpha} }^2 \opnorm{ \bm1_{n} \bm1_{n}' }
        \lesssim
        \frac{1}{\sqrt{n}} \cdot \frac{1}{n^{2}} \cdot 1 \cdot n
        = n^{-3/2}
        .
    \end{equation*}
    Thus, recalling that $\ellP{\bm1_{n}'\bmX\bm1_{n}} \lesssim \subG{\bm1_{n}'\bmX\bm1_{n}} \lesssim n$, we have
    \begin{equation}
        \opnorm{ \frac{1}{\sqrt{n}} \frac{1}{(\bmq'\bm1_{n})^2} \bmD_{q}^{1-\alpha} \bm1_{n} \bm1_{n}' \bmD_{q}^{1-\alpha} } \ellP{\bm1_{n}'\bmX\bm1_{n}}
        \lesssim
        n^{-3/2} \cdot n
        = n^{-1/2}
        .
    \end{equation}

    \item
    For the second term, note that $\ellP{\opnorm{\bmD_{q}^{-\alpha}\remthree\bmD_{q}^{-\alpha}}} \leq \opnorm{\bmD_{q}^{-\alpha}}^2 \ellP{\opnorm{\remthree}}$ by submultiplicativity of $\opnorm{\cdot}$ and homogeneity of $\ellP{\cdot}$.
    So, it follows from \cref{eq:bound_remthree} that
    \begin{equation}
        \ellP{\opnorm{\bmD_{q}^{-\alpha}\remthree\bmD_{q}^{-\alpha}}} \lesssim
        n^{-1/2}
        .
    \end{equation}

    \item For the final term, note that,
    \begin{align}
        \label{eq:remfour:expand}
        \ellP{\opnorm{ \remfour }}
        &
        \leq
        \bigg\| \bigg\|
                \lowone
                [ \bmD_{q}^{-\alpha} + \lowtwo ]
                \left[\frac{1}{\sqrt{n}} \left( \bmA - \frac{\bmd\bmd'}{\bmd'\bm1_{n}} \right)\right]
                [ \bmD_{q}^{-\alpha} + \lowtwo ]
        \bigg\|_{\op} \bigg\|_{\mathsf{L}_{\mmt}}
        \\&\qquad \nonumber
        +
        \ellP{\opnorm{
                \lowtwo
                \left[\frac{1}{\sqrt{n}} \left( \bmA - \frac{\bmd\bmd'}{\bmd'\bm1_{n}} \right)\right]
                \bmD_{q}^{-\alpha}
        }}
        \\&\qquad \nonumber
        +
        \ellP{\opnorm{
                \bmD_{q}^{-\alpha}
                \left[\frac{1}{\sqrt{n}} \left( \bmA - \frac{\bmd\bmd'}{\bmd'\bm1_{n}} \right)\right]
                \lowtwo
        }}
        \\&\qquad \nonumber
        +
        \ellP{\opnorm{
                \lowtwo
                \left[\frac{1}{\sqrt{n}} \left( \bmA - \frac{\bmd\bmd'}{\bmd'\bm1_{n}} \right)\right]
                \lowtwo
        }}
        .
    \end{align}

    It remains to bound each line.
    \begin{enumerate}
        \item By submultiplicativity of $\opnorm{\cdot}$ and two applications of the Cauchy--Schwarz inequality, we obtain
        \begin{align}
            \label{eq:remfour:expand:1}
            &
            \bigg\| \bigg\|
                    \lowone
                    [ \bmD_{q}^{-\alpha} + \lowtwo ]
                    \left[\frac{1}{\sqrt{n}} \left( \bmA - \frac{\bmd\bmd'}{\bmd'\bm1_{n}} \right)\right]
                    [ \bmD_{q}^{-\alpha} + \lowtwo ]
            \bigg\|_{\op} \bigg\|_{\mathsf{L}_{\mmt}}
            \\&\qquad \nonumber
            \leq
            \bigg\|
                |\lowone|
                \opnorm{ \bmD_{q}^{-\alpha} + \lowtwo }^{2}
                \opnorm{ \frac{1}{\sqrt{n}} \left( \bmA - \frac{\bmd\bmd'}{\bmd'\bm1_{n}} \right) }
            \bigg\|_{\mathsf{L}_{\mmt}}
            \\&\qquad \nonumber
            \leq
            \ellP[2\mmt]{
                |\lowone|
                \opnorm{ \bmD_{q}^{-\alpha} + \lowtwo }^{2}
            }
            \ellP[2\mmt]{
                \opnorm{ \frac{1}{\sqrt{n}} \left( \bmA - \frac{\bmd\bmd'}{\bmd'\bm1_{n}} \right) }
            }
            \\&\qquad \nonumber
            \leq
            \ellP[4\mmt]{\lowone}
            \ellP[4\mmt]{\opnorm{ \bmD_{q}^{-\alpha} + \lowtwo }^{2}}
            \ellP[2\mmt]{
                \opnorm{ \frac{1}{\sqrt{n}} \left( \bmA - \frac{\bmd\bmd'}{\bmd'\bm1_{n}} \right) }
            }
            \\&\qquad \nonumber
            =
            \ellP[4\mmt]{\lowone}
            \ellP[8\mmt]{\opnorm{ \bmD_{q}^{-\alpha} + \lowtwo }}^{2}
            \ellP[2\mmt]{
                \opnorm{ \frac{1}{\sqrt{n}} \left( \bmA - \frac{\bmd\bmd'}{\bmd'\bm1_{n}} \right) }
            }
            \\&\qquad \nonumber
            \leq
            \ellP[4\mmt]{\lowone}
            \left(\opnorm{ \bmD_{q}^{-\alpha}} + \ellP[8\mmt]{\opnorm{ \lowtwo }}\right)^{2}
            \ellP[2\mmt]{
                \opnorm{ \frac{1}{\sqrt{n}} \left( \bmA - \frac{\bmd\bmd'}{\bmd'\bm1_{n}} \right) }
            }
            \\&\qquad \nonumber
            \lesssim
            n^{-1/2}
            \cdot
            \left(1 + \left(\frac{\log n}{n}\right)^{1/2}\right)^{2}
            \cdot
            1
            \\&\qquad \nonumber
            \lesssim
            n^{-1/2}
            .
        \end{align}

        \item Note that, by the same approach as above,
        \begin{align}
            \label{eq:remfour:expand:2}
            &
            \ellP{\opnorm{
                    \lowtwo
                    \left[\frac{1}{\sqrt{n}} \left( \bmA - \frac{\bmd\bmd'}{\bmd'\bm1_{n}} \right)\right]
                    \bmD_{q}^{-\alpha}
            }}
            +
            \ellP{\opnorm{
                    \bmD_{q}^{-\alpha}
                    \left[\frac{1}{\sqrt{n}} \left( \bmA - \frac{\bmd\bmd'}{\bmd'\bm1_{n}} \right)\right]
                    \lowtwo
            }}
            \\&\qquad \nonumber
            \leq
            2
            \ellP{
                \opnorm{ \lowtwo }
                \opnorm{ \frac{1}{\sqrt{n}} \left( \bmA - \frac{\bmd\bmd'}{\bmd'\bm1_{n}} \right) }
                \opnorm{ \bmD_{q}^{-\alpha} }
            }
            \\&\qquad \nonumber
            =
            2
            \opnorm{ \bmD_{q}^{-\alpha} }
            \ellP{
                \opnorm{ \lowtwo }
                \opnorm{ \frac{1}{\sqrt{n}} \left( \bmA - \frac{\bmd\bmd'}{\bmd'\bm1_{n}} \right) }
            }
            \\&\qquad \nonumber
            \leq
            2
            \opnorm{ \bmD_{q}^{-\alpha} }
            \ellP[2\mmt]{\opnorm{ \lowtwo }}
            \ellP[2\mmt]{\opnorm{ \frac{1}{\sqrt{n}} \left( \bmA - \frac{\bmd\bmd'}{\bmd'\bm1_{n}} \right) }}
            \\&\qquad \nonumber
            \lesssim
            2
            \cdot 1
            \cdot \left(\frac{\log n}{n}\right)^{1/2}
            \cdot 1
            \\&\qquad \nonumber
            \lesssim
            \left(\frac{\log n}{n}\right)^{1/2}
            .
        \end{align}

        \item Note that
        \begin{align}
            \label{eq:remfour:expand:3}
            &
            \ellP{\opnorm{
                    \lowtwo
                    \left[\frac{1}{\sqrt{n}} \left( \bmA - \frac{\bmd\bmd'}{\bmd'\bm1_{n}} \right)\right]
                    \lowtwo
            }}
            \\&\qquad \nonumber
            \leq
            \ellP{
                \opnorm{ \lowtwo }^2
                \opnorm{ \frac{1}{\sqrt{n}} \left( \bmA - \frac{\bmd\bmd'}{\bmd'\bm1_{n}} \right) }
            }
            \\&\qquad \nonumber
            \leq
            \ellP[2\mmt]{\opnorm{ \lowtwo }^2}
            \ellP[2\mmt]{\opnorm{ \frac{1}{\sqrt{n}} \left( \bmA - \frac{\bmd\bmd'}{\bmd'\bm1_{n}} \right) }}
            \\&\qquad \nonumber
            =
            \ellP[4\mmt]{\opnorm{ \lowtwo }}^2
            \ellP[2\mmt]{\opnorm{ \frac{1}{\sqrt{n}} \left( \bmA - \frac{\bmd\bmd'}{\bmd'\bm1_{n}} \right) }}
            \\&\qquad \nonumber
            \lesssim
            \left(\left(\frac{\log n}{n}\right)^{1/2}\right)^2 \cdot 1
            \\&\qquad \nonumber
            =
            \frac{\log n}{n}
            .
        \end{align}
    \end{enumerate}
    Combining \cref{eq:remfour:expand,eq:remfour:expand:1,eq:remfour:expand:2,eq:remfour:expand:3} yields
    \begin{equation}
        \ellP{\opnorm{ \remfour }}
        \lesssim
        \left(\frac{\log n}{n}\right)^{1/2}
        .
    \end{equation}
\end{enumerate}
Thus, we finally have
\begin{equation}
    \ellP{\opnorm{ \remfive }}
    \lesssim
    \left(\frac{\log n}{n}\right)^{1/2}
    ,
\end{equation}
which establishes \cref{eq:DCSBM:approx:L:decay} as desired.
The above proof argument readily extends to establish \cref{eq:DCSBM:approx:L:signflip:decay}, using the following observations.
\begin{itemize}
    \item For any
    $\bmR \in \{-1,1\}^{n \times n}$,
    $\bmu \in \bbR^n$,
    and
    $\bmv \in \bbR^n$,
    \begin{equation*}
        \| \bmR \circ (\bmu \bmv') \|_{\op}
        \leq
        \| \bmR \circ (\bmu \bmv') \|_{\frob}
        =
        \| \bmu \bmv' \|_{\frob}
        =
        \| \bmu \bmv' \|_{\op}
        .
    \end{equation*}

    \item For any
    $\bmR \in \{-1,1\}^{n \times n}$,
    $\bmA \in \bbR^{n \times n}$,
    $\bmd_1 \in \bbR^n$,
    and
    $\bmd_2 \in \bbR^n$,
    \begin{align*}
        \bmR \circ (\clD(\bmd_1) \bmA \clD(\bmd_2))
        &
        = \bmR \circ ((\bmd_1 \bm1') \circ \bmA \circ (\bm1 \bmd_2'))
        = \bmR \circ (\bmd_1 \bm1') \circ \bmA \circ (\bm1 \bmd_2')
        \\&
        = (\bmd_1 \bm1') \circ \bmR \circ \bmA \circ (\bm1 \bmd_2')
        = \clD(\bmd_1) (\bmR \circ \bmA) \clD(\bmd_2)
        .
    \end{align*}

    \item Numerous matrix terms throughout the above proof are rank one.
    In particular, all matrices in the definition of $\remthree$ are rank one.

    \item Signflip matrices do not alter the maximum absolute row sum upper bound on the operator norm.
    In particular, similar to the proof of \cref{prop:prelim:bounds}, with probability one it holds that $\|\bmR \circ \bmX\|_{\op} \le \sqrt{\|\bmR \circ \bmX\|_{1} \|\bmR \circ \bmX\|_{\infty}} = \sqrt{\|\bmX\|_{1} \|\bmX\|_{\infty}} \lesssim n$.

    \item The Hadamard product of a signflip matrix with a symmetric matrix of sub-Gaussian entries remains sub-Gaussian.
    In particular, the stated bounds in \cref{prop:prelim:bounds} hold with or without signflipping the underlying matrices.
\end{itemize}
This completes the proof of \cref{thm:DCSBM:approx:L:signflip:decay}.
\qed

\section{Proof of \cref{thm:DCSBM:2_infty}}
\label{proof:DCSBM:2_infty}
Recall from \cref{eq:DCSBM:signalparts} that
\begin{align*}
    \bmU &\coloneqq
    \begin{bmatrix}
        \frac{\bmD_{q}^{1-\alpha} \bmJ}{\sqrt{n}} &
        \frac{\bmD_{q}^{ -\alpha} \bmX \bm1_n}{\bmq' \bm1_n}
    \end{bmatrix}
    , &
    \bmLambda &\coloneqq
    \begin{bmatrix}
        (\bmI_K - \bm1_K \bbvc') \bmM (\bmI_K - \bbvc \bm1_K') & -\bm1_K \\
        -\bm1_K' & 0
    \end{bmatrix}
    , &
    \bbvc &\coloneqq \frac{\bmJ'\bmq}{\bm1_{n}'\bmq}
    .
\end{align*}
Let $\btlU$ be the first $K$ columns of $\bmU$,
$\bmu_{K+1}$ be the $(K+1)$-st column of $\bmU$,
and
$\btlLambda$ be the upper-left $K \times K$ submatrix of $\bmLambda$.
Then,
\begin{equation}
    \btlS_{\alpha}
    =
    \bmU \bmLambda \bmU'
    =
    \btlU \btlLambda \btlU'
    - \btlU \bm1_{K} \bmu_{K+1}'
    - \bmu_{K+1} \bm1_{K}' \btlU'
    .
\end{equation}
Applying the triangle inequality and simplifying yields
\begin{align}
    \left\| \| \btlS_{\alpha} \|_{2, \infty} \right\|_{\psi_{2}}
    &
    \leq
    \left\| \| \btlU \btlLambda \btlU' \|_{2, \infty} \right\|_{\psi_{2}}
    + \left\| \| \btlU \bm1_{K} \bmu_{K+1}' \|_{2, \infty} \right\|_{\psi_{2}}
    + \left\| \| \bmu_{K+1} \bm1_{K}' \btlU' \|_{2, \infty} \right\|_{\psi_{2}}
    \\& \nonumber
    =
    \left\| \| \btlU \btlLambda \btlU' \|_{2, \infty} \right\|_{\psi_{2}}
    + \left\| \| \btlU \bm1_{K} \|_{\infty} \| \bmu_{K+1} \|_{2} \right\|_{\psi_{2}}
    + \left\| \| \bmu_{K+1} \|_{\infty} \| \btlU \bm1_{K} \|_{2} \right\|_{\psi_{2}}
    \\& \nonumber
    =
    \| \btlU \btlLambda \btlU' \|_{2, \infty}
    + \| \btlU \bm1_{K} \|_{\infty} \Big\| \| \bmu_{K+1} \|_{2} \Big\|_{\psi_{2}}
    + \| \btlU \bm1_{K} \|_{2} \Big\| \| \bmu_{K+1} \|_{\infty} \Big\|_{\psi_{2}}
    ,
\end{align}
where we have used the fact that $\btlU$ and $\btlLambda$
are deterministic.
We start by bounding the deterministic terms.
Note that
\begin{equation*}
    |(\btlU \bm1_{K})_{i}|
    = \left| \left( \frac{\bmD_{q}^{1-\alpha} \bmJ}{\sqrt{n}} \bm1_{K} \right)_{i} \right|
    = \left| \left( \frac{\bmD_{q}^{1-\alpha} \bm1_{n}}{\sqrt{n}} \right)_{i} \right|
    = \frac{q_{i}^{1-\alpha}}{\sqrt{n}}
    \lesssim
    \frac{1}{\sqrt{n}}
    ,
\end{equation*}
since
\begin{equation*}
    \max_{\ell=1,\dots,n} q_\ell^{1-\alpha}
    \leq
    \begin{cases}
        \qmax^{1-\alpha} & \text{if } \alpha \leq 1, \\
        \qmin^{1-\alpha} & \text{if } \alpha > 1,
    \end{cases}
\end{equation*}
which is constant with respect to $n$.
Thus,
\begin{align}
    \| \btlU \bm1_{K} \|_{\infty}
    &\lesssim
    \frac{1}{\sqrt{n}}
    , &
    \| \btlU \bm1_{K} \|_{2}
    &\lesssim
    1
    .
\end{align}
For the final deterministic term,
since $\btlU_{i:} = ( q_i^{1-\alpha} / \sqrt{n} ) \bme_{g_i}' \in \bbR^{K}$,
we have
\begin{align*}
    \left| (\btlU \btlLambda \btlU')_{ij} \right|
    =
    \left|
        ( q_i^{1-\alpha} / \sqrt{n} ) \bme_{g_i}'
        \btlLambda
        ( q_j^{1-\alpha} / \sqrt{n} ) \bme_{g_j}
    \right|
    =
    \frac{q_{i}^{1-\alpha} q_{j}^{1-\alpha}}{n}
    \left| \bme_{g_i}' \btlLambda \bme_{g_j} \right|
    =
    \frac{q_{i}^{1-\alpha} q_{j}^{1-\alpha}}{n}
    \left| \tllambda_{g_i g_j} \right|
    .
\end{align*}
Next, since $\|\bmM\|_{\max} \lesssim 1$ and $\bbvc$ is on the unit simplex, it holds that
\begin{align*}
    \left|
        \tllambda_{g_{i}g_{j}}
    \right|
    &=
    \left|((\bmI_K - \bm1_K \bbvc') \bmM (\bmI_K - \bbvc \bm1_K'))_{g_{i}g_{j}}\right|\\
    &=
    \left|((\bmI_K - \bm1_K \bbvc') (\bmM - \bmM\bbvc \bm1_K'))_{g_{i}g_{j}}\right|\\
    &=
    \left|\left(
    \bmM - \bm1_K \bbvc' \bmM - \bmM\bbvc \bm1_K' + \bm1_K \bbvc' \bmM \bbvc \bm1_{K}'\right)_{g_{i}g_{j}}\right|\\
    &\leq
    |\bmM_{g_{i}g_{j}}|
    +
    |\bbvc'\bmM\bme_{g_{j}}|
    +
    |\bme_{g_{i}}'\bmM\bbvc|
    +
    |\bbvc' \bmM \bbvc|\\
    &\le
    4 \|\bmM\|_{\max}\\
    &\lesssim
    1
    .
\end{align*}
Thus,
\begin{align*}
    \left| (\btlU \btlLambda \btlU')_{ij} \right|
    =
    \frac{q_{i}^{1-\alpha} q_{j}^{1-\alpha}}{n}
    \left| \tllambda_{g_i g_j} \right|
    \lesssim \frac{1}{n}
    ,
\end{align*}
and so
\begin{equation}
    \| \btlU \btlLambda \btlU' \|_{2, \infty}
    = \max_{i=1,\dots,n} \left[ \sum_{j = 1}^{n} \left| (\btlU \btlLambda \btlU')_{ij} \right|^2 \right]^{1/2}
    \lesssim \frac{1}{\sqrt{n}}
    .
\end{equation}

We now bound the random terms.
Note first that the entries of $\bmu_{K+1}$
are mean-zero sub-Gaussian random variables
with sub-Gaussian norm
bounded as
\begin{equation}
    \| (\bmu_{K+1})_{i} \|_{\psi_2}^2
    =
    \left\|
        \left(
            \frac{\bmD_{q}^{ -\alpha} \bmX \bm1_n}{\bmq' \bm1_n}
        \right)_{i}
    \right\|_{\psi_2}^2
    =
    \left\|
        \frac{q_i^{-\alpha}}{\bmq' \bm1_n}
        \sum_{j=1}^n X_{ij}
    \right\|_{\psi_2}^2
    =
    \left(
        \frac{q_i^{-\alpha}}{\bmq' \bm1_n}
    \right)^2
    \left\|
        \sum_{j=1}^n X_{ij}
    \right\|_{\psi_2}^2
    \lesssim
    \frac{1}{n^2}
    \left\|
        \sum_{j=1}^n X_{ij}
    \right\|_{\psi_2}^2
    ,
\end{equation}
where we have used the fact that
\begin{equation*}
    \max_{i=1,\dots,n} \left| \frac{q_i^{-\alpha}}{\bmq'\bm1_{n}} \right|
    \leq
    \begin{cases}
        \qmax^{-\alpha} / (n \cdot \qmin) & \text{if } \alpha \leq 0, \\
        \qmin^{-\alpha} / (n \cdot \qmin) & \text{if } \alpha > 0,
    \end{cases}
\end{equation*}
whence $\max_{i=1,\dots,n} \left| q_i^{-\alpha} / (\bmq'\bm1_{n}) \right| \lesssim 1/n$.
Now, from \cite[Example 2.5.8(iii)]{vershynin2018hdp} and the assumption $\|\bmM\|_{\max} \lesssim 1$, it holds that
\begin{equation*}
    \|X_{ij}\|_{\psi_{2}}
    \lesssim
    \|X_{ij}\|_\infty
    =
    \left\| A_{ij} - q_i q_j - \frac{q_i q_j M_{g_i g_j}}{\sqrt{n}} \right\|_\infty
    \leq
    \left\| A_{ij} - q_i q_j \right\|_\infty + \left| \frac{q_i q_j M_{g_i g_j}}{\sqrt{n}} \right|
    \leq
    1 + \|\bmM\|_{\max}
    \lesssim 1
    .
\end{equation*}
Thus, since
$X_{ij}$ are mean-zero and independent across $j$,
it follows from \cite[Proposition 2.6.1]{vershynin2018hdp} that
\begin{equation}
    \| (\bmu_{K+1})_{i} \|_{\psi_2}^2
    \lesssim
    \frac{1}{n^2}
    \left\|
        \sum_{j=1}^n X_{ij}
    \right\|_{\psi_2}^2
    \lesssim
    \frac{1}{n^2}
    \sum_{j=1}^n \| X_{ij} \|_{\psi_2}^2
    \lesssim
    \frac{1}{n}
    .
\end{equation}
As a result, $\| \bmu_{K+1} \|_{2}$ is sub-Gaussian since
\begin{equation}
    \Big\| \| \bmu_{K+1} \|_{2} \Big\|_{\psi_2}^{2}
    =
    \Big\| \| \bmu_{K+1} \|_{2}^{2} \Big\|_{\psi_1}
    =
    \left\| \sum_{i=1}^n |(\bmu_{K+1})_{i}|^2 \right\|_{\psi_1}
    \leq
    \sum_{i=1}^n \Big\| |(\bmu_{K+1})_{i}|^2 \Big\|_{\psi_1}
    =
    \sum_{i=1}^n \| (\bmu_{K+1})_{i} \|_{\psi_2}^{2}
    \lesssim 1
    .
\end{equation}
Similarly,
it follows from \cref{lem:max:subgaussian}
that $\|\bmu_{K+1}\|_\infty$ is also sub-Gaussian and satisfies
\begin{equation}
    \Big\| \|\bmu_{K+1}\|_\infty \Big\|_{\psi_{2}}
    =
    \left\| \max_{i=1,\dots,n} |(\bmu_{K+1})_{i}| \right\|_{\psi_{2}}
    \lesssim
    \left( \max_{i=1,\dots,n} \| (\bmu_{K+1})_{i} \|_{\psi_{2}} \right) \sqrt{\log n}
    \lesssim
    \sqrt{\frac{\log n}{n}}
    .
\end{equation}
Combining the above observations yields the stated bound for the signal term, namely
\begin{align}
    \left\| \| \btlS_{\alpha} \|_{2, \infty} \right\|_{\psi_{2}}
    &\leq
    \| \btlU \btlLambda \btlU' \|_{2, \infty}
    +
    \| \btlU \bm1_{K} \|_{\infty}
    \cdot
    \Big\| \| \bmu_{K+1} \|_{2}
    \Big\|_{\psi_{2}}
    +
    \| \btlU \bm1_{K} \|_{2}
    \cdot
    \Big\| \| \bmu_{K+1} \|_{\infty}
    \Big\|_{\psi_{2}}
    \\&\lesssim \nonumber
    \frac{1}{\sqrt{n}}
    +
    \frac{1}{\sqrt{n}}
    \cdot
    1
    +
    1
    \cdot
    \sqrt{\frac{\log n}{n}}
    \\&\lesssim \nonumber
    \sqrt{\frac{\log n}{n}}
    .
\end{align}

To bound the signflipped signal, we adopt a similar proof strategy as above.
Note that $\bmR \circ (\btlU \btlLambda \btlU')$ is a symmetric random matrix whose entries are independent mean-zero sub-Gaussian (bounded) random variables.
Thus, applying \cref{lem:opnorm:subgaussian} and the earlier uniform entrywise deterministic bound yields
\begin{equation}
    \left\|
        \big\|
            \bmR \circ (\btlU \btlLambda \btlU')
        \big\|_{\op}
    \right\|_{\psi_{2}}
    \lesssim
    \left( \max_{i,j=1,\dots,n} \| R_{ij} (\btlU \btlLambda \btlU')_{ij} \|_{\psi_{2}} \right) \sqrt{n}
    \lesssim
    \left( \max_{i,j=1,\dots,n} |(\btlU \btlLambda \btlU')_{ij} | \right) \sqrt{n}
    \lesssim
    \frac{1}{\sqrt{n}}
    .
\end{equation}
Next, observe that
\begin{align*}
    \left\|
        \bmR \circ (\btlU \bm1_{K} \bmu_{K+1}')
    \right\|_{\op}
    &
    =
    \left\|
        \diag(\btlU \bm1_{K})
        \bmR
        \diag(\bmu_{K+1})
    \right\|_{\op}
    \\
    &\leq
    \left\|
        \diag(\btlU \bm1_{K})
    \right\|_{\op}
    \left\|
        \bmR
    \right\|_{\op}
    \left\|
        \diag(\bmu_{K+1})
    \right\|_{\op}
    =
    \|\btlU \bm1_{K}\|_{\infty}
    \left\| \bmR \right\|_{\op}
    \|\bmu_{K+1}\|_{\infty}
    .
\end{align*}
By \cref{lem:opnorm:subgaussian}, $\| \bmR \|_{\op}$ is sub-Gaussian with $\big\| \| \bmR \|_{\op} \big\|_{\psi_2} \lesssim \sqrt{n}$.
Thus, it follows
from \cite[Lemma 2.7.7]{vershynin2018hdp}
that
\begin{equation}
    \left\|
        \big\|
            \bmR \circ (\btlU \bm1_{K} \bmu_{K+1}')
        \big\|_{\op}
    \right\|_{\psi_{1}}
    \leq
    \|\btlU \bm1_{K}\|_{\infty}
    \big\| \| \bmR \|_{\op} \big\|_{\psi_{2}}
    \big\| \|\bmu_{K+1}\|_{\infty} \big\|_{\psi_{2}}
    \lesssim
    \frac{1}{\sqrt{n}}
    \cdot
    \sqrt{n}
    \cdot
    \sqrt{\frac{\log n}{n}}
    = \sqrt{\frac{\log n}{n}}
    .
\end{equation}
By the same approach, it follows that
\begin{equation}
    \left\|
        \|
            \bmR \circ (\bmu_{K+1} \bm1_{K}' \btlU')
        \|_{\op}
    \right\|_{\psi_{1}}
    \lesssim
    \sqrt{\frac{\log n}{n}}
    .
\end{equation}
Finally, since $\|\cdot\|_{\psi_{1}} \le \|\cdot\|_{\psi_{2}}$ holds for sub-Gaussian random variables, we obtain
\begin{align}
    \left\|
        \|\bmR \circ \btlS_{\alpha}\|_{\op}
    \right\|_{\psi_{1}}
    &
    \leq
    \left\|
        \|\bmR \circ (\btlU \btlLambda \btlU')\|_{\op}
    \right\|_{\psi_{1}}
    +
    \left\|
        \|\bmR \circ (\btlU \bm1_{K} \bmu_{K+1}')\|_{\op}
    \right\|_{\psi_{1}}
    +
    \left\|
        \|\bmR \circ (\bmu_{K+1} \bm1_{K}' \btlU')\|_{\op}
    \right\|_{\psi_{1}}
    \\& \nonumber
    \lesssim
    \frac{1}{\sqrt{n}}
    +
    \sqrt{\frac{\log n}{n}}
    +
    \sqrt{\frac{\log n}{n}}
    \\& \nonumber
    \lesssim
    \sqrt{\frac{\log n}{n}}
    .
\end{align}
This completes the proof of \cref{thm:DCSBM:2_infty}.
\qed

\section{Proof of \cref{thm:DCSBM:noise:signflip:preserve}}
\label{proof:DCSBM:noise:signflip:preserve}
Recall that
\begin{equation*}
    \btlN_{\alpha}
    =
    \frac{1}{\sqrt{n}} \bmD_{q}^{-\alpha} \bmX \bmD_{q}^{-\alpha}
    ,
\end{equation*}
where $\bmD_{q} \coloneqq \diag(q_1,\dots,q_n)$
and
\begin{equation*}
    \bmX
    =
    \bmA
    -
    \left(
        \bmq \bmq' + \frac{1}{\sqrt{n}} \bmq \bmq' \circ \bmJ \bmM \bmJ'
    \right)
    =
    \bmA
    -
    \left(
        \bmq \bmq' + \frac{1}{\sqrt{n}} \bmB
    \right)
    .
\end{equation*}
Thus,
\begin{equation}
    \btlN_{\alpha}
    = \frac{1}{\sqrt{n}} \bmD_{q}^{-\alpha} \bmX \bmD_{q}^{-\alpha}
    =
    \frac{1}{\sqrt{n}}
    \bmD_{q}^{-\alpha}
    \left[
        \bmA
        -
        \left(
            \bmq \bmq' + \frac{1}{\sqrt{n}} \bmB
        \right)
    \right]
    \bmD_{q}^{-\alpha}
    ,
\end{equation}
and for any $\bmq$,
\begin{align}
    \bbE\, \btlN_{\alpha}
    &
    = \bm0_{n \times n}
    , \\
    \bbE |(\btlN_{\alpha})_{ij}|^2
    &
    =
    \frac{1}{n}
    q_{i}^{-2\alpha}
    q_{j}^{-2\alpha}
    P_{ij}(1-P_{ij})
    =
    \frac{1}{n}
    q_{i}^{-2\alpha}
    q_{j}^{-2\alpha}
    q_{i}q_{j}C_{g_{i}g_{j}}(1-q_{i}q_{j}C_{g_{i}g_{j}})
    \\& \nonumber
    =
    \frac{1}{n}
    q_{i}^{-2\alpha}
    q_{j}^{-2\alpha}
    q_{i}q_{j}\left(1+\frac{M_{g_{i}g_{j}}}{\sqrt{n}}\right)\left(1-q_{i}q_{j}\left(1+\frac{M_{g_{i}g_{j}}}{\sqrt{n}}\right)\right)
    .
\end{align}

In what follows, we proceed to analyze
\begin{align*}
    \|\btlN_{\alpha}\|_{\op}
    =
    \max\{ |\lambda_{1}(\btlN_{\alpha})|, |\lambda_{N}(\btlN_{\alpha})| \}.
\end{align*}
For convenience,
we will work with a rescaled version
\begin{equation*}
    \bbrN_{\alpha}
    \coloneqq
    \qmin^{2\alpha} \cdot \btlN_{\alpha}
    ,
\end{equation*}
so that the second moments of the entries are bounded above by $1/n$.

We now verify that $\bbrN_{\alpha}$ satisfies \cite[Assumptions A--D]{ajanki2016ufg}.
\begin{description}
    \item[{\cite[Assumption A]{ajanki2016ufg}}]
    Follows immediately from the rescaling of $\bbrN_{\alpha}$.

    \item[{\cite[Assumption B]{ajanki2016ufg}}]
    For $n$ sufficiently large
    \begin{equation*}
        \left| \frac{M_{g_{i}g_{j}}}{\sqrt{n}} \right|
        \leq
        \min\left(
            \frac{1}{2},
            \frac{\qmax^{-2} - 1}{2}
        \right)
        ,
    \end{equation*}
    so that
    \begin{equation*}
        \tau
        \leq
        \frac{\qmin^2}{2}
        =
        \qmin^2 \left(1 - \frac{1}{2}\right)
        \leq
        P_{ij} = q_{i}q_{j}C_{g_{i}g_{j}}
        \leq
        \qmax^2 \left(1 + \frac{\qmax^{-2} - 1}{2}\right)
        =
        1 - \frac{1 - \qmax^{2}}{2}
        \leq
        1 - \tau
        ,
    \end{equation*}
    where $\tau \coloneqq \min(\qmin^2, 1 - \qmax^2)/2 > 0$.
    From this, it follows that
    \begin{equation}
        \bbE |(\bbrN_{\alpha})_{ij}|^2
        =
        \qmin^{4\alpha}
        \cdot
        \left[
            \frac{1}{n}
            q_{i}^{-2\alpha}
            q_{j}^{-2\alpha}
            P_{ij}(1-P_{ij})
        \right]
        \geq
        \qmin^{4\alpha}
        \cdot
        \left[
            \frac{1}{n}
            \qmax^{-4\alpha}
            \tau(1-\tau)
        \right]
        =
        \frac{(\qmin/\qmax)^{4\alpha} \tau(1-\tau)}{n}
        ,
    \end{equation}
    and so
    $\bbrN_{\alpha}$ also satisfies
    \cite[Assumption B]{ajanki2016ufg}
    with parameter $(\qmin/\qmax)^{4\alpha} \tau(1-\tau)$.

    \item[{\cite[Assumption C]{ajanki2016ufg}}]
    We shall verify this assumption by invoking \cite[Theorem 6.1]{ajanki2019qve}.
    In order to do so, in the notation of \cite{ajanki2019qve}, our (finite-dimensional) setting corresponds to:
    \begin{align*}
        \fkX &= \{1,\dots,n\}
        , \\
        \scB &= \bbC^{n} \quad \text{(up to a trivial isomorphism)}
        , \\
        L^2 = L^2(\fkX; \bbC) &= \bbC^{n} \quad \text{(up to a trivial isomorphism)}
        , \\
        \clS &= 2^{\fkX}
        , \\
        \pi(A) &= \frac{|A|}{n}
        , \quad \text{for } A \subseteq \fkX, \\
        \langle \bmu, \bmw \rangle &= \frac{1}{n} \bmu'\bmw
        , \\
        \langle \bmw \rangle &= \frac{1}{n} \bmw'\bm1_n
        , \\
        S &: \bmu \in \bbC^{n} \mapsto \bmV \bmu \in \bbC^{n}
        ,
    \end{align*}
    where $\bmV$ is the $n \times n$ matrix of variances,
    i.e., $V_{ij} = \bbE |(\bbrN_{\alpha})_{ij}|^2$.

    Thus, \cite[Assumption A.1]{ajanki2019qve} amounts to
    \begin{equation}
        \forall_{\bmu,\bmw \in \bbC^{n}} \; \forall_{\bmp \in \bbR^{n}_{\geq 0}} \quad
        \bmu' \bmV \bmw = \bmu' \bmV' \bmw
        \quad \text{and} \quad
        \bmV \bmp \in \bbR^{n}_{\geq 0}
        ,
    \end{equation}
    which follows immediately since $\bmV$ is symmetric and has real nonnegative entries.

    \cite[Assumption A.2]{ajanki2019qve} follows immediately
    by noting that $S$ is already a bounded operator
    from $L^2$ (aka. $\bbC^{n}$) to $\scB$ (aka. $\bbC^{n}$)
    that can be represented as follows:
    \begin{equation}
        (S\bmw)_{i}
        =
        (\bmV \bmw)_i
        =
        \sum_{j=1}^n K_{ij} w_{j} \frac{1}{n}
        ,
    \end{equation}
    where the kernel matrix is $\bmK = n \bmV$ is symmetric and nonnegative.

    \cite[Assumption A.3]{ajanki2019qve}
    holds with $L = 1$ and $\rho = (\qmin/\qmax)^{4\alpha} \tau(1-\tau)$ since
    for any $\bmu \in \scB$ satisfying $\bmu \geq 0$
    and any $i \in \{1,\dots,n\}$ we have
    \begin{align}
        (S \bmu)_i
        =
        (\bmV \bmu)_i
        =
        \sum_{j=1}^n V_{ij} u_j
        \geq
        \sum_{j=1}^n \frac{(\qmin/\qmax)^{4\alpha} \tau(1-\tau)}{n} u_j
        =
        (\qmin/\qmax)^{4\alpha} \tau(1-\tau) \cdot \langle \bmu \rangle
        =
        \rho \cdot \langle \bmu \rangle
        .
    \end{align}

    \cite[Assumption B.1]{ajanki2019qve} follows by taking
    $K = 1$, $\bmZ = [1] \in \{0,1\}^{1 \times 1}$ and $\clI = \{\fkX\}$,
    then noting that $\bmZ$ is trivially fully indecomposable since it contains no zeros,
    $\pi(I_1) = \pi(\fkX) = 1 = 1/K$,
    and
    \begin{align*}
        &
        \forall_{i,j \in \{1,\dots,n\}} \quad
        K_{ij} = n V_{ij} \geq (\qmin/\qmax)^{4\alpha} \tau(1-\tau) = \varphi Z_{11}
        ,
    \end{align*}
    where $\varphi \coloneqq (\qmin/\qmax)^{4\alpha} \tau(1-\tau)$.

    Finally, note that
    \begin{align}
        \Gamma(\infty)
        &
        \coloneqq
        \lim_{\tau \to \infty}
        \min_{1 \le i \le n} \sqrt{\sum_{j=1}^{n} \left(\frac{1}{\tau} + 0 + \|S_{j}-S_{i}\|_{2}\right)^{-2}\frac{1}{n}}
        \\& \nonumber
        =
        \lim_{\tau \to \infty}
        \min_{1 \le i \le n} \sqrt{\sum_{j=1}^{n} \left(\frac{1}{\tau} + \sqrt{n \|\bmV_{j,:} - \bmV_{i,:}\|_2^2}\right)^{-2}\frac{1}{n}}
        \\& \nonumber
        =
        \lim_{\tau \to \infty}
        \min_{1 \le i \le n}
        \sqrt{
            \left(\frac{1}{\tau} + \sqrt{n \|\bmV_{i,:} - \bmV_{i,:}\|_2^2}\right)^{-2}\frac{1}{n}
            +
            \sum_{j \neq i} \left(\frac{1}{\tau} + \sqrt{n \|\bmV_{j,:} - \bmV_{i,:}\|_2^2}\right)^{-2}\frac{1}{n}
        }
        \\& \nonumber
        \geq
        \lim_{\tau \to \infty}
        \min_{1 \le i \le n}
        \sqrt{
            \left(\frac{1}{\tau} + \sqrt{n \|\bmV_{i,:} - \bmV_{i,:}\|_2^2}\right)^{-2}\frac{1}{n}
        }
        =
        \lim_{\tau \to \infty}
        \min_{1 \le i \le n}
        \frac{\tau}{\sqrt{n}}
        =
        \lim_{\tau \to \infty}
        \frac{\tau}{\sqrt{n}}
        =
        \infty
        ,
    \end{align}
    where we have used the fact that
    \begin{align*}
        \|S_{j}-S_{i}\|_{2}^{2}
        &= \sum_{\ell=1}^{n} |(S_{j}-S_{i})_{\ell}|^{2} \frac{1}{n}\\
        &= \sum_{\ell=1}^{n} |(S_{j})_{\ell}-(S_{i})_{\ell}|^{2} \frac{1}{n}\\
        &= \sum_{\ell=1}^{n} |K_{j,\ell}-K_{i,\ell}|^{2} \frac{1}{n}\\
        &= \sum_{\ell=1}^{n} |n V_{j,\ell} - n V_{i,\ell}|^{2} \frac{1}{n}\\
        &= n \|\bmV_{j,:} - \bmV_{i,:}\|_{2}^{2}
        .
    \end{align*}

    Thus, it follows from \cite[Theorem 6.1]{ajanki2019qve} (noting that $a=0$ in our case) that
    for $i = 1,\dots,n$ and $z \in \bbH$, we have
    \begin{equation}
        |m_i(z)|
        \leq
        P
        \coloneqq
        \max\left\{
            \Phi, \frac{\delta}{2} \, \Gamma^{-1}\left(\frac{4}{\delta^2}\right)
        \right\}
        ,
    \end{equation}
    where $\Phi < \infty$ and $\delta > 0$ (and hence $P$) depend only on $\varphi$ and $K$,
    which are independent of $n$.

    \item[{\cite[Assumption D]{ajanki2016ufg}}]
    Follows by noting that
    for all $\mmt \in \bbN$
    and all $i,j = 1,\dots,n$
    we have
    \begin{align}
        \bbE |(\bbrN_{\alpha})_{ij}|^{\mmt}
        &
        =
        \bbE |(\qmin^{2\alpha} \cdot \btlN_{\alpha})_{ij}|^{\mmt}
        =
        \bbE \left|\left(
            \qmin^{2\alpha}
            \cdot
            \frac{1}{\sqrt{n}} \bmD_{q}^{-\alpha} \bmX \bmD_{q}^{-\alpha}
        \right)_{ij} \right|^{\mmt}
        \\& \nonumber
        =
        \bbE \left|
            \qmin^{2\alpha}
            \cdot
            \frac{1}{\sqrt{n}}
            q_{i}^{-\alpha}
            q_{j}^{-\alpha}
            X_{ij}
        \right|^{\mmt}
        =
        \left|
            \qmin^{2\alpha}
            \cdot
            \frac{1}{\sqrt{n}}
        \right|^{\mmt}
        \bbE \left|
            q_{i}^{-\alpha}
            q_{j}^{-\alpha}
            X_{ij}
        \right|^{\mmt}
        \\& \nonumber
        \leq
        \left|
            \qmin^{2\alpha}
            \cdot
            \frac{1}{\sqrt{n}}
        \right|^{\mmt}
        \left|
            \qmin^{-2\alpha}
        \right|^{\mmt}
        =
        \frac{1}{n^{\mmt/2}}
        =
        \left(
            \frac{(\qmin/\qmax)^{4\alpha} \tau(1-\tau)}{(\qmin/\qmax)^{4\alpha} \tau(1-\tau)}
            \frac{1}{n}
        \right)^{\mmt/2}
        \\& \nonumber
        =
        \mu_{\mmt}
        \left[
            \frac{(\qmin/\qmax)^{4\alpha} \tau(1-\tau)}{n}
        \right]^{\mmt/2}
        \leq
        \mu_{\mmt}
        \left[ \bbE |(\bbrN_{\alpha})_{ij}|^2 \right]^{\mmt/2}
        ,
    \end{align}
    where $\mu_{\mmt} \coloneqq [(\qmin/\qmax)^{4\alpha} \tau(1-\tau)]^{-\mmt/2}$
    and recall that $\tau \coloneqq \min(\qmin^2, 1 - \qmax^2)/2 > 0$.
\end{description}
We have shown that {\cite[Assumptions A--D]{ajanki2016ufg}} hold, hence by \cite[Corollary 1.11]{ajanki2016ufg} there exists some $\xi_n > 0$ so that
\begin{align}
    \label{eq:local:law:noise:lower}
    &
    \forall_{\delta, D > 0}
    \exists_{N_0}
    \forall_{n \geq N_0}
    \quad
    \Pr \left( |\lambda_{1}(\bbrN_{\alpha}) - (-\xi_n)| > n^{\delta} n^{-2/3} \right)
    \leq n^{-D}
    ,
    \\&
    \label{eq:local:law:noise:upper}
    \forall_{\delta, D > 0}
    \exists_{N_0}
    \forall_{n \geq N_0}
    \quad
    \Pr \left( |\lambda_{n}(\bbrN_{\alpha}) - (+\xi_n)| > n^{\delta} n^{-2/3} \right)
    \leq n^{-D}
    .
\end{align}
In particular,
$\xi_n$ is the upper-edge
(i.e., the supremum of the support)
of the probability density $\rho$
defined in \cite[Equation (1.12)]{ajanki2016ufg}
that corresponds to the quadratic vector equation \cite[Equation (1.7)]{ajanki2016ufg}.
Likewise,
\begin{align}
    \label{eq:local:law:noiseflip:lower}
    &
    \forall_{\delta, D > 0}
    \exists_{N_0}
    \forall_{n \geq N_0}
    \quad
    \Pr \left( |\lambda_{1}(\bmR \circ \bbrN_{\alpha}) - (-\xi_n)| > n^{\delta} n^{-2/3} \right)
    \leq n^{-D}
    ,
    \\&
    \label{eq:local:law:noiseflip:upper}
    \forall_{\delta, D > 0}
    \exists_{N_0}
    \forall_{n \geq N_0}
    \quad
    \Pr \left( |\lambda_{n}(\bmR \circ \bbrN_{\alpha}) - (+\xi_n)| > n^{\delta} n^{-2/3} \right)
    \leq n^{-D}
    .
\end{align}
Now, for any choice of positive integer $\mmt$, note that
\begin{align}
    \bbE \left| \|\bmR \circ \bbrN_{\alpha}\|_{\op} - \|\bbrN_{\alpha}\|_{\op} \right|^{\mmt}
    &
    =
    \bbE \left|
        \|\bmR \circ \bbrN_{\alpha}\|_{\op} - \xi_{n}
        +
        \xi_{n} - \|\bbrN_{\alpha}\|_{\op}
    \right|^{\mmt}
    \\& \nonumber
    \leq
    2^{\mmt-1}
    \left(
        \bbE \left| \|\bmR \circ \bbrN_{\alpha}\|_{\op} - \xi_{n} \right|
        +
        \bbE \left| \|\bbrN_{\alpha}\|_{\op} - \xi_{n} \right|
    \right)
    .
\end{align}
We now study each term,
beginning with $\bbE \left| \|\bbrN_{\alpha}\|_{\op} - \xi_{n} \right|$.
Let $\clE = \left\{ \left| \|\bbrN_{\alpha}\|_{\op} - \xi_{n} \right| \leq n^{-1/2} \right\}$,
and consider the expansion
\begin{equation}
    \bbE \left| \|\bbrN_{\alpha}\|_{\op} - \xi_{n} \right|^{\mmt}
    =
    \bbE \left[ \left| \|\bbrN_{\alpha}\|_{\op} - \xi_{n} \right|^{\mmt} : \clE \right]
    \Pr(\clE)
    +
    \bbE \left[ \left| \|\bbrN_{\alpha}\|_{\op} - \xi_{n} \right|^{\mmt} : \clE^{C} \right]
    \Pr(\clE^{C})
    .
\end{equation}
We will now bound each term on its own.
\begin{enumerate}
    \item For the first term, it follows immediately from the definition of $\clE$ that
    \begin{equation}
        \bbE \left[ \left| \|\bbrN_{\alpha}\|_{\op} - \xi_{n} \right|^{\mmt} : \clE \right]
        \leq n^{-\mmt/2}
        .
    \end{equation}
    \item For the second term, we use the trivial bound $\Pr(\clE) \leq 1$.
    \item For the third term, it holds surely that
    \begin{equation*}
        \|\bbrN_{\alpha}\|_{\op} \le \frac{1}{\sqrt{n}} \|\bmD_{q}^{-\alpha}\|_{\op} \|\bmX\|_{\infty} \|\bmD_{q}^{-\alpha}\|_{\op} \lesssim n^{1/2}
        ,
    \end{equation*}
    where we used submultiplicativity of the operator norm, that $\|\bmX\|_{\op} \le \|\bmX\|_{\infty}$ since $\bmX$ is symmetric, that $\bmX$ is a matrix of centered Bernoulli random variables, and that the support of $\bmq$ is bounded away from zero.
    Recall from \cite[Corollary 1.3]{ajanki2016ufg} that $|\xi_n| \leq 2$,
    so we have
    \begin{equation*}
        \left| \|\bbrN_{\alpha}\|_{\op} - \xi_{n} \right|^{\mmt}
        \leq
        2^{\mmt-1}
        \left( \|\bbrN_{\alpha}\|_{\op}^{\mmt} + |\xi_{n}|^{\mmt} \right)
        \lesssim
        2^{\mmt-1}
        (n^{\mmt/2} + 2^{\mmt})
        \lesssim
        n^{\mmt/2}
        ,
    \end{equation*}
    and so
    \begin{equation}
        \bbE \left[ \left| \|\bbrN_{\alpha}\|_{\op} - \xi_{n} \right|^{\mmt} : \clE^{C} \right]
        \lesssim n^{\mmt/2}
        .
    \end{equation}
    \item For the fourth term,
    we use the basic inequality that for all real-valued triples $(a, b, c)$,
    \begin{equation*}
        \left| \max(a,b) - c \right|
        \leq \max(|a-c|,|b-c|)
        .
    \end{equation*}
    In particular, we conclude that
    \begin{align*}
        \left| \|\bbrN_{\alpha}\|_{\op} - \xi_{n} \right|
        &
        =
        \left|
            \max\left\{
                |\lambda_1(\bbrN_{\alpha})|,
                |\lambda_n(\bbrN_{\alpha})|
            \right\}
            -
            \xi_{n}
        \right|
        \\&
        \leq
        \max\left\{
            \left|
                |\lambda_1(\bbrN_{\alpha})|
                -
                \xi_{n}
            \right|
            ,
            \left|
                |\lambda_n(\bbrN_{\alpha})|
                -
                \xi_{n}
            \right|
        \right\}
        \\&
        \leq
        \max\left\{
            \left|
                \lambda_1(\bbrN_{\alpha})
                -
                (-\xi_{n})
            \right|
            ,
            \left|
                \lambda_n(\bbrN_{\alpha})
                -
                (+\xi_{n})
            \right|
        \right\}
        .
    \end{align*}
    Thus,
    \begin{align*}
        \Pr(\clE^{C})
        &
        =
        \Pr\left( \left| \|\bbrN_{\alpha}\|_{\op} - \xi_{n} \right| > n^{-1/2} \right)
        \\&
        \leq
        \Pr\left(
            \max\left\{
                \left|
                    \lambda_1(\bbrN_{\alpha})
                    -
                    (-\xi_{n})
                \right|
                ,
                \left|
                    \lambda_n(\bbrN_{\alpha})
                    -
                    (+\xi_{n})
                \right|
            \right\}
            > n^{-1/2}
        \right)
        \\&
        =
        \Pr\left(
            \left|
                \lambda_1(\bbrN_{\alpha})
                -
                (-\xi_{n})
            \right|
            > n^{-1/2}
            \text{ or }
            \left|
                \lambda_n(\bbrN_{\alpha})
                -
                (+\xi_{n})
            \right|
            > n^{-1/2}
        \right)
        \\&
        \leq
        \Pr\left(
            \left|
                \lambda_1(\bbrN_{\alpha})
                -
                (-\xi_{n})
            \right|
            > n^{-1/2}
        \right)
        +
        \Pr\left(
            \left|
                \lambda_n(\bbrN_{\alpha})
                -
                (+\xi_{n})
            \right|
            > n^{-1/2}
        \right)
        ,
    \end{align*}
    and applying
    \cref{eq:local:law:noise:lower,eq:local:law:noise:upper}
    with $\delta = 1/6$ and $D = \mmt$
    yields that
    there exists some $N_0$
    so that for all $n \geq N_0$
    we have
    \begin{align}
        \Pr(\clE^{C})
        &
        \leq
        \Pr\left(
            \left|
                \lambda_1(\bbrN_{\alpha})
                -
                (-\xi_{n})
            \right|
            > n^{-1/2}
        \right)
        +
        \Pr\left(
            \left|
                \lambda_n(\bbrN_{\alpha})
                -
                (+\xi_{n})
            \right|
            > n^{-1/2}
        \right)
        \\& \nonumber
        =
        \Pr\left(
            \left|
                \lambda_1(\bbrN_{\alpha})
                -
                (-\xi_{n})
            \right|
            > n^{\delta}n^{-2/3}
        \right)
        +
        \Pr\left(
            \left|
                \lambda_n(\bbrN_{\alpha})
                -
                (+\xi_{n})
            \right|
            > n^{\delta}n^{-2/3}
        \right)
        \\& \nonumber
        \leq
        2 n^{-\mmt}
        .
    \end{align}
\end{enumerate}
Thus, we have that for $n$ sufficiently large
\begin{equation}
    \bbE \left| \|\bbrN_{\alpha}\|_{\op} - \xi_{n} \right|^{\mmt}
    \lesssim
    n^{-\mmt/2}
    .
\end{equation}
Repeating the same argument for $\bbE \left| \|\bmR \circ \bbrN_{\alpha}\|_{\op} - \xi_{n} \right|$
and using \cref{eq:local:law:noiseflip:lower,eq:local:law:noiseflip:upper}
yields
\begin{equation}
    \bbE \left| \|\bmR \circ \bbrN_{\alpha}\|_{\op} - \xi_{n} \right|
    \lesssim
    n^{-\mmt/2}
    .
\end{equation}
Thus,
\begin{equation*}
    \bbE \left| \|\bmR \circ \bbrN_{\alpha}\|_{\op} - \|\bbrN_{\alpha}\|_{\op} \right|^{\mmt}
    \lesssim
    n^{-\mmt/2}
    ,
\end{equation*}
and hence
\begin{equation}
    \ellP{\opnorm{\bmR \circ \bbrN_{\alpha}} - \opnorm{\bbrN_{\alpha}}}
    \lesssim
    n^{-1/2}
    ,
\end{equation}
which completes the proof of \cref{thm:DCSBM:noise:signflip:preserve}.
Note that a faster rate
can immediately be obtained by choosing $\delta$
in \cref{eq:local:law:noise:lower,eq:local:law:noise:upper,eq:local:law:noiseflip:lower,eq:local:law:noiseflip:upper}
closer to zero.
That said, choosing $\delta = 1/6$ is sufficient for our purposes.
\qed

\section{Proof of \cref{thm:DCSBM:noise:recovery}}
\label{proof:DCSBM:noise:recovery}
First, adding and subtracting both $\bmR \circ \btlL_{\alpha}$ and $\bmR \circ \btlN_{\alpha}$ and applying the triangle inequality yields
\begin{align*}
    &
    \ellP{ \opnorm{\bmR \circ \bmL_{\alpha}} - \opnorm{\btlN_{\alpha}} }\\
    &\qquad
    \leq
    \ellP{ \opnorm{\bmR \circ \bmL_{\alpha}} - \opnorm{\bmR \circ \btlL_{\alpha}} }
    + \ellP{ \opnorm{\bmR \circ \btlL_{\alpha}} - \opnorm{\bmR \circ \btlN_{\alpha}} }
    + \ellP{ \opnorm{\bmR \circ \btlN_{\alpha}} - \opnorm{\btlN_{\alpha}} }
    .
\end{align*}
It remains to bound each term.
\begin{enumerate}
    \item A bound for the first term follows from \cref{thm:DCSBM:approx:L:signflip:decay} by noting that
    \begin{align*}
        &\ellP{ \opnorm{\bmR \circ \bmL_{\alpha}\|_{\op} - \|\bmR \circ \btlL_{\alpha}} } \\
        &\qquad
        \leq
        \ellP{ \opnorm{\bmR \circ \bmL_{\alpha} - \bmR \circ \btlL_{\alpha}} }
        =
        \ellP{ \opnorm{\bmR \circ (\bmL_{\alpha} - \btlL_{\alpha})} }
        \lesssim
        \left(\frac{\log n}{n}\right)^{1/2}
        .
    \end{align*}

    \item A bound for the second term follows from \cref{thm:DCSBM:signal:signflip:destroy} by noting that
    \begin{align*}
        &\ellP{ \opnorm{\bmR \circ \btlL_{\alpha}} - \opnorm{\bmR \circ \btlN_{\alpha}} } \\
        &\qquad
        \leq
        \ellP{ \opnorm{\bmR \circ \btlL_{\alpha} - \bmR \circ \btlN_{\alpha}} }
        = \ellP{ \opnorm{\bmR \circ \btlS_{\alpha}} }
        \lesssim
        \left\| \opnorm{\bmR \circ \btlS_{\alpha}} \right\|_{\psi_1}
        \lesssim
        \left(\frac{\log n}{n}\right)^{1/2}
        .
    \end{align*}

    \item A bound for the final term is provided by \cref{thm:DCSBM:noise:signflip:preserve}, i.e.,
    \begin{equation*}
        \ellP{\opnorm{\bmR \circ \btlN_{\alpha}} - \opnorm{\btlN_{\alpha}}}
        \lesssim
        n^{-1/2}
        .
    \end{equation*}
\end{enumerate}
This completes the proof of \cref{thm:DCSBM:noise:recovery}. \qed

\section{Numerical illustration of the impact of \texorpdfstring{$\quantile$}{q} and \texorpdfstring{$T$}{T}}
\label{sec:qT:experiment}

This appendix illustrates the influence of $\quantile$ and $T$ on the performance of NetFlipPA through two numerical examples: a smaller network with $n=1000$ nodes and a larger network with $n=2000$ nodes.
Each network is generated as a degree-corrected stochastic blockmodel (DCSBM) with $K=3$ communities whose adjacency matrix $\bmA \in \{0,1\}^{n \times n}$ is generated as
\begin{equation*}
    A_{ij}
    = A_{ji}
    \overset{\textnormal{ind}}{\sim}
    \operatorname{Bernoulli}(q_{i} q_{j} C_{g_{i} g_{j}}),
    \qquad
    1 \leq i \leq j \leq n
    ,
\end{equation*}
where
$C_{g_{i} g_{j}} = 1 + \frac{M_{g_{i} g_{j}}}{\sqrt{n}}$ with $M_{g_{i} g_{j}} = -2$ for $g_{i} \neq g_{j}$ and $M_{g_{i} g_{i}} = 5$.
The node-specific degree parameters $\bmq \in (0,1)^{n}$ are generated i.i.d., where $q_{i} = 0.4$ with probability $1/2$ and $q_{i} = 0.9$ otherwise for $1 \le i \le n$.
The vector of node community memberships is set to be $\bmg = [1 \cdot \bm1_{300}, 2 \cdot \bm1_{300}, 3 \cdot \bm1_{400}]$ for the smaller network and $\bmg = [1 \cdot \bm1_{600}, 2 \cdot \bm1_{600}, 3 \cdot \bm1_{800}]$ for the larger network.
In both cases, the correct selection is two positive eigenvalues and zero negative eigenvalues.

\begin{figure*}[!htbp]
    \strut
    \hfill
    \subfloat[Smaller network with $n=1000$ nodes\label{fig:qTstudy:1000}]{%
       \includegraphics[scale=0.6]{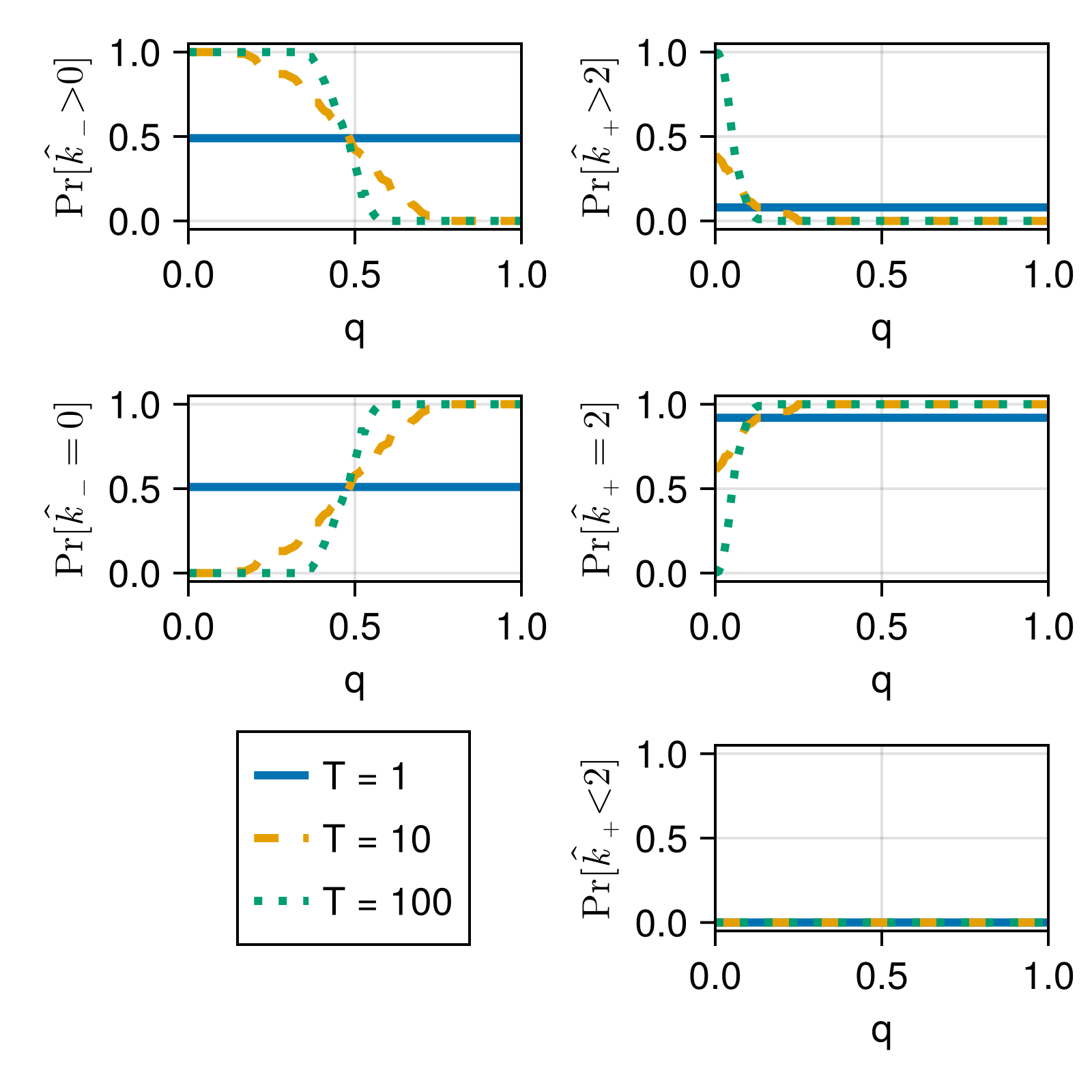}}
    \hfill
    \subfloat[Larger network with $n=2000$ nodes\label{fig:qTstudy:2000}]{%
       \includegraphics[scale=0.6]{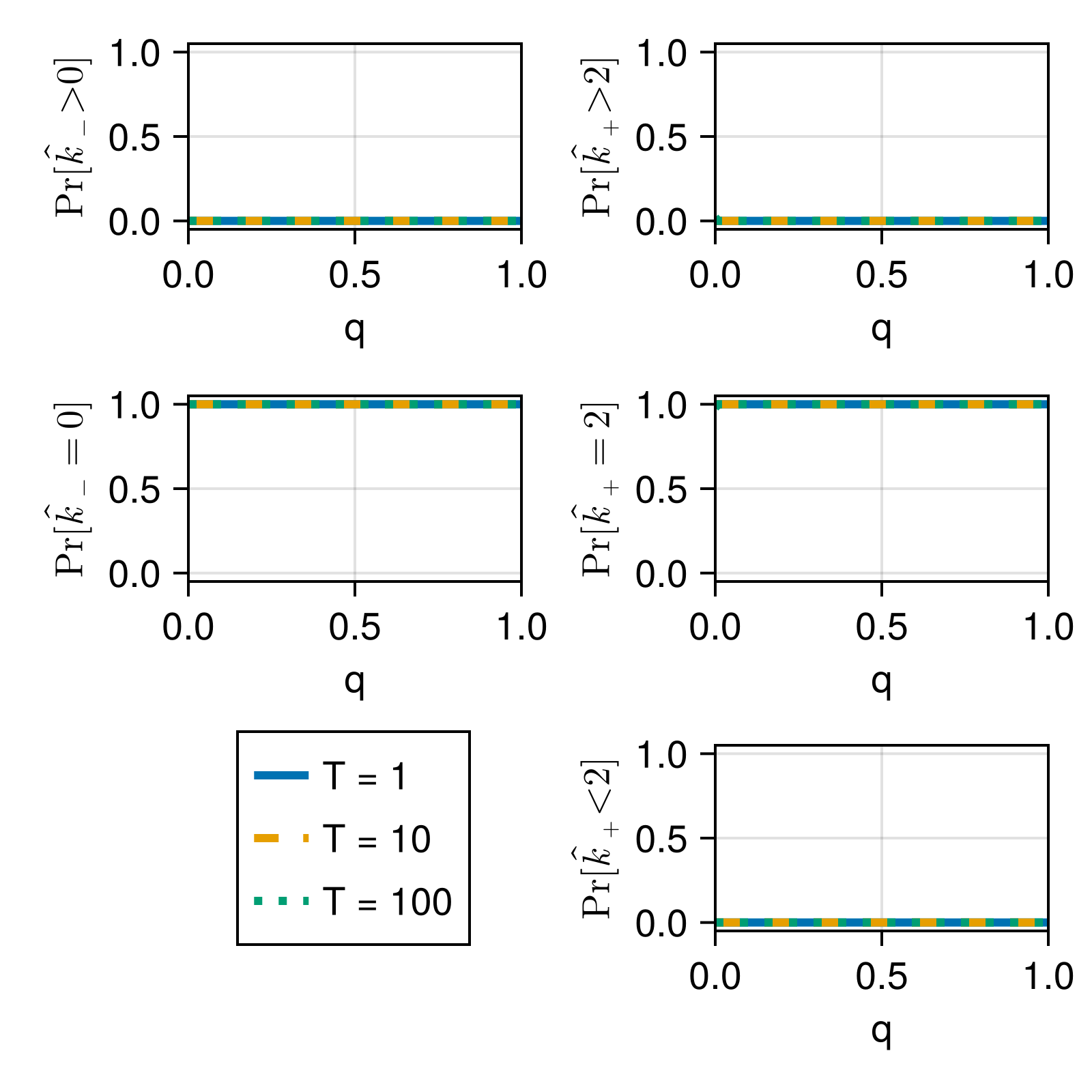}}
    \hfill
    \strut
    \caption{Performance of NetFlipPA as a function of the quantile $\quantile$ for different choices of the number of trials $T$. Each plot shows the proportion of NetFlipPA runs that produced estimates above, equal to, or below the correct selections of two positive eigenvalues and zero negative eigenvalues.}
    \label{fig:qTstudy}
\end{figure*}

For each network, \cref{fig:qTstudy} plots the performance of NetFlipPA as a function of the quantile $\quantile$ for various choices of the number of trials $T$.
Given that NetFlipPA is a randomized algorithm, we run it $100$ times on each network and plot the proportion of times the NetFlipPA selections $\hat{k}_{+}$ and $\hat{k}_{-}$ were above, equal to, and below the correct selections for the positive and negative eigenvalues, respectively.
We omit the plot for under-selection of the negative eigenvalues since the selection can never be below zero.

For the smaller network (\cref{fig:qTstudy:1000}), we note first that NetFlipPA with $T=1$ incorrectly over-selected for a constant proportion of the runs; its performance does not depend on $\quantile$ since the quantiles are all the same in that case.
For larger choices of $T$, the quantiles begin to spread out, leading to greater over-selection for small choices of $\quantile$ and vice versa for large choices of $\quantile$.
Notably, it is possible in this case to obtain perfect selection in all runs by selecting $T=10$ and $\quantile$ near one.
For the larger network (\cref{fig:qTstudy:2000}), NetFlipPA achieved correct selection in all runs, regardless of the choice of $\quantile$ and $T$.
This is illustrative of the general phenomenon that $\|\bmR \circ \btlL_{\alpha}\|_{\op}$ concentrates as the network size grows, in turn making the particular choice of quantile $\quantile$ or number of trials $T$ less influential.

Overall, in the absence of other external considerations, we recommend choosing a number of trials $T$ as large as can be computationally accommodated and choosing $\quantile$ between $0.95$ and $1$.
Large choices of $T$ help make the algorithm behave more deterministically from run to run, which is generally more necessary for smaller networks.
Meanwhile, the choice of $\quantile$ provides users the ability to incorporate domain knowledge (e.g., about what structures may be present in the network) to tune the trade-off between over- and under-selection.

\end{document}